\documentclass[11pt,reqno]{amsart}

\usepackage{amsrefs}
\usepackage{amsmath}
\usepackage{amsthm}
\usepackage{graphicx}
\usepackage{subcaption}
\usepackage{tikz}
\usetikzlibrary{trees,snakes,external,backgrounds,patterns}
\usepackage{amsfonts}
\usepackage{amssymb}
\usepackage{color}
\usepackage{bm}
\usepackage{bbm}
\usepackage{hyperref}
\usepackage{enumitem}
\usepackage[margin=1.1in]{geometry}
\usepackage{relsize}
\def\root{\oSlash}

\numberwithin{equation}{section}

\newtheorem{theorem}{Theorem}[section]
\newtheorem{lemma}[theorem]{Lemma}
\newtheorem{proposition}[theorem]{Proposition}

\newtheorem{corollary}[theorem]{Corollary}

\theoremstyle{definition}
\newtheorem{example}[theorem]{Example}
\newtheorem{definition}[theorem]{Definition}
\newtheorem{remark}[theorem]{Remark}
\newtheorem{Condition}[theorem]{Condition}

\newtheorem{innercustomthm}{Assumption}
\newenvironment{assumption}[1]
  {\renewcommand\theinnercustomthm{#1}\innercustomthm}
  {\endinnercustomthm}

\def\ind{{\mathbbm 1}}
\def\E{{\mathbb E}}

\def\R{{\mathbb R}}

\def\N{{\mathbb N}}
\def\PP{{\mathbb P}}

\def\FF{{\mathbb F}}

\def\P{{\mathcal P}}

\def\X{{\mathcal X}}

\def\Y{{\mathcal Y}}

\def\L{{\mathcal L}}

\def\G{{\mathcal G}}

\def\Z{{\mathbb Z}}

\def\F{{\mathcal F}}
\def\C{{\mathcal C}}

\def\Cov{{\mathrm{Cov}}}

\def\SQ{S^{\sqcup}}

\newcommand{\Erdos}{Erd\H{o}s-R\'enyi}
\newcommand{\CM}{\textnormal{CM}}

\newcommand{\lan}{\langle}
\newcommand{\ran}{\rangle}
\newcommand{\lfl}{\lfloor}
\newcommand{\rfl}{\rfloor}
\newcommand{\ti}{\tilde}

\newcommand{\Emb}{{\mathbb{E}}}

\newcommand{\Nmb}{{\mathbb{N}}}

\newcommand{\Pmb}{{\mathbb{P}}}

\newcommand{\Rmb}{{\mathbb{R}}}

\newcommand{\Tmb}{{\mathbb{T}}}

\newcommand{\Zmb}{{\mathbb{Z}}}

\newcommand{\Gmc}{{\mathcal{G}}}

\newcommand{\Lmc}{{\mathcal{L}}}

\newcommand{\Pmc}{{\mathcal{P}}}

\newcommand{\Tmc}{{\mathcal{T}}}

\newcommand{\rhohat}{{\hat{\rho}}}

\newcommand{\tily}{\tilde{y}}
\newcommand{\tilz}{\tilde{z}}
\newcommand{\tile}{\tilde{\eta}}
\newcommand{\z}{z}

\newcommand{\Etil}{{\tilde{E}}}

\newcommand{\Gtil}{{\tilde{G}}}

\newcommand{\Htil}{{\tilde{H}}}

\newcommand{\rhotil}{{\tilde{\rho}}}

\newcommand{\Vtil}{{\tilde{V}}}

\newcommand{\tree}{{\mathcal T}}

\newcommand{\parm}{\theta}
\newcommand{\cc}{c}
\newcommand{\bK}{\bar{K}}

\def\decayfn{c}
\def\newdecayfn{\bar{c}}

\def\nspace{\Xi}

\def\newspace{\base}
\def\newpathspace{\pathspace}
\def\newpathspacek{\pathspacek}
\def\base{\Polish}
\def\pathspace{\base_\infty}
\def\pathspacek{\base_k}

\def\path{E_\infty}

\def\comp{\mathsf{C}}
\def\compu{\mathsf{C}_{\mathrm{Unif}}}

\def\Polish{\mathcal{Y}}	
\def\polish{y}				
\def\RV{Y} 					
\def\newPolish{\mathcal{Z}}	
\def\newpolish{z}				
\def\newRV{Z}					

\def\compmax{{\mathsf{C}}_{\mathrm{max}}}

\def\indexset{\mathbb{I}}

\newcommand{\oSlash}{{\mbox{\o}}}

\title[Local weak convergence for networks of interacting processes]{Local weak convergence for sparse networks of interacting processes}

	\date{\today}
	 	\subjclass[2000]{Primary: 60K35, 60J05, 60J60, 60J80; Secondary: 60F17, 60B10, 82C22}
	 	\keywords{interacting diffusions, probabilistic cellular automata, discrete-time Markov chains, sparse graphs, random graphs, local weak convergence, mean-field limits, nonlinear Markov processes, \Erdos \ graphs, configuration model, unimodularity, Gibbs measures, Markov random fields} 
	 	\author[Lacker]{Daniel Lacker}
                \address{Columbia University, New York, New York} 
	 	\author[Ramanan]{Kavita Ramanan}
                \thanks{D. Lacker was partially supported by the Air Force Office of Scientific Research (AFOSR) Grant FA9550-19-1-0291. K. Ramanan was partially supported by the Army Research Office (ARO) Grant W911NF2010133 and a Simons Fellowship.} 
	 	\address{Division of Applied Mathematics, Brown University, 182 George Street, Providence, RI 02912} 
	 	\author[Wu]{Ruoyu Wu}
                 \address{Department of Mathematics, Iowa State University, 411 Morrill Road, Ames, IA 50011} 
	 	 \email{daniel.lacker@columbia.edu, kavita\_ramanan@brown.edu, ruoyu@iastate.edu}

\date{\today}
\begin{document}
\begin{abstract}
We study the limiting behavior of interacting particle systems indexed by large sparse graphs, which evolve either according to a discrete time Markov chain or a diffusion, in which particles interact directly only with their nearest neighbors in the graph. To encode sparsity we work in the framework of local weak convergence of marked (random) graphs. We show that the joint law of the particle system varies continuously with respect to local weak convergence of the underlying graph marked with the initial conditions. In addition, we show that the global empirical measure converges to a non-random limit for a large class of graph sequences including sparse \Erdos \ graphs and configuration models, whereas the empirical measure of the connected component of a uniformly random vertex converges to a random limit. Along the way, we develop  some related results on the time-propagation of ergodicity and empirical field convergence, as well as some general results on local weak convergence of Gibbs measures in the uniqueness regime which appear to be new. The results obtained here are also useful for obtaining autonomous descriptions of marginal dynamics of interacting diffusions and Markov chains on sparse graphs. While limits of interacting particle systems on dense graphs have been extensively studied, there are relatively few works that have studied the sparse regime in generality.    
\end{abstract}

\maketitle

\tableofcontents

\section{Introduction}

\subsection{Problem Description} 
We study  limits of large  systems of interacting particles 
whose dynamics are governed by a (possibly random) underlying 
 interaction graph, in the limit as the number of particles goes to infinity, while the
(expected) degree of the interaction graph remains finite. 
 We focus in parallel on the case of discrete-time processes and continuous-time
 diffusive processes. 

In discrete time, given a finite simple (possibly random) graph $G=(V,E)$, a Polish space $\X$,
  an initial configuration $x=(x_v)_{v \in V} \in \X^V$, and
  i.i.d.\ noises $\{\xi_v(k) : v \in V, \ k \in \N\}$, taking values in some other Polish space $\Xi$,
    we consider processes evolving according to 
\begin{align} \label{intro:discmodel}
X_v^{G,x}(k+1) = F\left(X_v^{G,x}(k),\mu_v^{G,x}(k),\xi_v(k+1)\right), \quad X^{G,x}_v(0) = x_v, \ \ v \in V.
\end{align}
Here $F$ is a given (suitably regular) function, and for any vertex $v$ that is not isolated, 
$\mu^{G,x}_v(k)$ is the local (random) empirical measure of the states of the neighbors of $v$ at time $k$, defined by
\[
\mu_v^{G,x}(k) = \frac{1}{|N_v(G)|} \sum_{u \in N_v(G)} \delta_{X^{G,x}_u(k)},
\] 
where $N_v(G) = \{u \in V: (u,v) \in E\}$  denotes the neighborhood of the vertex $v$. 
For diffusive dynamics, we replace the noises with  
independent $d$-dimensional Brownian motions $(W_v)_{v \in V}$. Given initial
conditions $x = (x_v)_{v \in V} \in (\R^d)^V$, the dynamics are characterized by the equation  
\begin{align}  \label{intro:contmodel}
dX^{G,x}_v (t) =  b(X^{G,x}_v (t),  \mu^{G,x}_v (t)) dt  +  \sigma (X^{G,x}_v(t),  \mu^{G,x}_v (t)) dW_v (t), \quad X^{G,x}_v(0) = x_v, \ \ v \in V,
\end{align}
where $b$ and $\sigma$ are suitably regular drift and diffusion coefficients.
We in fact study  more general classes of (possibly non-Markovian) particle systems,
which are fully specified  in Sections \ref{se:disc-setup} and \ref{se:cont-setup}.

Discrete-time particle systems of the form \eqref{intro:discmodel} fit into the class of \emph{probabilistic or stochastic cellular automata}, used in a variety of fields such as statistical physics  \cite{lebowitz1990statistical}, ecology \cite{DurrettLevin}, epidemiology \cite{grassberger1983critical}, and economics \cite{follmer1994stock} to name but a few;  see also the recent text \cite{LouisNardi} for a more comprehensive account.
Large systems of interacting diffusions of the form \eqref{intro:contmodel} arise as models in a range of applications including statistical 
physics \cite{Der03,RedRoeRus10}, neuroscience
  \cite{BalFasFagTou12,LucStan14,Med18}, 
and systemic risk  \cite{nadtochiy2018mean,spiliopoulos2018network}. 
    These systems are often too complex to be tractable, either analytically or numerically, 
   and it is natural to try to understand the behavior of the particle system in an asymptotic regime,  for suitable sequences of graphs $\{G_n\}_{n \in \N}$ with growing vertex set.
   We will be chiefly interested in the behavior of a ``typical" particle (represented by the root vertex, which we do not label explicitly in this introduction) and the (global) empirical measure process, 
\begin{align}	\label{intro:empiricalmeasure}
\mu^{G,x} (t) = \frac{1}{|V|} \sum_{v \in V} \delta_{X^{G,x}_v(t)}.
\end{align}

Suppose first that $G_n$ is the complete graph on $n$ vertices, and $x^n=(x^n_v)_{v \in G_n}$ are \emph{chaotic} in the sense that $\frac{1}{n}\sum_{v \in G_n} \delta_{x^n_v}$ converges weakly to a probability measure $\mu(0)$.
Then, under suitable assumptions on the coefficients, the limiting behavior of \eqref{intro:discmodel} and \eqref{intro:contmodel} as $n \rightarrow \infty$, known as a \emph{mean field limit},  
has been well studied. 
Specifically, the limiting dynamics of a representative randomly chosen vertex  in $G_n$ is described by a \emph{nonlinear Markov process}, governed in discrete time by the dynamics
\begin{align}
X(k+1) = F(X(k),\mu(k),\xi(k+1)), \quad \mu(k)={\rm Law}(X(k)), \label{intro:nonlinearMC}
\end{align}
or for diffusions by the dynamics
\begin{align} \label{intro:McKeanVlasov}
d X(t)  =  b( X(t), \mu(t)) dt + \sigma (X(t),\mu(t)) dW(t), \quad \mu(t) = {\rm Law} (X(t)), 
\end{align}
often referred to as the \emph{McKean-Vlasov equation}.
In fact, it is known that, under suitable assumptions on the initial conditions,  $\mu(\cdot)$ is also the (deterministic) limit, as $n \rightarrow \infty$, of the empirical measure process $\mu^{G_n,x^n}(\cdot)$.
For a derivation of these limits,  see \cite{del2004feynman} for discrete time models  and \cite{Mck67,sznitman1991topics,Kol10} and references therein for diffusive dynamics. 
The measure-valued function  $\mu(\cdot)$ can also be characterized as the unique solution to a nonlinear Kolmogorov equation (a recursive difference equation in discrete time or a partial differential equation in the diffusive setting), whence the name {\em nonlinear} Markov process.    
Such a  characterization is possible because the particles interact only \emph{weakly},
with the influence of any single 
particle on any other particle being of order $1/n$. 
This leads to asymptotic independence of any finite collection of particles 
and the convergence of the random (global)  empirical measure $\mu^{G_n,x^n}$ of the particle systems
to a deterministic limit (see \cite{sznitman1991topics,Mck67} for further discussion
of this phenomenon, known as propagation of chaos).

With the above intuition in mind, 
it is natural to expect that the
limiting dynamics of a typical particle 
could be described by exactly the same nonlinear Markov process
even for graph sequences $\{G_n\}_{n \in \N}$ in which each graph is not necessarily complete,  as long as they are sufficiently dense. 
Indeed, the interactions remain \emph{weak} 
in this setting, and thus one would expect the asymptotic independence property to persist. 
Recent works by several authors have rigorously
established this in various settings \cite{DelGiaLuc16,BhaBudWu18,CopDieGia18,OliRei18,Luc18quenched,bayraktar2019mean}, 
although the arguments are more
involved than in the complete graph case.

In this article, we   complement the above body of work by studying the convergence of $X^{G_n,x^n}$ for  a large class  of (possibly random) sparse 
graph sequences $\{G_n\}_{n \in \N}$ and initial configurations $x^n=(x^n_v)_{v \in G_n}$. 
In contrast to the setting of dense graph sequences, there are relatively few works that have studied convergence results in the sparse regime.
In the sparse setting, neighboring particles interact \emph{strongly} and do not become asymptotically independent, and the limiting dynamics of any finite set of 
particles is no longer described by the mean-field limit.
In particular, the graph structure plays an important role, and        
a completely different approach is required.

This paper, along with \cite{LacRamWu19b,LacRamWu-MRF}, supersedes an earlier arXiv preprint \cite{LacRamWu-original}, reorganizing and expanding upon several aspects of the material. Notably, the present paper significantly extends the results on local convergence from \cite{LacRamWu-original}.
The present paper and \cite{LacRamWu19b,LacRamWu19a,LacRamWu-MRF} treat complementary aspects of the same class of particle systems but may be read independently.

\subsection{Discussion of Results} \label{se:discussion}
The framework we adopt for the convergence analysis is that of \emph{local weak convergence}, a natural
mode of convergence for sparse graphs, which is reviewed in detail in Section \ref{se:background-localmarkedgraphs}.
Essentially, a sequence of (rooted, locally finite) graphs $\{G_n\}_{n \in \N}$ converges locally to a limiting graph $G$ if for each $r > 0$ the neighborhood of radius $r$ around the root $G_n$ is isomorphic to that of $G$ for  large enough $n$; see Section \ref{subsub-unmarked} for a precise definition.
A simple example that demonstrates the \emph{local} nature of this topology is the n-vertex cycle, which converges locally to the infinite line graph $\Z$ (i.e., the 2-regular tree).
Notably, as summarized in Section \ref{se:example graphs}, local limits are well known for
many common sparse random graph models: \Erdos \ graphs converge to Galton-Watson trees with Poisson offspring distribution, whereas random regular graphs converge to (non-random) infinite regular trees, and configuration models converge to so-called unimodular Galton-Watson trees. 
The remarkable feature of these random graph models is that they are \emph{locally tree-like}, which means the local limit is a tree; although these random graphs have  many cycles,  short cycles are very unlikely, and long cycles are irrelevant to local convergence.

The key notion that will be used in the dynamical setting considered here is   
  a similar local convergence notion that can be defined for \emph{marked graphs} $(G,\polish)$, in which elements
$\polish=(\polish_v)_{v \in G}$ of some fixed metric space $\Polish$ are attached to the vertices of the graph.
In our setting, the marks will represent either
  initial conditions $x = (x_v)_{v \in V}$ or the (random) trajectories $(X^{G,x}_v)_{v \in V}$ of the interacting process.
We now briefly summarize our main results.

 \subsubsection{Local convergence in law} 
 \label{subsub-inlaw} 
In both the discrete-time and diffusive settings of \eqref{intro:discmodel} and \eqref{intro:contmodel}, under suitable assumptions on the coefficients, we show in Theorems \ref{thm:local-law-disc} and \ref{thm:local-law-cont}, respectively, that if the sequence of (random) marked graphs $\{(G_n,x^n)\}$ converges to $(G,x)$ in law, with respect to  the topology of local convergence, then $\{(G_n,X^{G_n,x^n})\}$ converges in distribution to $(G,X^{G,x})$. Here $(G,X^{G,x})$ is the (random) marked graph formed by marking each vertex with the (random) trajectory of the particle indexed by vertex $v$.

Our results cover a wide class of initial conditions, including independent and identically distributed (i.i.d.)  marks on any locally convergent graph sequence, as well as a class of Gibbs measures (see  Section \ref{subsub-icconv}). 
In the latter case, Section \ref{subsub-icconv} and 
Appendix \ref{ap:Gibbsmeasures} develop some natural but apparently new results on local weak convergence of Gibbs measures, which may well be folklore. Essentially, we show that the Gibbs or Markov-random field property
of the marked graph is preserved  under  
local convergence of the underlying graph, as long as the (infinite-volume) Gibbs measure on the limiting graph is unique. We do not address the intriguing case where uniqueness fails for the Gibbs measure, which of course requires a finer analysis more tailored to specific models (see, e.g., \cite{montanari2012weak}).  
This complements the substantial recent literature studying Gibbs measures on sparse graphs, surveyed in \cite{dembo-montanari}, which has successfully analyzed the limiting behavior of various other quantities and processes derived from these Gibbs measures, such as the free energy/entropy density and belief propagation algorithms.

\subsubsection{Global Empirical Measure Convergence}
\label{subsub-inprob}

Our second set of results focuses on convergence of the empirical measures  $\{\mu^{G_n,x^n}\}_{n \in \N}$. 
Suppose $(G_n,x^n)$ converges locally in law to a limit marked graph $(G,x)$.   When the limit graph $G$  is infinite with positive probability,  
empirical measure convergence does not follow immediately from the resulting \emph{local} convergence results of $(G_n, X^{G_n,x^n})$ described 
in Section \ref{subsub-inlaw} because  the empirical measure is a \emph{global} quantity. 
In general, the asymptotic behavior of the  empirical measure sequence in this sparse graph setting is  more subtle than in the mean-field 
 or dense graph cases because pairs of particles (for example, neighboring particles) need no longer be asymptotically 
   independent, and so the limiting empirical measure can be random. 
   Indeed,  if the event that the limit graph $G$ is finite has positive probability,  it is straightforward to argue
   that on that event $\mu^{G_n,x^n}$ converges in law to the empirical measure $\mu^{G,x}$;
 see Proposition \ref{pr:localconvergence-empiricalmeasure}. 
 More interesting subtleties arise when dealing with potentially disconnected graphs, such as the sparse Erd\"{o}s-Renyi graph.
The topology of local convergence is defined only on the space of \emph{connected} and rooted graphs, so one must choose a
root and a connected component in order to claim local convergence of the Erd\"{o}s-Renyi graph $G_n = {\mathcal G}(n, p_n)$ with  
  $np_n \rightarrow \theta \in (0,\infty)$.   The root vertex $\root_n$  
 is normally taken to be uniformly distributed with corresponding connected component $\comp_{\root_n}(G_n)$, in which case the associated limit graph is well known to be a 
 Galton-Watson tree with a Poisson offspring distribution.
 This tree is finite with positive probability, and so the corresponding empirical measure sequence 
 converges to a random limit, which is explicitly described in  Theorem \ref{thm:CompEmpMeas}.
 Analogs of this result are shown also for a broad class of configuration models, and in fact 
 for a larger class of graph sequences (see Theorem \ref{thm:GenCompEmpMeas}) that satisfy certain properties
 pertaining to the behavior of the largest connected component (see Condition \ref{cond-graphs}).

But it is arguably more natural, when considering the empirical measure, to include  all vertices in $G_n$, rather than just those in a single component.  In this case, we show (see Theorem \ref{thm:localprob-disc}) that, as in the mean-field setting,
  the empirical measure $\mu^{G_n,x^n}$ converges to a \emph{non-random} limit that, additionally, coincides with the law of the root
  particle in the limit graph, ${\rm Law}(X^{G,x}_{\root})$. 
  However, since as mentioned above,  neighboring vertices in the graphs are not asymptotically independent and remain correlated in the limit,
the  reason for a {\em deterministic} limit is  different from that  in the mean-field setting.  
In the sparse regime this phenomenon can be attributed to a more global averaging effect that is a consequence of the asymptotic independence
of finite neighborhoods of any two independent uniformly chosen roots in the graph.
 Theorems \ref{thm:localprob-disc} and \ref{thm:localprob-cont} (for discrete time and diffusive processes, respectively), 
 elucidate the more general principle that the global averaging effect and a deterministic
limit coinciding with the law of the root particle hold more broadly 
whenever the sequence $(G_n,x^n)$ converges to some limiting rooted marked graph $(G,x)$ in 
the sense of \emph{convergence in probability in the local weak sense}.
This latter notion is stronger than convergence in law in the local weak sense and arguably describes 
 a more global notion of convergence that is defined for sequences of graphs that are not necessarily  connected.
We refer to  Section \ref{se:local weak in prob} for precise definitions of 
these modes of convergence, as well as a justification of the terminology, which is taken from   \cite{van2020randomII}.
 The proofs of Theorems \ref{thm:localprob-disc} and \ref{thm:localprob-cont} entail correlation decay estimates and 
(quenched) concentration estimates for the empirical measure on (random) graphs, which
exploit duality properties of random graphs. 
A final subtlety regarding these different forms of convergence is 
 illustrated by examples considered in Section \ref{se:lattices + trees}, where it is shown that 
under the weaker assumption of local convergence in law (but not in probability), the limiting empirical measure,
even if deterministic, may fail to coincide with the limiting law of a randomly chosen (or ``typical") particle,
thus highlighting yet another  departure from the   mean-field setting.


\subsubsection{Empirical field convergence and propagation of ergodicity}
\label{subs-properg}

The main results discussed thus far may be summarized as follows: If the initial data sequence
$\{(G_n,x^n)\}$ converges in a certain sense, then so does the particle system sequence
$\{(G_n,X^{G_n,x^n})\}$. This is true when the ``certain sense" is either local convergence in law or in probability. In other words, both modes of convergence \emph{propagate} under dynamics of the form \eqref{intro:discmodel} or \eqref{intro:contmodel}.
Our final results pertain to propagation of two other related properties, ergodicity and the convergence of empirical fields, when the underlying graph is fixed.

\subsection{Outlook} \label{se:outlook}

The results here are crucially used in a  related paper \cite{LacRamWu19b}, and the forthcoming
\cite{LacRamWu19a}, where 
we show that if the underlying graph is a regular tree, or more generally a (unimodular) 
Galton-Watson tree, then the dynamics of a single particle and its neighborhood can be autonomously characterized
as solutions to certain
  ``local equations."  This represents a dimension-reduction similar in spirit to the self-contained equations 
\eqref{intro:nonlinearMC} or \eqref{intro:McKeanVlasov} characterizing the dynamics of
a single typical particle at the mean field limit.
Thus, when combined with the results in the present paper, the local equations of \cite{LacRamWu19a} and \cite{LacRamWu19b} yield a concise, finite-dimensional description of the limiting behavior of a typical particle in \eqref{intro:discmodel} or \eqref{intro:contmodel} or of the corresponding empirical distributions
on locally tree-like graphs.
Achieving this limiting characterization, described in detail in \cite[Section 3.7]{LacRamWu19b}, was our original motivation for the present work,
  but we carry out the local convergence analysis here at a more natural level of generality, beyond what is immediately relevant for \cite{LacRamWu19b,LacRamWu19a}. 
  The reader well-versed in local weak convergence may observe the analogy between our results and those obtained in static settings such as \cite{dembo-montanari}, and \cite{aldous-steele}, the latter under the rubric of the objective method. But the convergence analysis in  our dynamic setting is very different and involves novel coupling arguments.   In addition, subtleties 								arise from the different notions of local convergence, and the analysis must be conducted with care. It is especially important for the applications in \cite{LacRamWu19b,LacRamWu19a} to codify precisely
  which forms of graph convergence lead to precisely which forms of convergence of the particle systems, as we will discuss in the following sections.

As we were finalizing the earlier version of this manuscript \cite{LacRamWu-original}, we learned of the independent work of Oliveira, Reis, and Stolerman \cite{OliReiSto19}, which proves  some results on local convergence and convergence of empirical measures of interacting diffusions that have a form  somewhat similar  to \eqref{intro:contmodel}. 
While these are similar in spirit to the results described in Sections \ref{subsub-inlaw} and \ref{subsub-inprob},  our results do not subsume theirs, nor vice versa. For instance, they allow random environments, whereas we allow more general forms of interaction in the coefficients. Methodologically our approaches are quite different, with theirs being more quantitative and thus more restrictive in the nature of permissible graph sequences.
We  compare these results with our own in more detail in Remark \ref{re:OliReiSto-comparison}.
We are unaware of any  results prior to our work on 
 discrete time models, or on the convergence of connected component empirical measures.

 The rest of the paper is organized as follows. In Section \ref{sec-notat} we fix notation and precisely describe the notion of local convergence for marked and unmarked graphs. Section \ref{se:mainresults} gives precise statements of the main results, whose proofs are provided in the  remaining Sections \ref{se:locallimits}--\ref{se:nonrandom}.   Appendix \ref{ap:localconvergence}   summarizes important properties of local weak convergence of
   marked graphs, and
   the remaining Appendices \ref{ap:Gibbsmeasures}--\ref{ap:condgraphs}  contain proofs of various technical results.

\section{Preliminaries and Notation} 
\label{sec-notat}

In this section, we introduce common notation and definitions used throughout the paper, and which are required to state the main results.
We let $\N$ denote the set of natural numbers and $\N_0 := \N \cup \{0\}$.  
For a Polish space $\Polish$, we write $\P(\Polish)$ for the set of Borel probability measures on $\Polish$, endowed always with the topology of weak convergence. Note that $\P(\Polish)$ itself becomes a Polish space with this topology, and we equip it with the corresponding Borel $\sigma$-field. We write $\delta_\polish$ for the Dirac delta measure at a point $\polish \in \Polish$. For a $\Polish$-valued random variable $\RV$, we write $\L(\RV)$ to denote its law, which is an element of $\P(\Polish)$.
Given two $\Polish$-valued random elements $\RV$ and $\RV'$, we write $\RV \stackrel{d}{=} \RV'$ to mean $\L(\RV)= \L(\RV')$.
Write also $C_b(\Polish)$ for the set of bounded continuous real-valued functions on $\Polish$.
Also, given any measure $\nu$ on a measurable space and any  $\nu$-integrable 
function $f$, we use the usual shorthand notation
$\langle \nu, f\rangle := \int f\,d\nu$.  
Denote by $A \Delta B$ the symmetric difference between two sets $A$ and $B$.

\subsection{Graphs}
\label{subs-graphs}

In this paper, unless explicitly stated otherwise, a \emph{graph} $G=(V,E)$ always has a finite or countably infinite vertex set, is simple (no self-edges or multi-edges), and is locally finite (i.e., the degree of each vertex, is finite).
We abuse notation by writing $v \in G$ to mean $v \in V$,  and similarly $|G|=|V|$ denotes the cardinality of the vertex set.  For any graph $G=(V,E)$ and any vertex $v \in V$, we write $N_v(G) := \{u \in V : (u,v) \in E\}$ for the set of neighbors of $v$ in $G$, noting that this set is empty if $v$ is an isolated vertex. As usual, $|A|$ denotes the cardinality of a set $A$.
Let $\mathrm{diam}(A)$ denote the diameter of a set $A \subset V$; precisely,
  for two vertices $u,v \in V$, the distance between $u$ and $v$ is the length of the shortest path from $u$ to $v$,
  and the diameter of $A$ is the maximal distance between any two of its vertices.
For a set $\Polish$ and a graph $G= (V,E)$, we write either $\Polish^V$ or $\Polish^G$
for the configuration space
$\{(\polish_v)_{v \in V}: \polish_v \in \Polish, \ v \in V\}$. We make use of a standard notation for configurations on subsets of vertices: For $\polish=(\polish_v)_{v \in V} \in \Polish^V$ and $A \subset V$, we write $\polish_A$ for the element
$\polish_A=(\polish_v)_{v \in A}$ of $\Polish^A$.

\subsection{Local convergence of marked graphs} \label{se:background-localmarkedgraphs}

This section describes the basic concepts of local convergence for marked and unmarked graphs.
For full details and proofs, see \cite[Section 3.2]{Bordenave2016} and Appendix \ref{ap:localconvergence}. 
The notion of local weak convergence was introduced by Benjamini and Schramm in \cite{benjamini2001recurrence}; other useful references on this topic include \cite{van2020randomII,Bordenave2016,aldous-steele}.

\subsubsection{Unmarked graphs and the space $\G_*$} \label{subsub-unmarked}
A \emph{rooted graph} $G=(V,E,\oSlash)$ is a graph $(V,E)$ (assumed as usual to be locally finite with either finite or countable vertex set) with a distinguished vertex $\oSlash \in V$.
We say two rooted graphs $G_i=(V_i,E_i,\oSlash_i)$  are \emph{isomorphic} if there exists a bijection $\varphi : V_1 \mapsto V_2$ such that $\varphi(\oSlash_1)=\oSlash_2$ and $(\varphi(u),\varphi(v)) \in E_2$ if and only if $(u,v) \in E_1$, for each $u,v \in V_1$.  We 
denote this by $G_1 \cong G_2$. We refer to the map $\varphi$ as an \emph{isomorphism} from $G_1$ to $G_2$, and denote by $I(G_1,G_2)$ the collection of all such isomorphisms from $G_1$ to $G_2$.

Let $\G_*$ denote the set of isomorphism classes of connected rooted graphs.
Given $k \in \N$ and $G=(V,E,\oSlash) \in \G_*$, let $B_k(G)$ denote the induced subgraph (rooted at $\oSlash$) consisting of those vertices whose graph distance from $\oSlash$ is no more than $k$.
We say that a sequence $\{G_n\} \subset \G_*$ \emph{converges locally} to $G \in \G_*$ if, for every $k \in \N$, there exists $n_k \in \N$ such that $B_k(G_n) \cong B_k(G)$ for every $n \ge n_k$.
There is a metric compatible with this notion of convergence that renders $\G_*$ a complete and separable space, such as
\begin{align}
d_*(G,G') &= \sum_{k=1}^\infty 2^{-k} \, \mathbbm{1}_{\{ I(B_k(G),B_k(G')) = \emptyset\}} \label{def:G* metric}
\end{align}
where as usual $\mathbbm{1}_{\{A\}}=1$ if $A$ holds and $\mathbbm{1}_{\{A\}}=0$ otherwise.

\begin{remark}
We will often omit the root from the notation, writing $G \in \G_*$ instead of $(G,\oSlash) \in \G_*$, when there is no need to make explicit reference to the root. But we understand that a graph $G \in \G_*$ always carries with it a root, which by default will be denoted $\oSlash$.
\end{remark}

\subsubsection{Marked graphs and the space $\G_*[\Polish]$}
We also need a notion of local convergence for \emph{marked} graphs, where each vertex of the graph has a mark (or label) associated to it; as mentioned,
these marks will later encode initial conditions or trajectories of  particles.
For a metric space $(\Polish,d)$, a \emph{$\Polish$-marked rooted graph} is a pair $(G,{\polish})$, where $G=(V,E,\oSlash) \in \G_*$, and ${\polish} = (\polish_v)_{v \in V} \in \Polish^V$ is a vector of marks.
For a $\Polish$-marked rooted graph $(G,\polish)$ and $k \in \N$, let $B_k(G,\polish)$ denote the induced $\Polish$-marked rooted subgraph consisting of vertices within the ball of radius $k$ centered at the root.
We say that two $\Polish$-marked rooted graphs $(G,{\polish})$ and $(G',{\polish}')$ are \emph{isomorphic} if there exists an
isomorphism $\varphi$ from $G$ to $G'$ such that $(\polish_v)_{v \in V} = (\polish'_{\varphi(v)})_{v \in V}$. We write $(G,{\polish}) \cong (G',{\polish}')$  to indicate isomorphism.

Let $\G_*[\Polish]$ denote the set of isomorphism classes of $\Polish$-marked rooted graphs.
We say that a sequence $\{(G_n,{\polish}^n)\} \subset \G_*[\Polish]$ \emph{converges locally} to $(G,{\polish}) \in \G_*[\Polish]$ if, for every $k \in \N$ and $\epsilon > 0$, there exists $n_k \in \N$ such that for all $n \ge n_k$ there exists an isomorphism $\varphi : B_k(G_n) \mapsto B_k(G)$ with $\max_{v \in B_k(G_n)}d({\polish}^n_v,{\polish}_{\varphi(v)}) < \epsilon$.
The space $\G_*[\Polish]$ can be equipped with a metric compatible with this notion of convergence, and if $(\Polish,d)$ is complete and separable then so is $\G_*[\Polish]$ (cf.\ \cite[Lemma 3.4]{Bordenave2016}). An equivalent metric which we will use on occasion is
\begin{align}
d_*((G,{\polish}),(G',{\polish}')) &= \sum_{k=1}^\infty 2^{-k}\left(1 \wedge \inf_{\varphi \in I(B_k(G),B_k(G'))}\max_{v \in B_k(G)}d(\polish_v,\polish'_{\varphi(v)})\right). \label{def:G* marked metric}
\end{align}

\subsubsection{Convergence notions in the local weak sense}
\label{se:local weak in prob}

For a (finite or countable, locally finite, possibly disconnected) graph $G=(V,E)$ and a vertex $v \in V$, we write $\comp_v(G)$ for (the isomorphism class of) the connected component of $v$, that is, the set of $u \in V$ for which there exists a path from $v$ to $u$. By viewing $v$ as the root, $\comp_v(G)$ is then an element of $\G_*$.
Note that even if two distinct vertices $u$ and $v$ belong to the same connected component of $G$, the \emph{rooted} graphs $\comp_u(G)$ and $\comp_v(G)$ can be non-isomorphic and thus induce distinct elements of $\G_*$.
When the graph is finite, we may choose a uniformly random vertex $U$ of $G$, and we write $\compu(G) := \comp_U(G)$ for the resulting $\G_*$-valued random variable.
That is, we write $\compu(G)$ for the random connected rooted graph obtained by assigning a root uniformly at random and then isolating the connected component containing this root. 
We define $\comp_v(G,\polish) := (\comp_v(G),\polish_{\comp_v(G)})$ and $\compu(G,\polish)$ similarly for marked graphs.

Fix throughout this section a sequence of finite (possibly disconnected) random graphs $\{G_n\}$.
Let $G$ be a random element of $\G_*$.

\begin{definition}
  \label{def-cip}
We say that $\{G_n\}$ \emph{converges in probability in the local weak sense} to $G$ if
\begin{align}
\lim_{n\to\infty}\frac{1}{|G_n|}\sum_{v \in G_n} f(\comp_v(G_n)) = \E[f(G)], \qquad \text{in probability}, \ \ \forall f \in C_b(\G_*), \label{def:localweakprob}
\end{align}
where we recall that $\comp_v(G_n)$ denotes the connected component of vertex $v$ of $G_n$, rooted at $v$. 
\end{definition}

Note that this definition is meaningful even if the sequence of graphs is non-random, in which case of course the phrase ``in probability" in \eqref{def:localweakprob} is redundant.

\begin{remark} \label{re:lwp-P(G)form}
Because $\G_*$ is a Polish space, a standard argument using a countable convergence-determining set in $C_b(\G_*)$ yields the following equivalent definition:
$\{G_n\}$ converges in probability in the local weak sense to $G$ if and only if
\begin{align*}
\lim_{n\to\infty}\frac{1}{|G_n|}\sum_{v \in G_n} \delta_{\comp_v(G_n)} = \L(G), \qquad \text{in probability in } \P(\G_*).
\end{align*}
\end{remark}

\begin{remark} \label{re:nonrandvertexset}
  Throughout the paper, if we say that a sequence of random graphs $\{G_n\}$ converges
  in probability in the local weak sense, it should be understood that we implicitly
  require that the vertex set of each graph $G_n$ is finite.
\end{remark}

The definition of convergence in probability 
is borrowed from \cite[Definition 2.7]{van2020randomII}, where one also defines \emph{converges in distribution} or \emph{in law} in the local weak sense as  follows:

\begin{definition}
  \label{def-cil}
  We say that $\{G_n\}$ \emph{converges in distribution or law in the local weak sense} to $G$ if 
\begin{align}
\lim_{n\to\infty}\E\left[\frac{1}{|G_n|}\sum_{v \in G_n} f(\comp_v(G_n))\right] = \E[f(G)], \qquad \forall f \in C_b(\G_*), \label{def:localweakdist}
\end{align}
where, recalling that $\compu(G_n)$ denotes the connected component of a uniformly randomly chosen root in $G_n$, we may write the expectation on the left-hand side of \eqref{def:localweakdist} as $\E[f(\compu(G_n))]$.
\end{definition}

Hence, convergence of $\{G_n\}$ to $G$ in distribution in the local weak sense is equivalent to convergence in law of $\{\compu(G_n)\}$ to $G$ in $\G_*$, and of course convergence in probability in the local weak sense is a stronger property.
The averaging over the entire vertex set in Definitions \ref{def-cil} and \ref{def-cip} may lead one to question whether these are really \emph{local} rather than \emph{global} notions of convergence; the key point is that the test functions $f$ must be continuous in the (local) topology on $\G_*$, which essentially deals with bounded (local) neighborhoods around the root vertex. This is why, for instance, the n-cycle converges in the local weak sense (both in law and in probability) to the 2-regular tree $\Z$. 
That said, Lemma \ref{le:localweakprob characterization} below 
(stated more generally for marked graphs) specifies the extent to which convergence in probability is a somewhat more \emph{global} notion: if a graph sequence converges in law in the local weak sense, then it also converges in probability if and only if two uniformly randomly chosen (root) vertices have asymptotically independent neighborhoods.

The above discussion is equally valid for marked graphs. Let $\Polish$ be a Polish space. Let $\polish^n=(\polish^n_v)_{v \in G_n}$ be random $\Polish$-valued marks on the vertices of $G_n$, and let $\polish=(\polish_v)_{v \in G}$ be random $\Polish$-valued marks on $G$.

\begin{definition}
  \label{conv-marked}
  We say that the sequence $\{(G_n,\polish^n)\}$ 
\emph{converges in probability in the local weak sense} to $(G,\polish)$ if
\begin{align}
\lim_{n\to\infty}\frac{1}{|G_n|}\sum_{v \in G_n} f(\comp_v(G_n,\polish^n)) = \E[f(G,\polish)], \qquad \text{in probability}, \ \ \forall f \in C_b(\G_*[\Polish]), \label{def:localweakprob-marked}
\end{align}
\end{definition}

Once again, convergence of $\{(G_n,\polish^n)\}$ to $(G,\polish)$ in probability in the local weak sense implies $\{\compu(G_n,\polish^n)\}$ converges in law to $(G,\polish)$ in $\G_*[\Polish]$. Remark \ref{re:nonrandvertexset} applies also for marked graphs.

Note that the ``root mark map" $\G_*[\Polish] \ni (G,\oSlash,\polish) \mapsto \polish_{\oSlash} \in \Polish$ is continuous. Thus, applying \eqref{def:localweakprob-marked} with $f$ of the
form $f(G,\oSlash,\polish) =  g(\polish_{\oSlash})$ for $g \in C_b(\Polish)$, we deduce that convergence in probability in the local weak sense implies convergence in probability of the empirical mark distributions:

\begin{lemma} \label{le:empmeas general}
  If $\{(G_n,\polish^n)\}$ converges in probability in the local weak sense to $(G,\oSlash,\polish)$, then  
  the empirical measure sequence $\{\frac{1}{|G_n|}\sum_{v \in G_n}\delta_{\polish^n_v}\}$ converges in probability to $\Lmc(\polish_{\oSlash})$ in $\P(\Polish)$.
\end{lemma}

Lastly, we state a useful equivalent characterization of convergence in probability in the local weak sense, valid for marked or unmarked graphs.
The proof is given in  Appendix \ref{apsub:auxiliary}.

\begin{lemma} \label{le:localweakprob characterization}
Suppose $\{G_n\}$ is a sequence of  finite (possibly disconnected) random graphs. Suppose  $\polish^n=(\polish^n_v)_{v \in G_n}$ are (random) marks  with values in a Polish space $\Polish$, for each $n \in \N$. Let $(G,\polish)$ be a random element of $\G_*[\Polish]$.
 Assume $|G_n| \to \infty$ in probability. Let $U^n_1$ and $U^n_2$ denote independent  vertices that are uniformly distributed  on $G_n$,  given $G_n$.  
Then $\{(G_n,\polish^n)\}$ converges in probability in the local weak sense to $(G,x)$ if and only if
\begin{align}
\E[g_1(\comp_{U^n_1}(G_n,\polish^n))g_2(\comp_{U^n_2}(G_n,\polish^n))] \to \E[g_1(G,\polish)]\E[g_2(G,\polish)], \quad \forall  g_1, g_2 \in C_b(\G_*[\Polish]). \label{def:asympindep}
\end{align}
\end{lemma}

\subsubsection{Examples of locally convergent graph sequences}
\label{se:example graphs}

Here we catalog some of the most well known examples of locally converging graphs.

\begin{example} \label{ex:GW}
  Consider the \Erdos \ graph $G_n \sim \G(n,p_n)$, with
  $\lim_{n\rightarrow\infty}np_n= \parm \in (0,\infty)$.
  Then $G_n$ converges in probability in the local weak sense to 
   the Galton-Watson tree with offspring distribution Poisson($\parm$), denoted $\mathrm{GW}(\mathrm{Poisson}(\parm))$.
  Similarly, suppose $G_n \sim \G_{n,m_n}$, which means $G_n$ is selected uniformly at random
  from all (labeled) graphs on $n$ vertices with $m_n$ edges.
  If $\lim_{n\rightarrow\infty}2m_n/n = \parm \in (0,\infty)$, then again 
$G_n$ converges in probability in the local weak sense to the same limit.
   See \cite[Proposition 2.6]{dembo-montanari}, \cite[Theorem 3.12]{Bordenave2016}, or \cite[Theorems 3.11]{van2020randomII}. for proofs of these facts.
\end{example}

\begin{example} \label{ex:unimodularGW}
Given  a graphic sequence $d(n) = (d_1(n),\ldots,d_n(n))$, with each $d_i(n)$ a positive integer less than $n$, 
 let $G_n \sim \CM(n,d(n))$ be a uniformly random graph on $n$ vertices with degree sequence $d(n)$. Alternatively, this may be constructed from 
the configuration model 
conditioned to have no multi-edges or self-edges (see \cite[Chapter 7]{van2016random}).
Suppose the sequence of degree distributions $\{\frac{1}{n}\sum_{i=1}^n\delta_{d_i(n)}\}$ converges to some distribution $\rho \in \P(\N_0)$ with a finite nonzero first moment, and assume also that the first moments converge, $\frac{1}{n}\sum_{i=1}^n d_i(n) \to \sum_{k \in \N_0}k\rho(k)$.  
Then 
$G_n$ converges in probability in the local weak sense
to the \emph{augmented} or \emph{unimodular Galton-Watson tree with degree distribution $\rho$}, denoted $\mathrm{UGW}(\rho)$ and defined as follows: The root has offspring distribution $\rho$, and each subsequent generation has an independent number of offspring according to the distribution $\widehat{\rho}$, where $\widehat{\rho}$ is defined by
\begin{equation}
  \label{def-widehatrho}
\widehat\rho(k) = \frac{(k+1)\rho(k+1)}{\sum_{n \in \N}n\rho(n)}, \ \ k \in \N_0.
\end{equation}
Note that $\widehat{\rho}=\rho$ when $\rho$ is Poisson. See \cite[Proposition 2.5]{dembo-montanari}, \cite[Theorem 3.15]{Bordenave2016}, or \cite[Theorem 4.1]{van2020randomII} for a derivation of this limit.
\end{example} 

\begin{example} \label{ex:regular} 
Let $G_n$ denote the uniform $\kappa$-regular graph on $n$ vertices, for $\kappa \ge 2$.
Then 
$G_n$ converges in probability in the local weak sense
to the infinite $\kappa$-regular tree; this is a well known consequence of the results of \cite{bollobas1980probabilistic}. Note that the infinite $\kappa$-regular tree is nothing but $\mathrm{UGW}(\delta_\kappa)$.
\end{example}


\subsubsection{Examples where graph convergence implies marked graph convergence}
\label{subsub-icconv}

In Section \ref{se:example graphs} 
we provided illustrative examples of many interesting  examples of graphs   $\{G_n\}$ that converge in the local weak sense (both in law and in probability).   
For many of our results,  we will  require that the sequence of  randomly marked
random graphs $\{(G_n,\RV^n)\}$ converge locally (either in law or in probability), where the random marks $\RV^n = (\RV_v^n)_{v \in G_n}$ represent random initial conditions taking values in some Polish space $\Polish$.  
It is thus natural to ask if there are important classes of random initial conditions for which the  local weak convergence of $\{G_n\}$ implies the local weak convergence of the corresponding randomly $\Polish$-marked graphs. 
It is shown in Corollary \ref{co:localweak iid} that this is true when the random initial conditions $\RV = (\RV_v)_{v \in G}$ are i.i.d.
A more general class of initial conditions for which this holds is
 the class of Gibbs measures, defined below. 
Throughout, fix the Polish space $\Polish$, a reference measure $\lambda \in \P(\Polish)$ and a bounded continuous function $\psi : \Polish^2 \to [0,\infty)$ that serves   as a pairwise interaction potential.

\begin{definition} \label{def-Gibbsfin}
   For each finite graph $G=(V,E)$, the $(\psi,\lambda)$-{\it Gibbs measure} on $G$
   is the probability measure $P_G \in \P(\Polish^V)$ defined  by
\begin{align*}
P_G(d(\polish_v)_{v \in V}) = \frac{1}{Z^G}\,\prod_{(u,v) \in E} \psi(\polish_v,\polish_u) \, \prod_{v \in V}\lambda(d\polish_v),
\end{align*}
where $Z^G > 0$ is the normalizing constant.
\end{definition}

This definition does not make sense for infinite graphs $G$ since $Z^G$ is infinite in that case.  Instead,
  as is standard practice, we   use an alternative characterization of  $P_G$  in terms of a certain conditional independence  or Markov random field property, which then admits a natural  extension to 
 locally finite infinite graphs $G=(V,E)$.
Given $(\psi,\lambda)$ as above and a finite set $A \subset V$, as usual let $\partial A := \{u \in V \setminus A : (u,v) \in E \text{ for some } v \in A\}$ denote the boundary of $A$, and define a map $\Polish^{\partial A} \ni \polish_{\partial A} \mapsto \gamma^G_A( \cdot \,|\, \polish_{\partial A}) \in \P(\Polish^A)$ by 
\begin{align}
\gamma^G_A( d\polish_A \,|\, \polish_{\partial A}) = \frac{1}{Z^G_A(\polish_{\partial A})} \, \prod_{(u,v) \in E: u \in A, \, v \in A \cup \partial A} \psi(\polish_v,\polish_u) \, \prod_{w \in A}\lambda(d\polish_w), \label{def:gammaGAgibbs}
\end{align}
where $Z^G_A(\polish_{\partial A}) > 0$ is the normalizing constant.
Note that for finite $G$,  any   random element $\RV^G = (\RV_v^G)_{v \in G}$ taking values in $\Polish^G$ whose law is the
$(\psi,\lambda)$-Gibbs measure $P_G \in \P(\Polish^V)$ satisfies for every finite $A \subset V$, 
\begin{align}
\gamma^G_A(\cdot \,|\, \RV^G_{\partial A}) = \Lmc(\RV^G_A \,|\, \RV^G_{\partial A}) = \Lmc(\RV^G_A \,|\, \RV^G_{V \setminus A}) \ \ a.s. \label{eq:ap:gamma1}
\end{align}
It is clear that $\gamma^G_A = \gamma^H_A$ whenever $A \cup \partial A$ is a common subset of the vertex sets of two graphs $G$ and $H$ that induce the same subgraph on $A \cup \partial A$. The observation \eqref{eq:ap:gamma1} motivates the following definition. 

\begin{definition}
  \label{def-Gibbsinf} 
  For a general (countable, locally finite) graph $G=(V,E)$, the set $\mathrm{Gibbs}(G) =\mathrm{Gibbs}(G,\psi,\lambda) \subset \P(\Polish^V)$ of $(\psi,\lambda)$-Gibbs measures on $G$ 
 is   the set of laws $\Lmc((\RV^G_v)_{v \in V})$, where  $(\RV^G_v)_{v \in V}$ is a random element of $\Polish^V$ such that
\[
\Lmc(\RV^G_A \,|\, \RV^G_{V \setminus A}) = \gamma^G_A(\cdot \,|\, \RV^G_{\partial A}) \ \ a.s.,
\]
for each finite set $A \subset V$, where $\gamma_A^G$ is as defined in \eqref{eq:ap:gamma1}.
\end{definition}

Unlike in the finite case, when the graph is infinite, the Gibbs measure may not be unique.
However, since the reference measure $\otimes_{v\in V} \lambda$ is invariant under permutations of the vertex set of the
graph and the interaction potential $\psi$ is homogeneous in the sense that it is the same on all edges of the graph, it is easy
to see from Definition \ref{def-Gibbsinf} that  $|\mathrm{Gibbs}(G_1, \psi, \lambda)| = |\mathrm{Gibbs}(G_2, \psi, \lambda)|$
whenever $G_1$ is isomorphic to $G_2$ (see also \cite[Chapter 5]{georgii-gibbs} for related assertions).   Therefore, we can
define $\mathcal{U} = \mathcal{U}_{\psi,\lambda}$ by
\begin{align}
\mathcal{U} := \{G \in \G_* : |\mathrm{Gibbs}(G)|=1\}. \label{def:GibbsU}
\end{align}
In other words,  $\mathcal{U}$ consists of (isomorphism classes of)  locally finite graphs $G$ for which $\mathrm{Gibbs}(G)$ is a singleton.
For $G \in \mathcal{U}$, let $P_G$ denote the unique element of $\mathcal{U}$.  Note that every finite connected graph belongs to $\mathcal{U}$, so this is consistent with the notation introduced in Definition \ref{def-Gibbsfin}.
Note that if $\psi\equiv 1$ then we recover the i.i.d.\ setting, where $P_G = \lambda^G$ for each $G$ and in particular $\mathcal{U} = \G_*$. 
For any $G \in \mathcal{U}$,  let $\RV^G$ denote a random element of $\Polish^G$ with law $P_G$, and  write $(G,\RV^G)$ for the corresponding random element of $\G_*[\Polish]$.

We now state  key convergence results for Gibbs measures,  whose proofs are given in 
Appendix \ref{ap:Gibbsmeasures} for completeness. 

\begin{proposition} \label{pr:lw gibbs}
Suppose $G_n,G \in \G_*$ with $G_n \to G$ in $\G_*$.   If $G \in \mathcal{U}$, then with $\RV^{G_n}, \RV^G$  being random Gibbs configurations as defined above, $\Lmc(G_n,\RV^{G_n}) \to \Lmc(G,\RV^G)$ in $\P(\G_*[\Polish])$.
\end{proposition}

Now, if $G$ is a random element of $\mathcal{U}$ with law $M$, we may define a random element $(G,\Polish^G)$ of $\G_*[\Polish]$ in the natural way, by first sampling $G$ and then generating $\RV^G$ according to the measure $P_G$. More precisely, the law of $(G,\RV^G)$ is determined by the identity
\begin{align*}
\E[f(G,\RV^G)] &= \int_{\mathcal{U}}\E[f(H,\RV^H)]\, M(dH), \quad f \in C_b(\G_*[\Polish]).
\end{align*}
Proposition \ref{pr:lw gibbs} ensures that the integrand is continuous in $H$ on $\mathcal{U}$,
so that this is well defined.

\begin{proposition} \label{pr:lwprob gibbs}
  Suppose $G$ is a random element of $\G_*$, with $G \in \mathcal{U}$ a.s. Suppose $G_n$ are finite (possibly disconnected) random graphs such that $G_n$ converges in probability (resp.\ in law) in the local weak sense to $G$. Then, with $\RV^{G_n}, \RV^G$ being random Gibbs configurations as defined above, 
  $(G_n,\RV^{G_n})$ converges in probability (resp.\ in law) in the local weak sense to $(G,\RV^G)$.
\end{proposition}
 
An immediate consequence of Propositions \ref{pr:lw gibbs} and \ref{pr:lwprob gibbs} is  that
  analogous convergence results hold when the initial marks are i.i.d.\ with law $\lambda \in \P(\Polish)$, conditionally on the graphs $\{G_n\}$, as stated below.

\begin{corollary} \label{co:localweak iid}
  Suppose $G$ is a random element of $\G_*$, and $G_n$ is a sequence of finite (possibly disconnected) random graphs
such that $G_n$ converges in probability (resp.\ in law) in the local weak sense to $G$. 
Let $\RV^n=(\RV^n_v)_{v \in G_n}$ and $\RV=(\RV_v)_{v \in G}$ be i.i.d.\ with law $\lambda$, given the graphs.
Then  $\{(G_n,\RV^{G_n})\}$ converges in probability (resp.\ in law) in the local weak sense to $(G,\RV^G)$.
\end{corollary}

\subsection{Space of unordered terminating sequences} \label{se:symmetricsequencespace}
We will study processes that  take
values in a sequence of  configuration spaces with corresponding underlying interaction 
graphs that have  different numbers of vertices, and vertices with different degrees.  
We want to be able to specify a single function that takes as input finite sequences of elements of $\Polish$ of arbitrary length and is insensitive to the order of these elements. To this end, for a set $\Polish$, we define in this paragraph a space $\SQ(\Polish)$ of finite unordered $\Polish$-valued sequences of arbitrary length (possibly zero). First, for $k \in \N$ we define the symmetric power (or unordered Cartesian product) $S^k(\Polish)$ as the quotient of $\Polish^k$ by the natural action of the symmetric group on $k$ letters. For convenience, let $S^0(\Polish)$ denote the one-point space. Define $\SQ(\Polish)$ as the disjoint union,
\[
\SQ(\Polish) = \bigsqcup_{k=0}^\infty S^k(\Polish).
\] 
A typical element of  $\SQ(\Polish)$ will be denoted $(x_v)_{v \in V}$, for a finite (possibly empty) set $V$; if the set is empty, then by convention $(\polish_v)_{v \in V} \in S^0(\Polish)$. It must be stressed that, of course, the element $(\polish_v)_{v \in V}$ has no order. 

Suppose now that $(\Polish,d)$ is a metric space, and endow $\SQ(\Polish)$, with the usual disjoint union topology, that is, the finest topology on $\SQ(\Polish)$ for which the injection $S^k(\Polish) \hookrightarrow \SQ(\Polish)$ is continuous for each $k \in \N$. A function $F : \SQ(\Polish) \rightarrow \Polish'$ to a metric space $\Polish'$ is continuous if and only if there is a sequence $(f_k)_{k \in \N_0}$, where
 $f_0 \in \Y$ and, for each $k \in \N$, $f_k : \Polish^k \rightarrow \Polish'$ is a continuous function that is  symmetric in its $k$ variables, such that
\[
F((\polish_i)_{i \in \{1,\ldots,k\}}) = \begin{cases}
f_k(\polish_1,\ldots,\polish_k) &\text{for } k \in \N, \ (\polish_1,\ldots,\polish_k) \in \Polish^k \\
f_0 &\text{for } k=0. 
\end{cases}
\]
If $\Polish$ is separable and completely metrizeable, then so is $\SQ(\Polish)$.
Note that a sequence $(\polish^n_v)_{v \in V_n}$ in $\SQ(\Polish)$ converges to $(\polish_v)_{v \in V}$ if and only if for all $\epsilon > 0$ there exists $N \in \N$ such that  for all $n \ge N$ there exists a bijection $\varphi : V_n \rightarrow V$ such that $\max_{v \in V_n}d(\polish^n_v,\polish_{\varphi(v)}) < \epsilon$. (Note that this implicitly requires that $|V_n|=|V|$ for sufficiently large $n$.)

We now point out the advantages of using
the space $\SQ(\Polish)$, which is perhaps a bit non-standard. 
Indeed, an alternative and common  way to express symmetric functions of
vectors of arbitrary length 
is as functions of the empirical measure of the coordinates of the vector. 
However, continuous functions on $\SQ(\Polish)$ are strictly more general than weakly continuous functions on the set of empirical measures $\P_{\mathrm{emp}}(\Polish):=\{\frac{1}{n}\sum_{k=1}^n\delta_{\polish_k} : n\in\N, \, \polish_1,\ldots,\polish_n \in \Polish\} \subset \P(\Polish)$ equipped with the topology of weak convergence.
Indeed, the map $(\polish_v)_{v \in V} \mapsto \frac{1}{|V|}\sum_{v \in V}\delta_{\polish_v}$ is a continuous surjection from $\SQ(\Polish) \setminus S^0(\Polish)$ to $\P_{\mathrm{emp}}(\Polish)$, but it is not one-to-one; for instance, if $\polish \in \Polish$ then $\polish \in S^1(\Polish)$ and $(\polish,\polish)\in S^2(\Polish)$ are distinct points in $\SQ(\Polish)$ but induce the same empirical measure. In short, $\SQ(\Polish)$ encodes both the distribution of elements of the sequence as well as the numbers of elements taking any particular value,
whereas the empirical measure just encodes the distribution.
For example, the function 
\[
\SQ(\Polish) \ni (\polish_v)_{v \in V} \mapsto \sum_{v \in V}b(\polish_v) \in \R^d
\]
is continuous if $b : \Polish \to \R^d$ is continuous, but it cannot be expressed as a function on $\P_{\mathrm{emp}}(\Polish)$. 
The stronger topology on $\SQ(\Polish)$, compared to the weak convergence topology on $\P_{\mathrm{emp}}(\Polish)$, is useful in several situations; for instance, the maximum function
\[
(\polish_v)_{v \in V} \mapsto \max_{v \in V} \polish_v \in \R,
\]
is continuous on $\SQ(\R)$, and although it can be expressed also as a function on $\P_{\mathrm{emp}}(\R)$, it is not weakly continuous. 
In addition to greater generality, the use of $\SQ(\Polish)$ instead of
$\P_{\mathrm{emp}}(\Polish)$ also yields simpler notation.

\subsection{Path Spaces}
\label{subs-fspaces}

For a discrete-time process with values in a Polish state space $\X$, we write $\X^\infty = \{(x(k))_{k \in \N_0}: x(k) \in \X, k \in \N_0\}$
for the corresponding sequence space, equipped always with the product topology. For $x=(x(k))_{k \in \N_0} \in \X^\infty$, we write $x[k] := (x(0),x(1),\ldots,x(k))$ for the truncated path, an element of $\X^{k+1}$.

For a fixed positive integer $d$,  
throughout we write 
\[
\C := C(\R_+;\R^d)
\]
for the path space of continuous functions,
endowed with the topology of uniform convergence on compacts.
For $t > 0$, we write $\C_t := C([0,t];\R^d)$ and 
for $x \in \C$  we write $\|x\|_{*,t} := \sup_{s \in [0,t]}|x(s)|$ 
and $x[t] := \{x(s), s \in [0,t]\}$ for the truncated path, viewed
as an element of $\C_t$.

\section{Main results}
\label{se:mainresults}

\subsection{Discrete-time models} \label{se:disc-setup}

In this section we introduce the family of discrete-time stochastic process $X^{G,x}$ that we study.
Consider two Polish spaces $\X$ and $\Xi$, in which the state and noise processes take values, respectively. 
Recall the notation for truncated paths from Section \ref{subs-fspaces}.
For a given graph $G$ and initial condition  $x=(x_v)_{v \in G} \in \X^G$, we define processes $(X^{G,x}_v)_{v \in G}$ by the following dynamics:
\begin{equation} \label{eq:random_graph-disc}
X_v^{G,x}(k+1) = F^k\left(X_v^{G,x}[k],X_{N_v(G)}^{G,x}[k],\xi_v(k+1)\right), \quad X^{G,x}_v(0) = x_v, \ \ v \in G, \ k \in \Nmb_0.
\end{equation}
We write $X^{G,x}_v$ for the trajectory $(X^{G,x}_v(k))_{k \in \N_0}$ of particle $v$ and similarly $X^{G,x}$ for the collection of trajectories $(X^{G,x}_v)_{v \in G}$.
For $k \in \N$, we may also write $X^{G,x}[k]=(X^{G,x}_v[k])_{v \in G}$ for the collection of truncated trajectories up to time $k$.
Note that $X^{G,x}$ is well defined for any locally finite marked graph $(G,x)$, even if it is disconnected. But if $(G,x) \in \G_*[\X]$, then we may view $(G,X^{G,x})$ as a random element of $\G_*[\X^\infty]$, which depends on $(G,x)$ only through its isomorphism class.
Indeed, if $(G_1,x^1)$ and $(G_2,x^2)$ are two isomorphic marked graphs, and $\varphi : G_1\to G_2$ denotes an isomorphism,
then clearly $(X^{G_2,x^2}_{\varphi(v)})_{v \in G_1} \stackrel{d}{=} (X^{G_1,x^1}_v)_{v \in G_1}$ due to the symmetry of the dynamics \eqref{eq:random_graph-disc}
 and  assumption \hyperref[assumption:A]{(A.1)} below, and we deduce that $(G_1,X^{G_1,x^1})$ and $(G_2,X^{G_2,x^2})$ induce the same law on $\G_*[\X^\infty]$. 
For $(G,x) \in \G_*[\X]$, the law of $(G,X^{G,x})$ is denoted $P^{G,x} \in \Pmc(\G_*[\X^\infty])$.
We always assume that the transition functions $(F^k)_{k \in \N_0}$ and noises $(\xi_v(k))_{v \in V, k \in \N}$ obey the following assumption:

\begin{assumption}{\textbf{A}} \label{assumption:A} { \ }
\begin{enumerate}
\item[(\textbf{A}.1)] The $\Xi$-valued random variables 
$(\xi_v(k))_{v \in G, k \in \N}$ are i.i.d., with the same law for each (possibly disconnected) graph $G$. 
\item[(\textbf{A}.2)] $F^k \colon \X^{k+1} \times \SQ(\X^{k+1}) \times \nspace \to \X$ is continuous for each $k \in \Nmb_0$.
\end{enumerate}
\end{assumption}

The main examples we have in mind for Assumption (\ref{assumption:A}.2) take the following form: for $\tily \in \X^{k+1}, \tilz_A \in \SQ(\X^{k+1})$, and $\tile \in \nspace$,   
\begin{align}
F^k(\tily,(\tilz_v)_{v \in A},\tile) = \begin{cases}
\widetilde{F}^k_0(\tily(k),\tile) &\text{if } A = \emptyset, \\
\widetilde{F}^k\left(\tily(k),\frac{1}{|A|}\sum_{v \in A}\delta_{\tilz_v(k)}, \tile \right) &\text{if } A \neq \emptyset,
\end{cases} \label{ex:interaction}
\end{align}
for given functions $\widetilde{F}^k_0 : \X \times \nspace \to \X$ and $\widetilde{F}^k : \X \times \P(\X) \times \nspace \to \X$.
Assumption (\ref{assumption:A}.2) holds if $\widetilde{F}^k_0$ and $\widetilde{F}^k$ are continuous, with, as usual,  $\P(\X)$ equipped with the topology of weak convergence. In fact, $\widetilde{F}^k$ needs only to be defined on the subspace of $\P(\X)$ consisting of empirical measures of finitely many points, as discussed in Section \ref{se:symmetricsequencespace}.

We next define solutions of \eqref{eq:random_graph-disc} for \emph{random} graph and initial position pairs $(G,x)$ in the natural way, implicitly taking the random $(G,x)$ to be independent of $(\xi_v(k))_{v \in G, k \in \Nmb}$. 
Given $M \in \Pmc(\Gmc_*[\X])$, we define a distribution $P[M] \in \Pmc(\Gmc_*[\X^\infty])$ by $P[M] = \int P^{G,x}\,M(d(G,x))$, or more precisely by setting
\begin{align}
	\int_{\G_*[\X^\infty]} \varphi\,dP[M] := \int_{\G_*[\X]} \lan P^{G,x}, \varphi \ran\,M(d(G,x)), \label{def:P[M]-disc}
\end{align}
for bounded and continuous functions $\varphi$ on $\G_*[\X^\infty]$. 
For a deterministic $(G,x) \in \G_*[\X]$, notice that $P[\delta_{(G,x)}]=P^{G,x}$.
If $(G,x)$ is a $\G_*[\X]$-valued random variable with law $M$, we may denote by $(G,X^{G,x})$ a $\G_*[\X^\infty]$-valued random variable with law $P[M]$.

\subsection{Diffusive models} \label{se:cont-setup}

Fix a dimension $d \in \N$ and recall from Section \ref{subs-fspaces} the function space  $\C = C(\R_+;\R^d)$ and, for $x \in \C$ and $t > 0$, the notation  $\|x\|_{*,t} := \sup_{s \in [0,t]}|x(s)|$.
We are given a  drift coefficient $b$ and a diffusion coefficient $\sigma$ satisfying the following:

\begin{assumption}{\textbf{A'}} \label{assumption:C}
  The coefficients $(b,\sigma) : \R_+ \times \C \times \SQ(\C) \rightarrow \R^d \times \R^{d \times d}$ are Lipschitz, in the sense that for each $T \in (0,\infty)$, there exists $K_T <  \infty$ such that, for any $t \in [0,T]$, $\tily,y \in \C$,  and $\tilz_A=(\tilz_v)_{v \in A},\z_A=(\z_v)_{v \in A} \in \SQ(\C)$, we have
\begin{align*}
&|b(t,\tily,\tilz_A) - b(t,y,\z_A)| + |\sigma(t,\tily, \tilz_A) - \sigma(t,y,\z_A)| \\
	&\qquad\qquad\qquad\le K_T\left(\|\tily-y\|_{*,T} + \frac{1}{|A|}\sum_{v \in A}\|\tilz_v-\z_v\|_{*,T}\right), \\
\text{and } \ \ \ &\int_0^T \left(|b(t,0,(0)_{v \in A})|^2 + |\sigma(t,0,(0)_{v \in A})|^2\right)dt < \infty
\end{align*}
where the average is understood to be zero if $|A|=0$.  Moreover, $b$ and $\sigma$ are progressively measurable; that is, $b$ is jointly measurable and non-anticipative in the sense that for each $t \ge 0$, $b(t,\tily,(\tilz_v)_{v \in A}) = b(t,y,(\z_v)_{v \in A})$ whenever $\tily(s)=y(s)$ and $\tilz_v(s)=\z_v(s)$ for all $s \le t$ and $v \in A$, and similarly for $\sigma$.
\end{assumption}

We allow path-dependence in the coefficients $(b,\sigma)$ both because such interactions
arise in applications and because this does not complicate the arguments. Assumption \ref{assumption:C} ensures that the SDE system of interest is well defined, and we relegate the standard proof to Appendix \ref{ap:uniqueness-infSDE}:

\begin{theorem}  \label{th:uniqueness-hominfSDE}
Suppose Assumption \ref{assumption:C} holds. Then for
each marked graph $(G,x) \in \G_*[\R^d]$, there exists a pathwise unique strong solution of the following SDE system: 
\begin{align}
dX^{G,x}_v(t) &= b(t,X^{G,x}_v,X^{G,x}_{N_v(G)})dt + \sigma(t,X^{G,x}_v,X^{G,x}_{N_v(G)})dW_v(t), \ \ X^{G,x}_v(0)=x_v, \ \ v \in G. \label{statements:SDE}
\end{align}
\end{theorem}

For each countable locally finite but possibly disconnected graph $G$ and each $x=(x_v)_{v \in G}$, the SDE system \eqref{statements:SDE} again admits a unique in law weak solution, constructed by simply combining the solutions on the different connected components of $G$, each of which is unique in law by Theorem \ref{th:uniqueness-hominfSDE}. We view $X^{G,x}_v$ as a $\C$-valued random variable for each $v \in G$, and 
$X^{G,x}=(X^{G,x}_v)_{v \in G}$ as a $\C^G$-valued random variable. For each $(G,x) \in \G_*[\R^d]$, we may view $(G,X^{G,x})$ as a random element of $\G_*[\C]$, and we write $P^{G,x}$ for its law.

We define solutions of \eqref{statements:SDE} for \emph{random} marked graphs $(G,x)$ in the natural way, as in the discrete case:
Given $M \in \P(\G_*[\R^d])$, we define $P[M] \in \P(\G_*[\C])$ by 
\begin{align}
\int_{\G_*[\C]} \varphi\,dP[M] := \int_{\G_*[\R^d]} \lan P^{G,x}, \varphi \ran\,M(d(G,x)), \label{def:P[M]-cont}
\end{align}
for $\varphi \in C_b(\G_*[\C])$. 
If $(G,x)$ is a $\G_*[\R^d]$-valued random variable with law $M$, we denote by $(G,X^{G,x})$ a $\G_*[\C]$-valued random variable with law $P[M]$.

\subsection{Local convergence of particle systems} \label{se:statements:localconvergence}

We can now state the first main results, proven in Section \ref{se:locallimits}, which guarantee that $P[M]$ is well-defined and varies continuously with $M$, for  both discrete time dynamics and diffusive systems. 

\begin{theorem}[Discrete time] \label{thm:local-law-disc}
Under Assumption \ref{assumption:A}, the following hold:
\begin{enumerate}[label=(\roman*)]
\item If $(G_n,x^n) \to (G,x)$ in $\G_*[\X]$ as $n \to \infty$, then $P^{G_n,x^n} \to P^{G,x}$ in $\P(\G_*[\X^\infty])$. 
\item For every $M \in \P(\G_*[\X])$, the measure $P[M]$ is well defined. Furthermore, if $M_n \to M$ in $\P(\G_*[\X])$ as $n \to \infty$, then $P[M_n] \to P[M]$ in $\P(\G_*[\X^\infty])$.
\end{enumerate}
\end{theorem}

In continuous time, for technical reasons we restrict to uniformly bounded initial conditions. Let $B_r(\R^d)$ denote the centered closed ball of radius $r > 0$ in $\R^d$.

\begin{theorem}[Diffusions] \label{thm:local-law-cont}
Suppose Assumption \ref{assumption:C} holds, and let $r \in (0,\infty)$. Then the following hold:
\begin{enumerate}[label=(\roman*)]
\item If $\sup_{n \in \Nmb} \sup_{v \in G_n} |x^n_v| \le r$ and $(G_n,x^n) \to (G,x)$ in $\G_*[B_r(\R^d)]$ as $n \to \infty$, then $P^{G_n,x^n} \to P^{G,x}$ in $\P(\G_*[\C])$.
\item For every $M \in \P(\G_*[B_r(\R^d)])$, the measure $P[M]$ is well defined. Furthermore, if $M_n \to M$  in $\P(\G_*[B_r(\R^d)])$ as $n \to \infty$, then $P[M_n] \to P[M]$  in $\P(\G_*[\C])$.
\end{enumerate}
\end{theorem}

\begin{remark} \label{re:gibbs}
    See Section \ref{se:example graphs} for examples of common locally convergent graph sequences, including \Erdos, configuration models, and random regular graphs.
There are many more examples, and in particular in Section \ref{se:lattices + trees} we briefly discuss the cases where $G$ is a lattice or regular tree, and
$\{G_n\}$ is a growing sequence of subgraphs.
As discussed in Section \ref{subsub-icconv} (see also  Appendix \ref{ap:Gibbsmeasures}),
these immediately also imply convergence of the sequences of initial conditions $\{(G_n,x^n)\}$
beyond the simple case of i.i.d.\ initial positions,
when $x^n$ is a Gibbs measure and the (infinite-volume) Gibbs measure on the limiting graph
is unique.
\end{remark}

\begin{remark} \label{re:local-law-root}
The ``root particle map" $(G,\oSlash, x) \mapsto x_{\oSlash}$ is always continuous, where we recall that $\oSlash$ denotes the root of the graph.
Thus, in the setup of Theorem \ref{thm:local-law-disc}: If a sequence of random variables $\{(G_n,x^n)\}$ converges in law to $(G,x)$ in $\G_*[\X]$, then the root particle $\{X^{G_n,x^n}_{\oSlash}\}$ converges in law in $\X^\infty$ to the root particle $X^{G,x}_{\oSlash}$ of the limiting graph.  Similarly, in the setup of Theorem \ref{thm:local-law-cont}: If $\{(G_n,x^n)\}$ converges in law to $(G,x)$ in $\G_*[B_r(\R^d)]$ for some $r > 0$, then $\{X^{G_n,x^n}_{\oSlash}\}$ converges in law in $\C$ to $X^{G,x}_{\oSlash}$.
\end{remark}

\subsection{Convergence of the global empirical measure to a deterministic limit} \label{se:statements:globalemp}

The analysis of the limit of the empirical measure is more subtle. 
For a finite (possibly disconnected) graph $G$ and initial state $x=(x_v)_{v \in G}$ (in $\X^G$ in the discrete case or $(\R^d)^G$ in the diffusive case), we define the \emph{global empirical measure} 
\begin{equation}
  \label{def-empmeasG}
	\mu^{G,x} := \frac{1}{|G|} \sum_{v \in G} \delta_{X_v^{G,x}}.
\end{equation}
Note that this is a random measure on $\X^\infty$ in the discrete case and $\C$ in the continuous case.
In the mean field setting, that is, when $G_n$ is the complete graph on $n$ vertices and $x^n=(x^n_v)_{v \in G_n}$ are i.i.d.\ (or, more generally, chaotic), it is well known that, under suitable assumptions on the coefficients, $\mu^{G_n,x^n}$ converges to a deterministic measure as $n\rightarrow\infty$, identified as the limit law of any single particle. 
This obviously fails for general graphs. In particular, if $(G_n,x^n) \rightarrow (G,x)$ locally with $G$ and $G_n$ finite for each $n$, then $\mu^{G_n,x^n}$ converges in law to the random measure $\mu^{G,x}$ (see Proposition \ref{pr:empirical-finite}).  
However, we show that the global empirical measure does converge to a deterministic limit under the stronger assumption of \emph{convergence in probability in the local weak sense}, discussed in Section \ref{se:local weak in prob}.
 Recall from Lemma \ref{le:empmeas general} that this mode of convergence implies the convergence of the empirical measure of the particles.  
The following two theorems
   are proved in Section \ref{se:pf:emp-global}.

\begin{theorem}[Discrete time] \label{thm:localprob-disc}
Suppose Assumption \ref{assumption:A} holds, and suppose a sequence of random finite marked graphs $\{(G_n,x^n)\}$ converges in probability in the local weak sense to some random element $(G,x)$ of $\G_*[\X]$.
Then $\{(G_n,X^{G_n,x^n})\}$ converges in probability in the local weak sense to $(G,X^{G,x})$.
In particular, $\{\mu^{G_n,x^n}\}$ converges in probability in $\P(\X^\infty)$ 
to $\L(X^{G,x}_{\oSlash})$, where $\oSlash$ is the root of $G$.
\end{theorem}

\begin{theorem}[Diffusions] \label{thm:localprob-cont}
	Suppose Assumption \ref{assumption:C} holds, and suppose a sequence of random finite marked graphs $\{(G_n,x^n)\}$ converges in probability in the local weak sense to some random element $(G,x)$ of $\G_*[\R^d]$. 
	Assume also that there exits $r \in (0,\infty)$ such that  $x^n_v \in B_r(\R^d)$ a.s.\ for each $v \in G_n$ and $n \in \N$.
	Then $\{(G_n,X^{G_n,x^n})\}$ converges in probability in the local weak sense to $(G,X^{G,x})$.
	In particular, $\{\mu^{G_n,x^n}\}$ converges in probability in $\P(\C)$ 
	to $\L(X^{G,x}_{\oSlash})$, where $\oSlash$ is the root of $G$.
\end{theorem}

Note again that Theorems \ref{thm:localprob-disc} and \ref{thm:localprob-cont} all cover the case of i.i.d.\ initial states $(x^n_v)_{v \in G_n, n \in \N}$ where $G_n$ is any of the examples of Section \ref{se:example graphs}, such as \Erdos, configuration models, or random regular graphs.
Beyond the i.i.d.\ case,
the discussion of Remark \ref{re:gibbs} can be repeated verbatim in the context of convergence in probability in the local weak sense.
In particular, the initial states may be generated by certain families of Gibbs measures, as long as the Gibbs measure on the limiting graph is \emph{unique}; see
Section \ref{subsub-icconv} and Appendix \ref{ap:Gibbsmeasures} for details.

\begin{remark} \label{re:OliReiSto-comparison}
  In the diffusion setting, our Theorem \ref{thm:localprob-cont} is very similar to Theorem 5 and Corollary 1 of \cite{OliReiSto19}, respectively, though we work with different assumptions. The results of \cite{OliReiSto19} allow on the one hand for more general directed and weighted networks than we consider, as well as random media. On the other hand, they work only with uniformly rooted graphs satisfying an exponential growth assumption \cite[Definition 9]{OliReiSto19}, and there is no counterpart to our
  Theorems \ref{thm:local-law-cont} or \ref{thm:CompEmpMeas}. In addition, our setup allows for more general forms of interactions, as well as unbounded drift, non-constant diffusion coefficient, and non-Markovian dynamics. 
Furthermore, although there are some parallels, our proofs are quite different from those of \cite{OliReiSto19}. 
Our proof of Theorem \ref{thm:localprob-cont} is based fundamentally on the fact that convergence in probability in the local weak sense ensures the asymptotic independence of the component graphs $(\comp_{U^n_1}(G_n),\comp_{U^n_2}(G_n))$, when $U^n_1$ and $U^n_2$ are independent random uniformly distributed vertices.
On the other hand, the proof of \cite[Theorem 6]{OliReiSto19}
establishes explicit quantitative estimates, using their assumption that the networks have exponential growth.  
\end{remark}

\subsection{Convergence of the connected component empirical measure}   \label{se:statements:componemp}
  Our next main result illustrates that the behavior of the global empirical measure can be markedly different
  for sequences of graphs that converge  only in law,  
   and not  in probability, in the local weak sense.
Recall from Section \ref{se:example graphs} that, for a finite graph $G_n$, the random graph $\compu(G_n)$ is the connected component of $G_n$ containing a (uniformly) randomly chosen root.
Thus,
\begin{equation}
  \label{def-empmeas}
\mu^{\compu(G_n,x^n)} = \frac{1}{|\compu(G_n)|}\sum_{v \in \compu(G_n)}\delta_{X_v^{G_n,x^n}}
\end{equation}
might be called the \emph{connected component empirical measure}.
The analysis of this random measure is more delicate, so we focus on  two important cases,  the \Erdos \ graph $G_n \sim \Gmc(n,p_n)$ with $np_n \to \parm \in (0,\infty)$ and the configuration model $G_n \sim \CM(n,d(n))$ with $\frac{d_k(n)}{n} \to \rho_k$ for some distribution $\rho \in \P(\N_0)$ with finite nonzero first and second moments, as discussed in Examples \ref{ex:GW} and \ref{ex:unimodularGW} in Section \ref{se:example graphs}, although our proof in Section \ref{se:componemp} in fact applies to a larger
class of graph sequences (that satisfy Condition \ref{cond-graphs} therein). 
In these cases the limit of $\mu^{\compu(G_n,x^n)}$ is  (usually)  a \emph{random} measure (see Remark \ref{rem-randomemp}).   
This is intuitively clearest in the subcritical regime, where the limit tree is almost surely finite. 
Theorem \ref{thm:CompEmpMeas} below explains this precisely.
Note that the fact $\mu^{\compu(G_n,x^n)}$ converges to a \emph{random} limit does not contradict Theorem \ref{thm:localprob-disc} or Theorem \ref{thm:localprob-cont}. On the one hand, the sequence of connected component graphs $\{\compu(G_n)\}$ has the same limit as $\{G_n\}$ in the sense of local convergence in distribution, as in Definition \ref{def:localweakdist}. In addition, the sequence $\{G_n\}$ converges \emph{in probability} in the local weak sense, in the sense of Definition \ref{def:localweakprob}. But, on the other hand, the sequence $\{\compu(G_n)\}$ does not converge in probability in the local weak sense in general, and thus Theorems \ref{thm:localprob-disc} and \ref{thm:localprob-cont} do not apply to the graph sequence $\{\compu(G_n)\}$.

We will make the following assumption on the  initial conditions or, equivalently, distribution of marks on the random graphs $(G,x)$ and $\{(G_n,x^n)\}$:

\begin{assumption}{\textbf{B}} \label{assumption:B}
For any Borel set $A \subset \G_*$  with $\PP(G \in A) > 0$, and any random induced subgraph sequence $H_n \subset G_n$ that converges in probability in the local weak sense to a random element of $\G_*$ that has law $\L(G \,|\, G \in A)$,  the marked random graph $(H_n, x^n_{H_n})$ converges in probability in the local weak sense to 
a   marked random graph that has law   $\L((G,x)\,|\,\, G \in A)$.
\end{assumption}

Assumption \ref{assumption:B} holds, for example, when the initial states are i.i.d. More generally, it holds if $x^n$ and $x$ are given by $(\psi,\lambda)$-Gibbs measures on $G_n$ and $G$, respectively, as discussed in Section \ref{subsub-icconv}, as long as the (infinite-volume) Gibbs measure is unique in the
sense that $\PP(G \in \mathcal{U}_{\psi,\lambda})=1$, with $\mathcal{U}_{\psi,\lambda}$ as in \eqref{def:GibbsU}. See Corollary \ref{co:assmpB-Gibbs} for a proof. 

Recall in the following the \Erdos \ model and configuration models and their local limits, described in Examples \ref{ex:GW} and \ref{ex:unimodularGW}, respectively.

\begin{theorem} \label{thm:CompEmpMeas}
Suppose $G$ and $G_n$ satisfy one of the following: 
\begin{enumerate}[label=(\roman*)]
\item $G = \tree \sim \mathrm{UGW}(\mathrm{Poisson}(\parm))$ and $G_n \sim \Gmc(n,p_n)$ with $np_n \rightarrow \parm$, for some $\parm \in (0,\infty)$.
\item $G = \tree \sim \mathrm{UGW}(\rho)$ and $G_n \sim \CM(n,d(n))$,  for a graphic sequence $d(n)=(d_1(n),\dotsc,d_n(n))$ that is well-behaved in the following sense:
there exists $\rho \in \P(\N_0)$ with 
nonzero first moment and $\rho_2<1$ such that 
\[ \frac{d_k(n)}{n} \to \rho_k \mbox{ for each  } k \in \N_0 \quad \mbox{  and  } \quad \frac{1}{n}\sum_{k=1}^n k^2d_k(n) \to \sum_{k \in \N_0}k^2\rho_k < \infty,
\]
    and furthermore,  there exists $\delta > 0$ such that for all $n$, $d_k(n)=0$ whenever $k \ge n^{1/4-\delta}$.
\end{enumerate}
In the discrete-time case, suppose Assumption \ref{assumption:A} holds, write $\base=\X$ and
$\pathspace = \X^\infty$, and  let $X^{G_n,x^n}$ and $X^{G,x}$ be 
defined via \eqref{eq:random_graph-disc}.
In the  diffusive case, suppose Assumption \ref{assumption:C} holds, write $\base=\R^d$ and
$\pathspace = \C$, and let $X^{G_n,x^n}$ and $X^{G,x}$ be defined by \eqref{statements:SDE}.  
In both cases, 
assume the random initial states $x$ and $\{x^n\}$  are such that Assumption \ref{assumption:B} holds.
Then the $\P(\base_\infty)$-valued connected component empirical measure $\mu^{\compu(G_n,x^n)}$ defined in \eqref{def-empmeas} converges in law to the (random) empirical measure $\widetilde\mu^{\Tmc,x}$ defined by
\begin{align*}
\widetilde\mu^{\Tmc,x} = \begin{cases}
\mu^{\Tmc,x} &\text{on } \{|\Tmc| < \infty\} \\
\Lmc(X^{\Tmc,x}_{\oSlash} \, | \, |\Tmc| = \infty) &\text{on } \{|\Tmc|=\infty\}.
\end{cases}
\end{align*}
\end{theorem}

In fact, as mentioned earlier, in Theorem \ref{thm:GenCompEmpMeas} of Section \ref{se:componemp} we establish an analogous result that holds for a general class of graph sequences
that satisfy Condition \ref{cond-graphs} stated therein. 
Theorem \ref{thm:CompEmpMeas} follows as an immediate consequence of this more general result and Proposition \ref{prop-graphs}, which verifies this condition for the two graph sequences stated in the theorem.
The assumptions on the degree sequence in case (ii) of Theorem \ref{thm:CompEmpMeas} stem from the seminal work of \cite{MolloyReed1998size} on the behavior of the largest connected component in $\CM(n,d(n))$, and in particular, the case of $\rho_2 = 1$ is subtle; see Remark \ref{rem-MolloyReed} for further discussion.

\begin{remark}
  \label{rem-randomemp}
  Note in Theorem \ref{thm:CompEmpMeas} that $\mu^{\tree,x}$ is non-random if and only if $|\tree|=\infty$ almost surely, which never occurs in case (i) but can occur in
  case (ii). 
In case (ii), define 
\begin{equation} \label{def-parm}
\parm := \frac{\sum_{k=0}^\infty k(k-1)\rho_k}{\sum_{k=0}^\infty k\rho_k}.  
\end{equation} 
Interestingly, in both cases, 
in the subcritical regime $\parm \le 1$, we have $|\tree|<\infty$ almost surely (see, e.g., \cite[Theorems 2.1.2 and 2.1.3]{durrett2007random}), and the limit $\L(\mu^{\Tmc,x})$ of $\Lmc(\mu^{\compu(G_n,x^n)})$ may be nonatomic, although its support is a set of discrete measures. 
Note that $\L(\mu^{\Tmc,x})$ is an element of $\Pmc(\Pmc(\X^\infty))$ in the discrete case or $\P(\P(\C))$ in the continuous case. 
On the other hand, in the supercritical regime $\parm > 1$, the measure $\L(\widetilde\mu^{\Tmc,x})$ always has an atom with mass
$\Pmb(|\Tmc|=\infty) >0$ (see, e.g., \cite[Theorem 2.1.4]{durrett2007random}) at the point $\Lmc(X^{\Tmc,x}_{\oSlash} \, | \, |\Tmc| = \infty)$. 
This mass could be $1$ when $\rho_1=0$ in case (ii), and tends to $1$ as $\parm \rightarrow\infty$ in case (i), that is, as the graphs become increasingly dense.
\end{remark}

\subsection{Lattices and trees} \label{se:regtree}

In this section we highlight what can go wrong for less homogeneous graph sequences. Suppose $\Tmb^d_n$ is the $d$-regular tree of height $n$, for $d \ge 2$. That is, all vertices except the leaves have degree $d$, and all leaves are a distance $n$ from the root $\oSlash$. This tree has $d(d-1)^
{h-1}$ vertices at distance $h$ from the root, and there are
\[
|\Tmb^d_n| = 1 + d\sum_{h=0}^{n-1}(d-1)^h = 1 + d\frac{(d-1)^n - 1}{d - 2}
\]
vertices in total. As $n\rightarrow\infty$, the fraction of vertices which are leaves approaches $(d-2)/(d-1)$. That is, a macroscopic proportion of vertices are leaves, and this greatly influences the behavior of the empirical measure $\mu^{\Tmb^d_n,x^n}$. Particles at different heights behave differently, and the root particle is the only particle at height zero. Hence, one cannot expect the empirical measure to converge to the same limit law as the root particle, and we show in Proposition \ref{pr:regtree} that indeed $\lim_{n\to\infty}\mu^{\Tmb^d_n,x^n} \neq \lim_{n\to\infty}\Lmc(X^{\Tmb^d_n,x^n}_{\oSlash})$. The point is that we should expect from Theorem \ref{thm:localprob-disc} that $\mu^{\Tmb^d_n,x^n}$ behaves not like $\Lmc(X^{\Tmb^d_n,x^n}_{\oSlash})$ but rather like $\Lmc(X^{\Tmb^d_n,x^n}_{U_n})$, where $U_n$ is a \emph{uniformly random} vertex in $\Tmb^d_n$. In this situation, the local limits of $(\Tmb^d_n,U_n)$ and $(\Tmb^d_n,\oSlash)$ are quite different; the latter is the infinite $d$-regular tree, whereas the former is the so-called $d$-\emph{canopy tree} defined in Section \ref{se:lattices + trees}. 

On the other hand, in Section \ref{se:lattices + trees} we show that for lattices the story is simpler. If $\Z^d$ denotes the $d$-dimensional integer lattice and $\Z^d_n = \Z^d \cap [-n,n]^d$, then again particles at vertices of different distances from the origin behave quite differently. But the graph $\Z^d_n$ grows much more slowly than the tree, and the ``boundary" vertices occupy a vanishing fraction of the graph as $n\to\infty$. More precisely, we have here that $(\Z^d_n,0)$ and $(\Z^d_n,U_n)$ both converge locally to $(\Z^d,0)$, where here $U_n$ is a uniformly random vertex in $\Z^d_n$. Hence, in this case,  $\lim_{n\to\infty}\mu^{\Z^d_n,x^n} = \lim_{n\to\infty}\Lmc(X^{\Z^d_n,x^n}_0) = \Lmc(X^{\Z^d,x^n}_0)$, as we show in Proposition \ref{pr:lattice} for i.i.d.\ initial states $x^n$.

Along the way to proving the above claims for lattices and trees, Section \ref{se:nonrandom} will also develop the more general principles of propagation of empirical field convergence (Section \ref{se:empfield}) and propagation of ergodicity (Section \ref{se:propergo}) on a fixed graph.

\section{Local convergence of particle systems} \label{se:locallimits}

In this section we formalize and prove Theorems \ref{thm:local-law-disc} and \ref{thm:local-law-cont}, along the way proving that the measures $P[M]$ defined in Sections \ref{se:disc-setup} and \ref{se:cont-setup} are well defined. Recall from Section \ref{se:background-localmarkedgraphs} the definitions of the spaces of rooted connected graphs and marked graphs, and refer to \cite[Lemma 3.4]{Bordenave2016} and Appendix \ref{ap:localconvergence} for the essential facts about these Polish spaces.
We begin with the discrete-time case:

\begin{proof}[Proof of Theorem \ref{thm:local-law-disc}]
We begin with (i), which we prove inductively after noting first that $(G_n,X^{G_n,x^n}(0))=(G_n,x^n) \to (G,x)$ by assumption. Now suppose $(G_n,X^{G_n,x^n}[k])$ converges in law to $(G,X^{G,x}[k])$ for some $k \in \N_0$, where we recall that $X^{G,x}_v[k]=(X^{G,x}_v(0),\ldots,X^{G,x}_v(k))$ denotes the trajectory up to time $k$.
Recalling the structure of the dynamics \eqref{eq:random_graph-disc}, we may write
\begin{align*}
(G_n,X^{G_n,x^n}[k+1]) &= \Psi_k\big(G_n,X^{G_n,x^n}[k],\xi(k+1)\big),
\end{align*}
where $\Psi_k : \G_*[\X^{k+1} \times \nspace] \to \G_*[\X^{k+2}]$ is defined by
\begin{align*}
\Psi_k(H,y[k],\xi(k+1)) &= \big(H, \, \big(\overline{F}^k(y_v[k],y_{N_v(H)}[k],\xi_v(k+1))\big)_{v \in H}\big),
\end{align*}
where $\overline{F}^k : \X^{k+1} \times \SQ(\X^{k+1}) \times \nspace \to \X^{k+2}$ is given,
for $\tily \in {\mathcal X}^{k+1}$, $\tilz_A = (\tilde{z}_v)_{v\in A}\in \SQ(\X^{k+1})$, and $\tile \in \nspace$, by
\begin{align*}
 \overline{F}^k(\tilde{y}, \tilz_A, \tile  )(i) &= \begin{cases}
 F^k(\tily, \tilz_A,\tile) &\text{if } i=k+1, \\
 \tily (i)  &\text{if } i =0,1,\ldots,k, 
 \end{cases}
\end{align*}

Clearly $\overline{F}^k$ is continuous because $F^k$ is, by Assumption \ref{assumption:A}. It is straightforward to check that $\Psi_k$ is thus also continuous. Now, because 
$(\xi_v(k+1))_{v \in V}$ are independent of $X^{G_n,x^n}[k]$, it is easy to deduce from the convergence in law of $(G_n,X^{G_n,x^n}[k])$ to $(G,X^{G,x}[k])$ in $\G_*[\X^{k+1}]$ that we also have the convergence in law of $(G_n,X^{G_n,x^n}[k],\xi(k+1))$ to $(G,X^{G,x}[k],\xi(k+1))$ in $\G_*[\X^{k+1} \times \nspace]$. It then follows from the continuous mapping theorem that $(G_n,X^{G_n,x^n}[k+1])$ converges in law to $(G,X^{G,x}[k+1])$ in $\G_*[\X^{k+2}]$, which completes the proof of (i).

To prove (ii), note that it follows from (i) that the map $\G_*[\X] \ni (G,x) \mapsto \lan P^{G,x}, \varphi \ran \in \R$ is continuous, for each $\varphi \in C_b(\G_*[\X^\infty])$.
From this and the definition \eqref{def:P[M]-disc} it then follows immediately that the measure $P[M]$ is well defined for every $M \in \P(\G_*[\X])$, and that $P[M_n] \to P[M]$ in $\P(\G_*[\X^\infty])$ whenever $M_n \to M$ in $\P(\G_*[\X])$.
This completes the proof.
\end{proof}

The proof of Theorem \ref{thm:local-law-cont},  the continuous-time case, is more involved,
and we break it down into two steps. First, we prove a version of a standard second moment estimate, which for our purposes must be uniform in the choice of graph. Then, the main line of the proof is an adaptation of standard Lipschitz-based stability arguments to the setting of $\G_*[\C]$. Recall in the following that $B_\ell(\R^d)$ denotes the closed ball of radius $\ell \in (0,\infty)$:

\begin{lemma} \label{le:momentbound-tightness}
  Suppose Assumption \ref{assumption:C} holds. 
  Then, for each $r,T \in  (0,\infty)$,
\begin{align}
\sup_{(G,x) \in \G_*[B_r(\R^d)]}\sup_{v \in G}\E\left[\sup_{t \in [0,T]}|X^{G,x}_v(t)|^2\right] < \infty. \label{le:def:momentbound}
\end{align}
\end{lemma}
\begin{proof}
Recall the notation $\|x\|_{*,t} = \sup_{0 \le s \le t}|x(s)|$ for $x \in \C$.
Fix $T \in (0,\infty)$ and $(G,x) \in \G_*[B_r(\R^d)]$. A standard argument using Assumption \ref{assumption:C}, Doob's inequality, and Gronwall's inequality yields, for each $v \in G$,
\begin{align*}
\E\left[\|X_v^{G,x}\|_{*,t}^2\right] &\le C\left(1 + |x_v|^2 + \frac{1}{N_v(G)}\sum_{u \in N_v(G)}\E\left[\int_0^t\|X^{G,x}_u\|_{*,s}^2ds\right]\right),
\end{align*}
where the constant $C$  depends only on $T$ and the growth constant $K_T$ of Assumption \ref{assumption:C}. In particular,
$C$ does not depend on $(G,x)$ or $v$. Hence, recalling that $x \in B_r(\R^d)$,  it follows that 
\begin{align*}
\sup_{v \in G}\E\left[\|X_v^G\|_{*,t}^2\right] &\le C\left(1 + r^2 + \int_0^t \sup_{v \in G}\E\left[\|X_v^G\|_{*,s}^2\right] ds\right),
\end{align*}
and we complete the proof of \eqref{le:def:momentbound} by another application of Gronwall's inequality.
\end{proof}

\begin{proof}[Proof of Theorem \ref{thm:local-law-cont}]
We prove only (i), because (ii) follows from (i) exactly as in the last paragraph of the proof of Theorem \ref{thm:local-law-disc}, except with $\X$ and $\X^\infty$ replaced by $\R^d$ and $\C$, respectively. 
We prove convergence of $(G_n,X^{G_n,x^n})$ to $(G,X^{G,x})$ by coupling the two systems and using the Lipschitz assumption.
It may be helpful here to recall the notation introduced in Section \ref{se:background-localmarkedgraphs} for marked graph convergence. In particular, $B_k(G)$ denotes the ball of radius $k \in \N$ around the root in a rooted graph $G$.
While $(G,x)$ denotes not a marked graph but rather an isomorphism class thereof, we use the same notation to denote an arbitrary representative, and similarly for $(G_n,x^n)$ for each $n$. We may further assume without loss of generality that the vertex sets of $G_n$ and $G$ are contained in $\N$. 

To prove the convergence of $(G_n,X^{G_n,x^n})$ to $(G,X^{G,x})$ in $\G_*[\C]$ it suffices to prove the convergence of $(G_n,X^{G_n,x^n}[T])$ to $(G,X^{G,x}[T])$ in $\G_*[\C_T]$ for each $T \in (0,\infty)$. Thus, we fix $T \in (0,\infty)$ for the rest of the proof.
Note that by Lemma \ref{le:momentbound-tightness} there exists $M < \infty$ such that
\begin{align}
\sup_{n \in \N}\sup_{v \in G_n}\E\left[\sup_{t \in [0,T]}|X^{G_n,x^n}_v(t)|^2\right] \le M, \quad \text{ and }\quad \sup_{v \in G}\E\left[\sup_{t \in [0,T]}|X^{G,x}_v(t)|^2\right] \le M. \label{pf:localcont-1}
\end{align}

Let $\epsilon > 0$ and $\ell \in \N$. Choose $k \in \N$ large, to be determined later but certainly strictly larger than $\ell$. Choose $N \in \N$ large enough such that for each $n \ge N$ there exists an isomorphism $\varphi_n : B_{k+1}(G) \to B_{k+1}(G_n)$ such that
\begin{align}
\max_{v \in B_{k+1}(G)}| x_v-x^n_{\varphi_n(v)} |^2 \le \frac{\epsilon}{3}. \label{pf:localcont-2}
\end{align}
This is possible of course because we assumed $(G_n,x^n) \to (G,x)$ in $\G_*[B_r(\R^d)]$. Now suppose that $(\Omega,\F,\FF,\PP)$ is a filtered probability space supporting independent $\FF$-Brownian motions $(W_v)_{v \in G}$ and $(B^n_v)_{v \in G_n, \, n \in \N}$. By Theorem \ref{th:uniqueness-hominfSDE}, we may construct a unique strong solution of the SDE system
\begin{align*}
dX_v(t) &= b(t,X_v,X_{N_v(G)})dt + \sigma(t,X_v,X_{N_v(G)})dW_v(t), \ \ X_v(0)=x_v, \ \ v \in G.
\end{align*}
Define $W^n_v = W_{\varphi_n^{-1}(v)}$ for $v \in B_{k+1}(G_n)$ and $W^n_v=B^n_v$ for $v \in G_n \backslash B_{k+1}(G_n)$, and then consider the SDEs
\begin{align*}
 dX^n_v(t) &= b(t,X^n_v,X^n_{N_v(G_n)})dt + \sigma(t, X^n_v,X^n_{N_v(G_n)})dW_v^n(t), \ \ X^n_v(0)=x^n_v, \ \ v \in G_n.
\end{align*}
This way $X \stackrel{d}{=} X^{G,x}$ and $X^n \stackrel{d}{=} X^{G_n,x^n}$ for each $n$.
Define $Y^n_v := X^n_{\varphi_n(v)}$ for $v \in B_k(G)$. Noting that $\{X^n_u : u \in N_{\varphi_n(v)}(G_n)\} = \{Y^n_u : u \in N_{v}(G)\}$ and $W_{\varphi_n(v)}^n = W_v$ for $v \in B_k(G)$,  we find
\begin{align*}
dY^n_v(t) &= b(t,Y^n_v,Y^n_{N_v(G)})dt + \sigma(t,Y^n_v,Y^n_{N_v(G)})dW_v(t), \ \ Y^n_v(0)=x^n_{\varphi_n(v)}, \ \ v \in B_k(G).
\end{align*}
We are now in a position to compare $X_v$ and $Y^n_v$ for $v \in B_k(G)$. To this end, let $\Delta_v^n(t) = \E[ \|X_v - Y^n_v\|_{*,t}^2]$
for $0 \le t \le T$, recalling that $\|\cdot\|_{*,t}$ denotes the supremum norm over $[0,t]$. Then, for $0 \le t \le T$, a standard estimate using the Lipschitz continuity of $b$ and  $\sigma$ from Assumption \ref{assumption:C}, the Cauchy-Schwarz inequality, the It\^{o} isometry and  Doob's inequality yield 
\begin{align*} 
\Delta_v^n(t) &\le 3|x^n_{\varphi_n(v)}-x_v|^2 + \frac12C_1 \int_0^t\Bigg(\Delta_v^n(s) + \frac{1}{|N_v(G)|}\sum_{u \in N_v(G)}\Delta_u^n(s)\Bigg)ds \\
	&\le \epsilon + C_1 \int_0^t \max_{u : \, d_G(u,v) \le 1}\Delta_u^n(s) ds, 
\end{align*} 
where the last inequality uses \eqref{pf:localcont-2},   
$C_1 := 12(T+4)K_T^2$ with $K_T$ equal to the Lipschitz constant of Assumption \ref{assumption:C}, and $d_G$ denotes graph distance in $G$.
 For $d_G(v,\oSlash) < k$, we iterate this inequality $m := k-d_G(v,\oSlash)$ times to reach the boundary of $B_k(G)$:
\begin{align*}
\Delta_v^n(t) &\le  \epsilon + C_1 \int_0^t \max_{u : \, d_G(u,v) \le 1}\Delta_u^n(t_1) \, dt_1 \\
	&\le \epsilon + C_1t\epsilon + C_1^2 \int_0^t\int_0^{t_1} \max_{u : \, d_G(u,v) \le 2}\Delta_u^n(t_2) \, dt_2 \,dt_1 \\
	&\le \cdots \\
	&\le \epsilon\sum_{j=0}^{m-1}\frac{(C_1t)^j}{j!} + C_1^m \int_0^t\int_0^{t_1}\cdots \int_0^{t_{m-1}} \max_{u : \, d_G(u,v) \le m}\Delta_u^n(t_m) \, dt_m \cdots dt_2\,dt_1 \\
	&\le \epsilon e^{C_1T} + 4M\frac{(C_1t)^m}{m!},
\end{align*}
where the final step used $\sup_{n \in \N} \sup_{v \in G}\Delta_v^n(u) \le 4M$, which follows from \eqref{pf:localcont-1}. Note next that $x^m/m!$ is monotonically decreasing in $m$ for $m > x > 0$. Recall that we fixed $\ell \in \N$ at the beginning of the proof. If we choose $k$ so that $k-\ell > C_1 T$, then we get
\begin{align}
\max_{v \in B_{\ell}(G)}\E[\|X^n_{\varphi_n(v)}-X_v\|_{*,T}^2] = \max_{v \in B_{\ell}(G)}\Delta_v^n(T) &\le \epsilon e^{C_1T} + 4M\frac{(C_1 T)^{k-\ell}}{(k-\ell)!}, \label{pf:localcont-3}
\end{align}
using the aforementioned monotonicity of $m \mapsto x^m/m!$ along with the fact that $m = k-d_G(v,\oSlash) \ge k-\ell$ for $v \in B_\ell(G)$.

Now consider the metric $d_{*,1}$ on $\G_*[\C_T]$ defined in Appendix \ref{ap:localconvergence}:
\begin{equation*}
d_{*,1}((H,{y}),(H',{y}')) = \sum_{j=1}^\infty 2^{-j}\left(1 \wedge \inf_{\varphi \in I(B_j(H),B_j(H'))}\frac{1}{|B_j(H)|}\sum_{v \in B_j(H)}\|y_v - y'_{\varphi(v)}\|_{*,T}\right),
\end{equation*}
where $I(A,B)$ denotes the set of isomorphisms between two graphs $A$ and $B$.
We define the $1$-Wasserstein distance on $\P(\G_*[\C_T])$ using this (bounded) metric, namely
\begin{equation*}
W_1(P,Q) := \inf_{\pi}\int d_{*,1}\,d\pi,
\end{equation*}
where the infimum is over those $\pi \in \P(\G_*[\C_T] \times \G_*[\C_T])$ with marginals $P$ and $Q$.
Notice that the above construction of $X^n$ and $X$ produces a coupling of $\Lmc(G_n,X^{G_n,x^n}[T])$ and $\Lmc(G,X^{G,x}[T])$. This yields, for $n \ge N$,
\begin{align*}
& W_1(\Lmc(G_n,X^{G_n,x^n}[T]),\Lmc(G,X^{G,x}[T])) \\
& \quad \le \E\left[d_{*,1}((G_n,X^{G_n,x^n}),(G,X^{G,x}))\right]  \\
& \quad \le 2^{-\ell} + \E\left[\sum_{j=1}^\ell 2^{-j} \frac{1}{|B_j(G)|}\sum_{v \in B_j(G)}\|X^n_{\varphi_n(v)} - X_v\|_{*,T} \right],
\end{align*}
where the second inequality is obtained by bounding each of the terms $j > \ell$ in the summation representation for $d_{*,1}$ by $2^{-j}$, and in each of the terms $j \le \ell$  the infimum over $I(B_j(G),B_j(G_n))$ is estimated by using the particular isomorphism $\varphi_n|_{B_j(G)} \in I(B_j(G),B_j(G_n))$ constructed above, for $n \ge N$. We may bound this further using \eqref{pf:localcont-3} to get
\begin{align*}
& W_1(\Lmc(G_n,X^{G_n,x^n}[T]),\Lmc(G,X^{G,x}[T])) \\
& \quad \le 2^{-\ell} + (1-2^{-\ell}) \left( \max_{v \in B_{\ell}(G)}\E[\|X^n_{\varphi_n(v)}-X_v\|_{*,T}^2] \right)^{1/2} \\
& \quad \le 2^{-\ell} + (1-2^{-\ell})\left(\epsilon e^{C_1T} + 4M\frac{(C_1 T)^{k-\ell}}{(k-\ell)!}\right)^{1/2}.
\end{align*}
To summarize the logic, for each $\epsilon > 0$ and $\ell \in \N$ we have shown that for sufficiently large $k \in \N$ there exists $N \in \N$ such that the above inequality is valid for all $n \ge N$. Sending first $n\to\infty$, next $k\to\infty$,  then $\ell \to \infty$ and finally $\epsilon \to 0$ shows that $\lim_{n\to\infty}W_1(\Lmc(G_n,X^{G_n,x^n}[T]),\Lmc(G,X^{G,x}[T]))$ $= 0$, which completes the proof.
\end{proof}

\section{Correlation decay} \label{se:corrdecay}

Before studying the convergence of empirical measures announced in Sections \ref{se:statements:globalemp} and \ref{se:statements:componemp}, we first develop some basic ideas of correlation decay that will be essential in the proofs.
Our statements of correlation decay are given first for a non-random graph and initial position. This will adapt immediately to the case of a random graph and initial state, but the correlation will be conditional on the realization of the graph and initial state.

\subsection{Correlation decay in discrete time}

The discrete-time case is easy, simply because  two particles at vertices $u$ and $v$  influence each other within the first $k$ units of time only if their graph distance is smaller than $2k$.

\begin{lemma} 	\label{lem:corrdecay-disc}
Suppose Assumption \ref{assumption:A} holds. Let $G$ be a graph and $x=(x_v)_{v \in G} \in \X^G$.
Let $A_1$ and $A_2$ be subsets of $G$. Then, for $k \in \N_0$ and bounded measurable functions $f_i \colon (\X^{k+1})^{A_i} \to \Rmb$, $i=1,2$, we have
\[
\left|\Cov\big(f_1(X^{G,x}_{A_1}[k]),f_2(X^{G,x}_{A_2}[k])\big)\right| \le 2\ind_{\{2k \ge d_G(A_1,A_2)\}} \|f_1\|_\infty \|f_2\|_\infty.
\]
\end{lemma}
\begin{proof}
This bound is trivial for $2k \ge d_G(A_1,A_2)$. 
For $2k < d_G(A_1,A_2)$, 
it is clear from the dynamics \eqref{eq:random_graph-disc} that $X_{A_1}[k]$ and $X_{A_2}[k]$ are then independent.
\end{proof}

Suppose now that $(G,x)$ is a random element of $\G_*[\X]$, and recall from the definitions in Section \ref{se:disc-setup} that the conditional law of $(G,X^{G,x})$ given $(G,x)$ is $P^{G,x}$.
Under Assumption \ref{assumption:A}, if $A_1$ and $A_2$ are $(G,x)$-measurable random rooted subgraphs of $G$, then, for $k \in \N_0$ and bounded measurable functions $g_1,g_2 \colon \G_*[\X^{k+1}] \to \R$, we have
\begin{align}
\left|\Cov\big(g_1(A_1,X^{G,x}_{A_1}[k]),g_2(A_2,X^{G,x}_{A_2}[k]) \,|\, (G,x)\big)\right| \le 2\ind_{\{2k \ge d_G(A_1,A_2)\}} \|g_1\|_\infty \|g_2\|_\infty, \ \ a.s. \label{def:corrdecay-conditional-disc}
\end{align}

\subsection{Correlation decay for the diffusion system}
\label{subs-cordeccont}

Correlation decay is more complicated in the diffusive case, because the influence of a single particle propagates instantaneously throughout the graph. The Lipschitz Assumption \ref{assumption:C} enables a natural coupling argument, given in the following lemma. We denote by $\|f\|_{BL}$ the bounded Lipschitz metric of a real-valued function $f$ defined on a metric space $(\Polish,d)$, which is given by
\begin{align*}
\|f\|_{BL} = \sup_{x \in \Polish}|f(x)| + \sup_{\stackrel{x,y \in \Polish}{x \neq y}}\frac{|f(x)-f(y)|}{d(x,y)}.
\end{align*}
For any metric space $(\Polish,d)$ and any $k \in \N$, we implicitly endow the product space $\Polish^k$ with the $\ell_1$-metric $((x_1,\ldots,x_k),(y_1,\ldots,y_k)) \mapsto \sum_{i=1}^kd(x_i,y_i)$.

\begin{lemma} \label{lem:corrdecay-cont}
Suppose Assumption \ref{assumption:C} holds, and let $r>0$. For each $t \ge 0$, there exists a constant $\cc_t \in \R_+$ such that the following holds: Let $G$ be a graph and $x=(x_v)_{v \in G} \in (\R^d)^G$ with $\sup_{v \in G}|x_v| \le r$.
Let $A_1$ and $A_2$ be finite subsets of $G$. Then, for bounded Lipschitz functions $f_i \colon \C_t^{A_i} \to \R$, $i=1,2$, we have
\[
\left|\Cov\big(f_1(X^{G,x}_{A_1}[t]),f_2(X^{G,x}_{A_2}[t])\big)\right| \le \cc_t(|A_1| + |A_2|) \|f_1\|_{BL} \|f_2\|_{BL} \left(\frac{\cc_t^{\lceil d_G(A_1,A_2)/2 \rceil }}{\lceil d_G(A_1,A_2)/2\rceil !}\right)^{1/2},
\]
where we adopt the conventions that $\cc_t^\infty/\infty ! := 0$ and, as usual, $0!=1$.
\end{lemma}
\begin{proof}
  For ease of notation, we omit the $G$ and $x$ from the notation, writing $X=X^{G,x}$, $d=d_G$, and $N_v=N_v(G)$ for $v \in G$.
Also, fix   finite subsets $A_1$ and $A_2$ of $G$. 
The idea behind the proof is to couple $X$ with  two other processes $Y$ and $Z$, which are driven partially by  different collections of Brownian motions.

Recall that $X=X^{G,x}$ defined in \eqref{statements:SDE} is driven by Brownian motions $W = (W_v)_{v \in G}$. 
Let $\ti W := (\ti W_v)_{v \in G}$ be independent copies of $W$. 
Let $Y$  be another particle  system defined as in \eqref{statements:SDE},
but with $X$ replaced with $Y$ and  $W_v$ replaced with
$\ti W_v$ for $v$
such that $d(v,A_1) \ge d(v,A_2)$.  In a similar fashion, let
$Z$ be defined as in \eqref{statements:SDE}, but where $X$ is replaced with $Z$
and  $W_v$ is replaced with $\ti W_v$ for $v$
such that $d(v,A_1) < d(v,A_2)$. Precisely, $(X,Y,Z)$ are defined as the unique solutions of the following sets of equations:
\begin{align*}
X_v(t) & = x_v +  \int_0^t b(s,X_v,X_{N_v}) \, ds  +  \int_0^t \sigma(s,X_v,X_{N_v}) \, dW_v(s), \quad v \in G,
\end{align*}
as well as
\begin{align*}
	Y_v(t) & = x_v +  \int_0^t b(s,Y_v,Y_{N_v}) \, ds +  \int_0^t \sigma(s,Y_v,Y_{N_v}) \, d\ti W_v(s), \quad \mbox{ if } d(v,A_1) \ge d(v,A_2), \\
	Y_v(t) & = x_v +  \int_0^t b(s,Y_v,Y_{N_v}) \, ds  +  \int_0^t \sigma(s,Y_v,Y_{N_v}) \, dW_v(s), \quad \mbox{ if } d(v,A_1) < d(v,A_2),	\\ 
	Z_v(t) & = x_v +  \int_0^t b(s,Z_v,Z_{N_v}) \, ds  +  \int_0^t \sigma(s,Z_v,Z_{N_v}) \, dW_v(s), \quad \mbox{ if } d(v,A_1) \ge d(v,A_2), \\
	Z_v(t) & = x_v +  \int_0^t b(s,Z_v,Z_{N_v}) \, ds +  \int_0^t \sigma(s,Z_v,Z_{N_v}) \, d\ti W_v(s), \quad \mbox{ if } d(v,A_1) < d(v,A_2).
\end{align*}
Because the SDE \eqref{statements:SDE} is unique in law by Theorem \ref{th:uniqueness-hominfSDE}, each of $X$, $Y$ and $Z$ have the same law. Moreover, $Y$ is independent of $Z$ by construction. Therefore, for $f_1$ and $f_2$ as in the statement of the lemma,
\begin{align*}
		\Emb [f_1(X_{A_1}[t])] \Emb [f_2(X_{A_2}[t])] & = \Emb [f_1(Y_{A_1}[t])] \Emb [f_2(Z_{A_2}[t])] = \Emb [f_1(Y_{A_1}[t])f_2(Z_{A_2}[t])] 
\end{align*}
and hence, recalling that $\|x\|_{*,t} := \sup_{s \in [0,t]}|x(s)|$,
\begin{align}
		& |\Cov(f_1(X_{A_1}[t]),f_2(X_{A_2}[t]))| \notag \\
		&\quad  = |\Emb[f_1(X_{A_1}[t])f_2(X_{A_2}[t]) - f_1(Y_{A_1}[t])f_2(Z_{A_2}[t])]| \notag \\
		&\quad  \le \|f_1\|_{BL} \|f_2\|_{BL} \Emb \left[\|X_{A_1}-Y_{A_1}\|_{*,t} \right]+ \|f_1\|_{BL} \|f_2\|_{BL} \Emb \left[\|X_{A_2}-Z_{A_2}\|_{*,t}\right], \label{eq:Diff_Cov}
\end{align}
	
In what follows,  $\cc_i,$ $i=1,2,3,$ represent suitably chosen
          finite constants (depending only on $t$ and $K_t$ from Assumption \ref{assumption:C} but not on the underlying graph), which
        we do not identify explicitly. 
        For each $0 \le k \le \lceil\frac{d(A_1,A_2)}{2}\rceil-1$ and $v \in G$ such that
        $d(v,A_1) \le k$, we have $d(v,A_1) < d(v,A_2)$. 
	It then follows from the evolution of $X$ and $Y$, Jensen's inequality, Doob's inequality, and the Lipschitz condition in Assumption \ref{assumption:C} that
	\begin{align*}
	  \max_{v: \, d(v,A_1) \le k} \Emb \left[\|X_v-Y_v\|_{*,t}^2 \right]& \le \max_{v: \, d(v,A_1) \le k}
          2t \int_0^t \Emb \left[|b(s,X_v,X_{N_v})-b(s,Y_v,Y_{N_v})|^2\right] \, ds \\
		& \quad + \max_{v: \, d(v,A_1) \le k} 8 \int_0^t \Emb \left[|\sigma(s,X_v,X_{N_v})-\sigma(s,Y_v,Y_{N_v})|^2\right] \, ds \\
		& \le \cc_1 \int_0^t \max_{v: \, d(v,A_1) \le k+1} \Emb \left[\|X_v-Y_v\|_{*,s}^2 \right]\, ds.
	\end{align*}
Since $\sup_{v \in G} |x_v| \leq r$, 
recall from Lemma \ref{le:momentbound-tightness} that $\sup_{v \in G} \E \left[\|X_v\|_{*,t}^2 \right] \le \cc_2$.
Since $X$ and $Y$ have the same law, this implies $\sup_{v \in G} \E \left[ \|Y_v\|_{*,t}^2\right] \le \cc_2$. 
These two bounds imply
	\begin{align*}
		\max_{v: \, d(v,A_1) \le \lceil\frac{d(A_1,A_2)}{2}\rceil} \Emb \left[\|X_v-Y_v\|_{*,t}^2 \right] & \le \cc_3 :=4\cc_2. 
	\end{align*}
	From the last two displays, we can recursively get for each $k = \lceil\frac{d(A_1,A_2)}{2}\rceil-1, \lceil\frac{d(A_1,A_2)}{2}\rceil-2,\dotsc,0$,
	\begin{equation*}
		\max_{v: \, d(v,A_1) \le k} \Emb \left[\|X_v-Y_v\|_{*,t}^2\right] \le \cc_3 \frac{(\cc_1 t)^{\lceil\frac{d(A_1,A_2)}{2}\rceil-k}}{(\lceil\frac{d(A_1,A_2)}{2}\rceil-k)!}.
	\end{equation*}
In particular,  when $k=0$ we have
\begin{equation*}
\max_{v \in A_1} \Emb \left[\|X_v-Y_v\|_{*,t}^2 \right] \le \cc_3 \frac{(\cc_1 t)^{\lceil\frac{d(A_1,A_2)}{2}\rceil}}{\lceil\frac{d(A_1,A_2)}{2}\rceil!}, 
\end{equation*}
and hence, recalling that we work with the $\ell_1$ distance on the product space $\C_t^{A_1}$, 
\begin{equation*}
\Emb \left[\|X_{A_1}-Y_{A_1}\|_{*,t} \right]\le |A_1| \sqrt{\max_{v \in A_1} \E [\|X_v-Y_v\|_{*,t}^2]} \le |A_1| \left(\cc_3 \frac{(\cc_1 t)^{\lceil\frac{d(A_1,A_2)}{2}\rceil}}{\lceil\frac{d(A_1,A_2)}{2}\rceil!}\right)^{1/2}.
\end{equation*}
Similarly, we can obtain the bound 
\begin{equation*}
\Emb \left[\|X_{A_2}-Z_{A_2}\|_{*,t} \right] \le |A_2| \left(\cc_3 \frac{(\cc_1 t)^{\lceil\frac{d(A_1,A_2)}{2}\rceil}}{\lceil\frac{d(A_1,A_2)}{2}\rceil!}\right)^{1/2}.
\end{equation*}
Combining these two displays with \eqref{eq:Diff_Cov} yields the desired result with $c_t := \max( \sqrt{c_3}, c_1t)$. 
\end{proof}

Let us lastly state the analogue of \eqref{def:corrdecay-conditional-disc} in the continuous-time setting. Let $(\Polish,d)$ be a metric space, and let $f \colon \G_*[\Polish] \to \R$ be Lipschitz with respect to the metric $d_*$ introduced in \eqref{def:G* marked metric}. Then, for any finite graph $G \in \G_*$, we have
\begin{align*}
d_*((G,x),(G,y)) &\le \max_{v \in G}d(x_v,y_v) \le \sum_{v \in G}d(x_v,y_v),
\end{align*}
for any $x,y \in \Polish^G$, and we may thus view $x \mapsto f(G,x)$ as a Lipschitz function on $\Polish^G$ with Lipschitz constant no greater than that of $f$ itself (recalling from above that we equip $\Polish^G$ with the $\ell_1$ metric).
This allows us to deduce the continuous-time analogue of \eqref{def:corrdecay-conditional-disc} from Lemma \ref{lem:corrdecay-cont}, which we state as follows:
Suppose now that $(G,x)$ is a random element of $\G_*[B_r(\R^d)]$ for some $r > 0$, with $G$ almost surely finite.
Under Assumption \ref{assumption:C}, if $A_1$ and $A_2$ are $(G,x)$-measurable random rooted subgraphs of $G$, then, for $t \ge 0$ and bounded Lipschitz functions $g_1,g_2 \colon \G_*[\C_t] \to \R$, we have
\begin{align}
&\left|\Cov\big(g_1(A_1,X^{G,x}_{A_1}[t]),g_2(A_2,X^{G,x}_{A_2}[t]) \,|\, (G,x)\big)\right| \nonumber \\
	&\qquad \le \cc_t(|A_1| + |A_2|)\|g_1\|_{BL} \|g_2\|_{BL} \left(\frac{\cc_t^{\lceil  d_G(A_1,A_2)/2 \rceil}}{\lceil d_G(A_1,A_2)/2  \rceil!}\right)^{1/2}, \ \ a.s., \label{def:corrdecay-conditional-cont}
\end{align}
with the same constants $(\cc_t)_{t \ge 0}$ as in Lemma \ref{lem:corrdecay-cont}.

\section{Convergence of empirical measures} \label{se:emp-main}

In this section we will introduce unified notation to simultaneously analyze both discrete and diffusive dynamics, since the remainder of the arguments are essentially
the same in each case. 
  With $r \in (0,\infty)$  as in the statement of Theorem \ref{thm:local-law-cont}, we define
\begin{alignat*}{5}
\base & = \X, \quad  && \pathspacek= \X^{k+1}, \quad && \pathspace= \X^\infty, \quad && k \in \indexset=\N_0, \qquad \text{in the discrete case, or} \\
\base & = B_r(\R^d), \quad  && \pathspacek= \C_k, \quad && \pathspace = \C, \quad && k \in \indexset=\R_+, \qquad \text{in the continuous case},
\end{alignat*}
where we recall that $\C_k :=C([0,k];\R^d)$ and $\C :=C(\R_+;\R^d)$.

    We  assume throughout that   Assumption \ref{assumption:A} (in the discrete case)
    or  Assumption \ref{assumption:A}' (in the diffusive setting) 
    holds.
    For any $(G,x) \in \G_*[\base]$,
    let  $(G,X^{G,x})$ be the $\G_*[\pathspace]$-valued
    random element, where $X^{G,x}$ satisfies the  dynamics  described by 
    \eqref{eq:random_graph-disc} or \eqref{statements:SDE}, respectively.
    Let the empirical measure $\mu^{G,x}$ be the corresponding   $\P(\pathspace)$-valued random element defined in \eqref{def-empmeasG}.

Fix $T \in (0,\infty)$.  
In both cases, we have seen from \eqref{def:corrdecay-conditional-disc} and \eqref{def:corrdecay-conditional-cont} that there exists
for each $k \in \indexset$, a function $\decayfn_k : \N \to \R_+$ that depends on $T$, such that $\decayfn_k(\infty) := \lim_{n\to\infty}\decayfn_k(n)=0$ and
for all $k \in [0,T] \cap \indexset$, 
\begin{align}
|\Cov(f_1(A_1,X^{G,x}_{A_1}[k]), f_2(A_2,X^{G,x}_{A_2}[k]) \, | \, (G,x))| \le (|A_1|+|A_2|)\decayfn_k(d_G(A_1,A_2)), \ \ a.s., \label{def:corrdecay-unified}
\end{align}
for any random element $(G,x)$ of $\G_*[\base]$ with $G$ a.s.\ finite, any finite subsets $A_1,A_2\subset G$, and any bounded Lipschitz functions $f_1,f_2 : \G_*[\pathspacek] \to \R$ with $\|f_1\|_{BL},\|f_2\|_{BL}\le 1$. (In the discrete case, $f_1$ and $f_2$ need not be Lipschitz, but we will not need this generality.)

Our first result considers  the easy case of a finite limiting graph, for which local convergence
implies empirical measure convergence. The following  is an immediate consequence of
Theorem \ref{thm:local-law-disc}, Theorem \ref{thm:local-law-cont},
and Proposition \ref{pr:localconvergence-empiricalmeasure}.

\begin{proposition}[Finite graph case] \label{pr:empirical-finite}
  Suppose $(G,x)$ and $(G_n,x^n)$ are 
  random marked graphs  with $(G_n,x^n) \rightarrow (G,x)$ in law in $\G_*[\base]$. 
  If $G$ and $G_n$ are a.s.\ finite, then $\mu^{G_n,x^n}$ converges in law to $\mu^{G,x}$ in $\P(\pathspace)$.  
\end{proposition}

When the limiting graph $G$ is infinite, $\mu^{G,x}$ is not well defined, and  
  the problem becomes more challenging.
In this case, the correlation decay property \eqref{def:corrdecay-unified} plays a decisive role, allowing us to derive a quenched  asymptotic independence property which ultimately implies that the empirical measure concentrates around its mean.

\subsection{Empirical measure convergence for $G_n$} \label{se:pf:emp-global}

\begin{proof}[Proof of Theorems \ref{thm:localprob-disc} and \ref{thm:localprob-cont}]
It suffices to  prove   convergence in probability in the local weak sense of $(G_n,X^{G_n,x^n})$ to $(G,X^{G,x})$, since then
  convergence of the empirical measure sequence $\mu^{G_n,x^n}$ to $\Lmc(X^{G,x}_{\oSlash})$  follows immediately from Lemma \ref{le:empmeas general}.

Fix $k \in \indexset$.
Given $G_n$, let $U^n_1$ and $U^n_2$ be independent random elements, uniformly distributed on  the vertex set of $G_n$. Abbreviate $\comp^n_i := \comp_{U^n_i}(G_n,X^{G_n,x^n}[k])$ for $i=1,2$.
By Lemma \ref{le:localweakprob characterization}, it suffices to show that
\begin{align}
\lim_{n\to\infty}\E[f_1(\comp^n_1)f_2(\comp^n_2)] &= \E[f_1(G,X^{G,x}[k])]\E[f_2(G,X^{G,x}[k])], \quad \forall f_1,f_2 \in C_b(\G_*[\pathspacek]). \label{pf:global1}
\end{align}
By a standard approximation argument, we may 
 assume $f_i$ is bounded and Lipschitz with $\|f_i\|_{BL} \le 1$, for $i=1,2$.  
 Moreover, if $(H,y) \in \G_*[\pathspacek]$,  then we have the simple estimate $d_*(B_\ell(H,y),(H,y)) \le 2^{-\ell}$, where recall that $d_*$ is the
   metric introduced in \eqref{def:G* marked metric}. It thus suffices to prove
that 
 \begin{align}
\lim_{n\to\infty}\E[f_1(B_\ell(\comp^n_1))f_2(B_\ell(\comp^n_2))] &= \E[f_1(B_\ell(G,X^{G,x}[k]))]\E[f_2(B_\ell(G,X^{G,x}[k]))]. \label{pf:global2}
\end{align}

 Convergence in probability in the local weak sense of $G_n$ to $G$ 
 is known to imply that $d_{G_n}(U^n_1,U^n_2) \to \infty$ in probability \cite[Corollary 2.13]{van2020randomII}.  
 Further, it also implies that $\{|B_\ell(\comp_{U_i^n}(G_n))| : n \in \Nmb, i=1,2\}$ is stochastically bounded, because $|B_\ell(\comp_{U_i^n}(G_n))| \to |B_\ell(G)|$ in distribution as $n \to \infty$ and $|B_\ell(G)| < \infty$ almost surely.
 Hence, the correlation decay estimate \eqref{def:corrdecay-unified}, with $A_i := B_\ell(\comp^n_i)$, $i = 1, 2$,
 implies that 
\[
\Cov\big(f_1(B_\ell(\comp^n_1)), \, f_2(B_\ell(\comp^n_2))\,|\,(G_n,x^n)\big) \to 0
\]
in probability. 
 Since the sequence
 $|\Cov\big(f_1(B_\ell(\comp^n_1)), \, f_2(B_\ell(\comp^n_2))\,|\,(G_n,x^n)\big)|, n \in \N,$ is also uniformly
 bounded,   this implies 
\begin{align*}
& \lim_{n\to\infty} \E[f_1(B_\ell(\comp^n_1))f_2(B_\ell(\comp^n_2))] \\
& \quad = \lim_{n\to\infty} \E\Big[\E\big[f_1(B_\ell(\comp^n_1)) \,|\, (G_n,x^n)\big] \E\big[f_2(B_\ell(\comp^n_2)) \,|\, (G_n,x^n)\big]\Big].
\end{align*}
Now recall the notation $P^{H,y} := \L(H,X^{H,y}) \in \P(\G_*[\pathspace])$ for $(H,y) \in \G_*[\base]$. Define $g_i \in C_b(\G_*[\pathspace])$ by $g_i(H,\bar{y}) := f_i(B_\ell(H,\bar{y}[k]))$ for $i=1,2$
 and $(H,\bar{y}) \in \G_*[\pathspace]$. 
We then have  
\begin{align*}
\E\big[f_i(B_\ell(\comp^n_i)) \,|\, (G_n,x^n)\big] = \lan P^{\comp_{U^n_i}(G_n,x^n)}, g_i\ran, \quad i=1,2,
\end{align*}
and the preceding equation can be rewritten  as
\begin{align*}
\lim_{n\to\infty} \E[f_1(B_\ell(\comp^n_1))f_2(B_\ell(\comp^n_2))] &= \lim_{n\to\infty} \E\Big[\lan P^{\comp_{U^n_1}(G_n,x^n)}, g_1\ran \lan P^{\comp_{U^n_2}(G_n,x^n)}, g_2\ran \Big].
\end{align*}
From Theorem \ref{thm:local-law-disc} (in the discrete case)  and Theorem \ref{thm:local-law-cont} (in the diffusive case, under the additional assumption
that $x_v^n \in B_r(\R^d)$ a.s. for each $v \in G_n$ and $n \in \N$) we know that the map $h_i: \G_*[\base] \mapsto \R$ that takes 
$(H,y) \mapsto  \lan P^{H,y},g_i\ran$ is continuous. 
Thus, we may use the convergence in probability in the local weak sense of $(G_n,x^n)$ to $(G,x)$, in its equivalent form given in Lemma \ref{le:localweakprob characterization}, to deduce that
\begin{align*}
  \lim_{n\to\infty} \E[f_1(B_\ell(\comp^n_1))f_2(B_\ell(\comp^n_2))]
  &= \lim_{n \to \infty} \E[ h_1(\comp_{U_1^n}(G_n,x^n))] \E[h_2(\comp_{U_2^n}(G_n,x^n))] \\
  & = \E[ h_1(G,x)]\E[ h_2(G,x)] \\
	&= \E\left[f_1(B_\ell(G,X^{G,x}[k])) \right] \E\left[ f_2(B_\ell(G,X^{G,x}[k])) \right].
\end{align*}
This establishes \eqref{pf:global2}, thus completing the proof.
\end{proof}

\subsection{Empirical measure convergence for $\compu (G_n)$} \label{se:componemp}

We now turn to the class of examples where the empirical measure
sequence can converge to a stochastic limit.
We first state and prove a result on  
global empirical measure convergence for a general class of graph sequences  that satisfy
the  properties stated in Condition \ref{cond-graphs} below,  and
then deduce Theorem \ref{thm:CompEmpMeas}  by verifying these conditions for the 
graph sequences considered therein.  We continue to use the unified notation introduced at the
beginning of   Section \ref{se:emp-main}.

In what follows, given any graph $H$ with vertex set contained in $\N$, let $\compmax(H)$ denote the largest connected component of $H$; if there is a tie, we choose the component with the smallest minimal vertex $v \in \N$.
This tie-break is purely for convenience, as the following condition requires in part that there is no tie, with high probability, for the graphs that we consider.

\begin{Condition} \label{cond-graphs}
The random graphs $G$ and $\{G_n\}$ satisfy the following properties.
Write $s_G := \PP(|G| = \infty)$. 
\begin{enumerate} 
\item If $s_G > 0$, then for every $\epsilon > 0$, 
$\lim_{n \rightarrow \infty}  \PP \left( \left|\frac{\left| \compmax(G_n) \right|}{s_{G} n}  - 1\right| > \epsilon  \right) = 0$;
\item If $0 < s_G < 1$, then $G_n \setminus \compmax(G_n)$ converges in probability in the local weak sense to the random graph with law $\mathcal{L}(G \,|\, |G|<\infty)$.
\end{enumerate} 
\end{Condition}

The following result shows that this condition is satisfied by the graph sequences
considered in Theorem \ref{thm:CompEmpMeas}.   Its proof builds on many well known properties of random graphs,      and is thus relegated to Appendix \ref{ap:condgraphs}. 

      \begin{proposition}
        \label{prop-graphs}
        Suppose the random graphs $G_n, n \in \N,$ and $G$ are as described in either (i) or (ii) of Theorem \ref{thm:CompEmpMeas}. Then they satisfy Condition \ref{cond-graphs}.
      \end{proposition}

We now state our generalization of Theorem \ref{thm:CompEmpMeas}.
Using  the common notation introduced, we proceed with a  unified treatment of discrete and diffusive dynamics.

\begin{theorem}  \label{thm:GenCompEmpMeas}
  Suppose $\{G_n\}_{n \in \N}$ and $G$ satisfy Condition \ref{cond-graphs}, $G_n$ converges in probability in the local weak sense to $G$, and
  the initial conditions $(x^n)$ and $x$  are such that Assumption \ref{assumption:B} is satisfied. 
 Further, 
  suppose Assumption \ref{assumption:A} (resp.\ Assumption \ref{assumption:A}')
  holds and  let $\mu^{\compu(G_n,x^n)}$ be the  connected component empirical measure
    defined in \eqref{def-empmeas}. 
    Then   $\mu^{\compu(G_n,x^n)}$ converges in law in  $\P(\pathspace)$ to the    random measure $\tilde{\mu}^{G,x}$, where
  \[  \tilde{\mu}^{G,x} :=
  \left\{
  \begin{array}{ll}
    \mu^{G,x}   & \mbox{ on } \{|G| < \infty\}, \\
       \Lmc \left( X_{\oSlash}^{G,x} \,\big|\, |G| = \infty \right)  & \mbox{ on } \{|G|= \infty\}.
  \end{array}
  \right. 
  \]
\end{theorem}
\begin{proof}[Proof of Theorem \ref{thm:GenCompEmpMeas}.] 
Recall $s_G := \PP(|G| = \infty)$. 
First, if $s_G=0$, then the claim follows from Proposition \ref{pr:empirical-finite}.
We thus focus on the case when $s_G \in  (0,1]$. 

Recall that $\compu(G_n) = \comp_{U_n}(G_n)$, where $U_n$ is a uniformly random vertex of $G_n$. 
Let $S_n  := \{ U_n \in \compmax(G_n) \}$ 
   denote the event that  $\compu(G_n)$ is the (rooted) maximal component of $G_n$. 
By Condition \ref{cond-graphs}(1),
   with high probability the giant component $\compmax(G_n)$   of $G_n$ has size $s_G n$, and we deduce from dominated convergence that
   \begin{equation}
     \label{snlimit}
     \lim_{n \rightarrow \infty} \PP(S_n) = \lim_{n \to \infty} \frac{1}{n} \mathbb{E}[|\compmax(G_n)|] = s_G.
   \end{equation}
Further, for any bounded continuous function $F \colon \Pmc(\pathspace) \to \R$ we have
\begin{align*}
\E[F(\mu^{\compu(G_n,x^n)})] &= \E[F(\mu^{\compu(G_n,x^n)}) \, | \, S_n]\Pmb(S_n) + \E[F(\mu^{\compu(G_n,x^n)}) \, | \, S_n^c]\Pmb(S_n^c).
\end{align*}
It then suffices to show that 
\begin{align}
\lim_{n\to\infty}\E[F(\mu^{\compu(G_n,x^n)}) \, | \, S_n^c] &= \E[F(\mu^{G,x}) \, | \, |G| < \infty], \quad \text{if } s_G<1, \label{eq:empirical-measure-U(G_n)-claim-disc} \\
\lim_{n\to\infty} \E[F(\mu^{\compu(G_n,x^n)}) \, | \, S_n] &= F\left(\Lmc(X^{G,x}_{\oSlash} \, | \, |G|=\infty )\right). \label{pf:ER-U-empiricallimit1-disc}
\end{align}
Indeed, these claims together with \eqref{snlimit} imply that
\begin{align*}
\lim_{n\to\infty}\E[F(\mu^{\compu(G_n,x^n)})]  &= s_{G}F\big(\Lmc(X^{G,x}_{\oSlash} \, | \, |G|=\infty )\big) + (1-s_{G})\E[F(\mu^{G,x}) \, | \, |G| < \infty],
\end{align*}
which completes the proof. 
We now present the  proofs of these two claims.

{\ }

\noindent{\bf Step 1.} We first prove \eqref{eq:empirical-measure-U(G_n)-claim-disc}.
 We may assume $s_G < 1$, and thus $\PP(S_n^c) > 0$ for sufficiently large $n$ by \eqref{snlimit}. 
Define marked random graphs
$(\widetilde{G}_n,\tilde{x}^n) = (G_n,x^n) \setminus \compmax(G_n,x^n)$ for $n \in \N$, and define $(\widetilde{G},\tilde{x})$ such that $\L(\widetilde{G},\tilde{x}) = \L((G,x)\,|\, |G|< \infty)$.
Then Condition \ref{cond-graphs}(2) implies that
$\widetilde{G}_n$ converges in probability in the local weak sense to $\widetilde{G}$. Hence, by Assumption \ref{assumption:B}, $(\widetilde{G}_n,\tilde{x}^n)$ converges in probability in the local weak sense to to a graph that has the same law as $(\widetilde{G},\tilde{x})$.  That is, for $g \in C_b(\G_*[\base])$,
\begin{align*}
\frac{1}{|G_n \setminus \compmax(G_n)|}\sum_{v \in G_n \setminus \compmax(G_n)} g(\comp_v(G_n,x^n)) \to \E[g(G,x) \,|\, |G| < \infty], \quad \text{in probability.} 
\end{align*}
Note that $\L(U_n\,|\,(G_n,x^n),S_n^c)$ is uniform on $G_n \setminus \compmax(G_n)$.
Therefore, for $g \in C_b(\G_*[\base])$, 
\begin{align}
  \E[g(\compu(G_n,x^n)) \, | \, S_n^c] &= \E\left[ \frac{1}{|G_n \setminus \compmax(G_n)|}\sum_{v \in G_n \setminus \compmax(G_n)} g(\comp_v(G_n,x^n)) \, | \, S_n^c \right] \nonumber \\
	&\to \E[g(G,x)\,|\,|G| < \infty]. \label{pf:G_n-complim}
\end{align}
Here we used the following elementary fact: If $\{A_n\}$ are events with $\liminf_{n\to\infty}\PP(A_n) > 0$, and if $\{Z_n\}$ is a uniformly bounded sequence of real random variables converging in probability to a constant $c$, then $\E[ Z_n | A_n] - c = \E[( Z_n -c)\ind_{A_n}]/\PP(A_n) \to 0$.
The limit \eqref{pf:G_n-complim}, valid for all $g \in C_b(\G_*[\base])$, means that
$\L(\compu(G_n,x^n) \, | \, S_n^c) \to \L((G,x) \, | \, |G| < \infty)$ in $\P(\G_*[\base])$. Because the limiting law is supported on finite graphs, 
it follows from Proposition \ref{pr:empirical-finite} that $\L(\mu^{\compu(G_n,x^n)} \, | \, S_n^c)$ converges to $\L(\mu^{G,x} \, | \, |G| < \infty)$ in $\P(\P(\pathspace))$, and 
\eqref{eq:empirical-measure-U(G_n)-claim-disc} follows.

{\ }

\noindent{\bf Step 2.} We now prove \eqref{pf:ER-U-empiricallimit1-disc}.
 We first claim that the sequence of marked random graphs $\compmax(G_n,x^n)$ converges in probability in the local weak sense to the marked random graph with law $\mathcal{L}((G,x) \,|\, |G|=\infty)$.
 By Assumption \ref{assumption:B}, 
 it suffices to show that $\compmax(G_n)$ converges in probability in the local weak sense to the random graph with law $\mathcal{L}(G \,|\, |G|=\infty)$.
For this, note that if $s_G=1$, then such a convergence follows immediately from that of $G_n$ and Condition \ref{cond-graphs}(1). 
Next we verify the claim for $0<s_G<1$.
Note that we can write
\begin{align*}
\frac{1}{|G_n|} \sum_{v \in G_n} \delta_{\comp_v(G_n)} & = \frac{|G_n \setminus \compmax(G_n)|}{|G_n|} \frac{1}{|G_n \setminus \compmax(G_n)|} \sum_{v \in G_n \setminus \compmax(G_n)} \delta_{\comp_v(G_n)} \\
    	& \quad + \frac{|\compmax(G_n)|}{|G_n|} \frac{1}{|\compmax(G_n)|} \sum_{v \in \compmax(G_n)} \delta_{\comp_v(G_n)}.
\end{align*}
By the assumption that $G_n$ converges to $G$ in probability in the local weak sense, and the equivalent form given in Remark \ref{re:lwp-P(G)form},  we have the following two limits in probability, in $\P(\G_*)$:
\begin{equation*}
\frac{1}{|G_n|} \sum_{v \in G_n} \delta_{\comp_v(G_n)} \to \L(G), \:\: \text{and} \:\: \frac{1}{|G_n \setminus \compmax(G_n)|} \sum_{v \in G_n \setminus \compmax(G_n)} \delta_{\comp_v(G_n)} \to \L(G \,|\, |G|<\infty), 
\end{equation*}
where the second limit also uses Condition \ref{cond-graphs}(2).
Moreover, by Condition \ref{cond-graphs}(1) we also have the convergence in probability
\begin{align*}
\frac{|\compmax(G_n)|}{|G_n|} \to s_G, \quad \text{and} \quad \frac{|G_n \setminus \compmax(G_n)|}{|G_n|} \to 1-s_G.
\end{align*}
Since also
\begin{align*}
\L(G) = (1-s_G) \L(G \,|\, |G|<\infty) + s_G \L(G \,|\, |G|=\infty),
\end{align*}
we deduce that
 \begin{equation*}
	\frac{1}{|\compmax(G_n)|} \sum_{v \in \compmax(G_n)} \delta_{\comp_v(G_n)} \to \L(G \,|\, |G|=\infty)
\end{equation*}
in probability.
This proves that $\compmax(G_n)$ converges in probability in the local weak sense to the random graph with law $\mathcal{L}(G \,|\, |G|=\infty)$
when $0<s_G<1$, and hence, the claim on the convergence of the corresponding marked random graph sequence also holds for all $s_G \in (0,1]$. 

  Noting that $\L(U_n\,|\,(G_n, x^n),S_n)$ is the uniform distribution on $\compmax(G_n)$, we deduce that the sequence of marked random graphs with laws
  $\L(\compu(G_n,x^n)\,|\,S_n)$ also converges in probability in the local weak sense to the random graph with law $\mathcal{L}((G,x) \,|\, |G|=\infty)$.
Then, \eqref{pf:ER-U-empiricallimit1-disc} follows from Theorems \ref{thm:localprob-disc} (in the discrete case) and \ref{thm:localprob-cont} (in the diffusive setting). 
\end{proof}

\section{Empirical field convergence on some non-random graphs}
\label{se:nonrandom}

In this section we discuss various examples of dynamics on convergent sequences of
non-random graphs, as summarized in Section \ref{se:regtree}.
We continue to unify  the treatment of discrete and diffusive
dynamics, using the notation introduced at the beginning of Section \ref{se:emp-main}.
We assume throughout this section that Assumption \ref{assumption:A} holds in the discrete case and
Assumption \ref{assumption:C} holds in the diffusive setting.

We begin with some further notation.
For a graph $G=(V,E)$, let $\mathrm{Aut}(G)$ denote the set of automorphisms, i.e., bijections from $V$ to $V$ which preserve edges in the sense that $(u,v) \in E \Leftrightarrow (\varphi u,\varphi v) \in E$.
The group of automorphisms (with composition as the group operation) acts on the configuration space $S^G$, for any set $S$, in the natural way:
For $s=(s_v)_{v \in G} \in S^G$ and $\varphi \in \mathrm{Aut}(G)$, we define
\[
\varphi s := (s_{\varphi v})_{v \in G}.
\]
A crucial point is the $\mathrm{Aut}(G)$-invariance of the dynamics: 
\begin{align}
X^{G,\varphi x} \stackrel{d}{=} \varphi X^{G,x}, \qquad \forall x \in \X^G, \ \varphi \in \mathrm{Aut}(G). \label{eq:autoidentity}
\end{align}
Indeed, this follows from the well-posedness of the dynamics of the process $X^{G,x}$ for each $(G,x)$.
  Our results below pertain to  empirical fields, defined as follows.   
  Given  a set $S$,  a finite graph $G$,  
  a configuration $s \in S^G$, and a family of automorphisms
  $\varphi_v \in \mathrm{Aut}(G)$,  $v  \in G$, 
  the associated    empirical field  is given by 
  \[   \frac{1}{|G|}\sum_{v \in G}\delta_{\varphi_v s}. 
  \]
  A  natural special case is when the graph $G$  is the $d$-dimensional torus and $\varphi_v(\cdot) = \cdot - v$.

\subsection{A sufficient condition for empirical field convergence}

We first present a general principle, which states  that if the
average distances in a sequence of graphs $G_n$ converges to zero,
 then the random \emph{empirical fields} of  marked random graphs exhibiting correlation decay will concentrate around their means.
Recall in the following that $d_G$ denotes the graph distance on a graph $G$.

\begin{proposition} \label{pr:corrdecay-general}
  Suppose $G \in \G_*$ is infinite 
  and let $\{A_n\} \subset G$ be a sequence of finite  subsets of the vertex set of $G$ with
  $|A_n| \to \infty$  as $n \to \infty$. 
  Suppose for every $\ell \in \N$, there exists
    $\decayfn_\ell: \N \to \R_+$ that satisfies $\lim_{n\to\infty}\decayfn_\ell(n) = 0$ as well as 
\begin{equation}
  \lim_{n\rightarrow\infty} \frac{1}{|A_n|^2}\sum_{u,v \in A_n}\decayfn_\ell(d_G(u,v)) = 0.
  \label{eq:avggrowthcondition-disc}
\end{equation}
Also, let $\newPolish$ be a complete separable metric space,
and suppose $\newRV$ is a random element of $\newPolish^G$   
that satisfies the following decay of correlation: for every $\ell \in \N$, $A,B \subset G$ with $|A|, |B| \le \ell$, and
all $f \in C_b(\newPolish^A)$ and
$g \in C_b(\newPolish^B)$ with $\|f\|_{BL},\|g\|_{BL} \le 1$,
\begin{equation} 
	|\Cov(f(\newRV_A),g(\newRV_B))| \le \decayfn_\ell(d_{G}(A,B)).  
	\label{asmp:cordecay}
\end{equation}
In addition, suppose for each $v \in G$ we are given $\varphi_v \in \mathrm{Aut}(G)$ satisfying $\varphi_v\oSlash = v$, and let  $\mu_n := \frac{1}{|A_n|}\sum_{v \in A_n}\delta_{\varphi_v \newRV}$.
Then for each $f \in C_b(\newPolish^G)$, we have
\begin{align}
  \label{result-erg1}
\lim_{n\to\infty} \big| \lan \mu_n, f\ran - \E[\lan \mu_n, f\ran ]\big| = 0, \ \ \  \text{ in probability}. 
\end{align}
\end{proposition}
\begin{proof}
To show \eqref{result-erg1}, it suffices to show that for each $r \in \N$ and bounded Lipschitz $f \in C_b(\newPolish^{B_r(G)})$, as $n \rightarrow \infty$
\begin{align} \label{toshow1} 
\frac{1}{|A_n|}\sum_{v \in A_n} \big(f(\varphi_v \newRV_{B_r(G)}) - \E[f(\varphi_v \newRV_{B_r(G)})] \big) \rightarrow  0, \ \ \  \text{ in probability}. 
\end{align}
Now, use \eqref{asmp:cordecay} to see that with for $\ell := |B_r(G)|$ 
\begin{align*}
& \E\left[ \left(\frac{1}{|A_n|}\sum_{v \in A_n} \big(f(\varphi_v \newRV_{B_r(G)}) - \E[f(\varphi_v \newRV_{B_r(G)})] \big)\right)^2 \right] \\
	& \quad = \frac{1}{|A_n|^2} \sum_{v,u \in A_n} \Cov(f(\newRV_{\varphi_vB_r(G)}),f(\newRV_{\varphi_uB_r(G)})) \\
& \quad \le \|f\|_{BL}^2 \frac{1}{|A_n|^2} \sum_{v,u \in A_n}
\decayfn_\ell((d_G(u,v)-2r)^+),
\end{align*} 
where the inequality uses the assumption that $\varphi_v\oSlash=v$ to deduce that $\varphi_vB_r(G)$ is precisely the ball of radius $r$ centered around $v$, for each $v \in G$. The claim
\eqref{toshow1} now follows from Markov's inequality and \eqref{eq:avggrowthcondition-disc}. 
\end{proof}

\begin{remark}  \label{rem-ergass} 
  It is easy to see that if $G$ has bounded degree,  then  \eqref{eq:avggrowthcondition-disc} holds automatically for any $c_\ell$ with
  $\lim_{n \rightarrow \infty} c_\ell(n) = 0$,  as long as $|A_n| \to \infty$.  
Indeed, in this case for any $m \in \N$ we have
$a_m := \sup_{n \in \N} \sup_{v \in G_n} |\{u \in G_n : d_G(u,v) \le m\}| < \infty.$ 
   Hence, it follows that for any $\ell, m \in \N$,
\begin{align*}
\frac{1}{|G_n|^2} \sum_{v,u \in G_n} \decayfn_\ell(d_{G}(u,v)) &\le \sup_{k > m}\decayfn_\ell(k) + \frac{1}{|G_n|^2}\sum_{u,v \in G_n : d_G(u,v) \le m} \|\decayfn_\ell\|_\infty \\
	&\le \sup_{k > m}\decayfn_\ell(k) + \frac{a_m\|\decayfn_\ell\|_\infty}{|G_n|}.
\end{align*}
The  claim now follows on sending first $n\to\infty$ and then $m\to\infty$.
  Without a bounded degree assumption, the condition \eqref{eq:avggrowthcondition-disc} requires that correlations decay quickly enough to overcome the growth rate of the graph. 
\end{remark}

\subsection{Propagation of empirical field convergence} \label{se:empfield}

We next state an abstract result pertaining to the time-propagation of the property of convergence of random empirical fields, which is interesting in its own right and will serve us well in the more concrete examples of Section \ref{se:lattices + trees}.

\begin{proposition} \label{pr:empfield}
Let $G=(V,E,\oSlash)$ be a rooted, countable, locally finite graph. For each $v \in G$, suppose there exists $\varphi_v \in \mathrm{Aut}(G)$ such that $\varphi_v\oSlash=v$. Let $A_n \subset V,$
    $n \in \N$, be
  a sequence of finite subsets that satisfies $|A_n| \to \infty$. Let $x=(x_v)_{v \in G}$ be a $\newspace^G$-valued random variable that satisfies 
\begin{align}
\lim_{n\to\infty} \frac{1}{|A_n|}\sum_{v \in A_n}f(\varphi_vx)  = \E[f(x)], \ \ \text{in probability}, \label{asmp:fieldconv}
\end{align}
for each $f \in C_b(\newspace^G)$. Then, for each $f \in C_b(\newpathspace^G)$, 
\begin{align}
\lim_{n\to\infty} \frac{1}{|A_n|}\sum_{v \in A_n}f(\varphi_vX^{G,x})  = \E[f(X^{G,x})], \ \ \text{in probability}, \label{claim:fieldconv}
\end{align} 
 In particular, we have $\frac{1}{|A_n|}\sum_{v \in A_n}\delta_{X^{G,x}_v} \to \Lmc(X^{G,x}_{\oSlash})$   in probability in $\P(\newpathspace)$.
\end{proposition} 
\begin{proof}
  The ``in particular" claim at the end follows by taking $f$ in \eqref{claim:fieldconv} of the form $f(y) = \tilde{f}(y_{\oSlash})$ for $\tilde{f} \in C_b(\newpathspace)$, and by noting that $(\varphi_vX^{G,x})_{\oSlash} = X^{G,x}_{\varphi_v \oSlash}=X^{G,x}_v$. To prove \eqref{claim:fieldconv}, it suffices to show that  for each $k \in \N$ and each bounded Lipschitz $f \in C_b(\newpathspacek^G)$,   
\begin{equation} 
\lim_{n\to\infty} \frac{1}{|A_n|}\sum_{v \in A_n}f(\varphi_vX^{G,x}[k])  = \E[f(X^{G,x}[k])], \ \ \text{in probability}. \label{pf:fieldconv}
\end{equation} 
Now, the (conditional) correlation decay estimate \eqref{def:corrdecay-unified} implies that there exists a family of
  functions $\decayfn_\ell: \N \mapsto [0,\infty)$ with $\decayfn_\ell(n) \rightarrow 0$ as $n \rightarrow \infty$, $\ell \in \N$, such that: 
  for every $\ell \in \N$, $A,B \subset G$ with $|A|, |B| \le \ell$, and
$f \in C_b(\newpathspacek^A)$ and
$g \in C_b(\newpathspacek^B)$ with $\|f\|_{BL},\|g\|_{BL} \le 1$,
\begin{align*} 
	|\Cov(f(X^{G,x}_A[k]),g(X^{G,x}_B[k]) \,|\, x)| \le \decayfn_\ell(d_{G}(A,B)), \quad a.s.
\end{align*}
Note that the assumed existence of automorphisms $\varphi_v$ implies that each vertex has the same degree as the root $\oSlash$, and so the graph $G$ necessarily has bounded degree. Hence,
  by
  Remark \ref{rem-ergass}, the limit \eqref{eq:avggrowthcondition-disc} also holds for every $\ell \in \N$.  
Thus,     we may apply Proposition \ref{pr:corrdecay-general} with $\newRV_v=X^{G,x}_v[k]$ to find
\begin{align*} 
\lim_{n\to\infty}\frac{1}{|A_n|}\sum_{v \in A_n}\big(f(\varphi_vX^{G,x}[k]) - \E[f(\varphi_vX^{G,x}[k]) \, | \, x] \big) = 0, \quad \text{in probability.}
\end{align*}
Next, define $g \in C_b(\base^G)$ by $g(y) := \E[f(X^{G,y}[k])] =\E[f(X^{G,x}[k])\,|\,x=y]$, and note that $g(\varphi_vy) = \E[f(\varphi_vX^{G,y}[k])]$ according to \eqref{eq:autoidentity}. Using \eqref{asmp:fieldconv}, it follows that 
\begin{align*}
\frac{1}{|A_n|}\sum_{v \in A_n}\E[f(\varphi_vX^{G,x}[k]) \, | \, x] &= \frac{1}{|A_n|}\sum_{v \in A_n} g(\varphi_vx) \to \E[g(x)] = \E[f(X^{G,x}[k])],
\end{align*}
in probability. This shows \eqref{pf:fieldconv} and thus completes the proof.
\end{proof}

The condition \eqref{asmp:fieldconv} is a fairly general form of empirical field convergence.  As we show in 
 Corollary \ref{co:transitive} below, it holds, for example, if the
graph is suitably symmetric and the initial states exhibit correlation decay. 
Recall that a graph $G$ is said to be \emph{vertex transitive} if for all $u,v \in V$ there exists $\varphi \in \mathrm{Aut}(G)$ such that $\varphi u=v$. 

\begin{corollary} \label{co:transitive}
Suppose the graph $G$ is vertex transitive, and suppose $x=(x_v)_{v \in G}$ satisfies the following properties:
\begin{itemize} 
\item $\varphi x\stackrel{d}{=} x$ for all $\varphi \in \mathrm{Aut}(G)$.
\item For every $\ell \in \N$, there  exists a function $\newdecayfn_\ell : \N \to \R_+$ such that $\lim_{n\to\infty}\newdecayfn_\ell(n) = 0$ and
\begin{align*}
\Cov(f( x_A),g( x_B)) \le \newdecayfn_\ell(d_G(A,B)), \quad A, B \subset G, |A|, |B| \le \ell, 
\end{align*}
for  all $f \in \C_b(\X^A)$ and $g \in \C_b(\X^B)$,  with  $\|f\|_{BL} \le 1$ and $\|g\|_{BL} \le 1$. 
\end{itemize}
Then, for any sequence $\{A_n\}$ of finite subsets of $G$ with $|A_n| \to \infty$, we have $\frac{1}{|A_n|}\sum_{v \in A_n}\delta_{X^{G,x}_v}$ $\to \Lmc(X^{G,x}_v)$ in probability in $\P(\newpathspace)$.
\end{corollary}
\begin{proof} 
Since $G$ is vertex transitive, for each $v \in G$ there exists $\varphi_v \in \mathrm{Aut}(G)$ such that $\varphi_v\oSlash=v$. Also, a vertex transitive graph is regular and in particular of bounded degree. Hence, by Remark \ref{rem-ergass}, \eqref{eq:avggrowthcondition-disc} holds.  Using the assumption of correlation decay, we may apply Proposition \ref{pr:corrdecay-general} to get
\begin{align*}
\lim_{n\to\infty}\frac{1}{|A_n|}\sum_{v \in A_n} \big(f(\varphi_vx ) - \E[f(\varphi_v x)]\big) = 0,
\end{align*}
in probability, for $f \in C_b(\newspace^G)$. But $\varphi_v x \stackrel{d}{=} x$, and we deduce that $x$ satisfies \eqref{asmp:fieldconv}.  We complete the proof by applying Proposition \ref{pr:empfield}. 
\end{proof}

\begin{remark}
The two assumptions on $x$ are clearly satisfied if $(x_v)_{v \in G}$ are i.i.d.
\end{remark}

\subsection{Lattices and regular trees} \label{se:lattices + trees} 
We lastly highlight what can go wrong regarding empirical measure convergence, as discussed in Section \ref{se:regtree}. This situation is illustrated most clearly by the following two results. 
To begin with, let $\Z^d$ denote the integer lattice, and let $\Z^d_n = \Z^d \cap [-n,n]^d$.

\begin{proposition} \label{pr:lattice}
Let $d \in \N$.
Suppose $x=(x_v)_{v \in \Z^d}$ are i.i.d. $\newspace$-valued random elements, and let $x^n=(x_v)_{v \in \Z^d_n}$ for $n \in \N$. Then
\begin{align*}
\lim_{n\to\infty}\frac{1}{|\Z^d_n|}\sum_{v \in \Z^d_n}\delta_{X^{\Z^d_n,x^n}_v} = \lim_{n\to\infty}\frac{1}{|\Z^d_n|}\sum_{v \in \Z^d_n}\delta_{X^{\Z^d,x}_v} = \Lmc(X^{\Z^d,x}_0)
\end{align*}
in probability.
\end{proposition}
\begin{proof}
The second limit follows from Corollary \ref{co:transitive}, since $\Z^d$ is vertex transitive.
To prove the first, note that it is straightforward to check that
\begin{align}
\lim_{n\to\infty}\frac{1}{|\Z^d_n|}\sum_{v \in \Z^d_n}\delta_{(\Z^d_n,v)} = \delta_{(\Z^d,0)} \label{pf:Zdlocal}
\end{align}
in $\P(\G_*)$. To see this, define $|v|$ for $v \in \Z^d$ as the $\ell_\infty$ distance from $v$ to the origin, i.e., the unique value $r \in \N_0$ for which $v \in \Z^d_r \setminus \Z^d_{r-1}$.
Then $B_{n-|v|}(\Z^d_n,v) \cong B_{n-|v|}(\Z^d,0)$, and we have
\begin{align*}
\frac{1}{|\Z^d_n|}\sum_{v \in \Z^d_n} d_*\big((\Z^d_n,v), (\Z^d,0)\big) &\le \frac{1}{|\Z^d_n|}\sum_{v \in \Z^d_n}  2^{-(n-|v|)} \\
	&= \frac{2^{-n}}{|\Z^d_n|} + \sum_{r=1}^n 2^{-(n-r)} \frac{|\Z^d_r \setminus \Z^d_{r-1}|}{|\Z^d_n|} \\
	&= \frac{2^{-n}}{(2n+1)^d} + \sum_{r=1}^n 2^{-(n-r)} \frac{(2r+1)^d-(2r-1)^d}{(2n+1)^d}.
\end{align*}
Noting that $(2r+1)^d - (2r-1)^d \le 2d(2n+1)^{d-1}$ for $1 \le r \le n$, we see that the above
converges to zero as $n\to\infty$. 
This proves \eqref{pf:Zdlocal}, and we then deduce from the i.i.d.\ assumption on $x$ and from Corollary \ref{co:localweak iid} that
\begin{align*}
\lim_{n\to\infty}\frac{1}{|\Z^d_n|}\sum_{v \in \Z^d_n}\delta_{(\Z^d_n,v,x)} = \Lmc((\Z^d,0,x))
\end{align*}
in probability in $\P(\G_*[\newspace])$. This shows that $(\Z^d_n,x)$ converges in probability in the local weak sense to $(\Z^d,x)$, and we may thus apply Theorem \ref{thm:localprob-disc} or Theorem \ref{thm:localprob-cont}  
(and Remark \ref{re:local-law-root}) to complete the proof.
\end{proof}

Next, we consider the infinite $d$-regular tree  $\Tmb^d$, and let $\Tmb^d_n$ denote the $d$-regular tree of height $n$; that is, if $\oSlash \in \Tmb^d$ denotes an arbitrary choice of root, then $\Tmb^d_n$ is the induced subgraph with vertex set $B_n(\Tmb^d)$. 
In order to describe the convergence of $\mu^{\Tmb^d_n,x}$, we consider the following infinite tree $\widetilde{\Tmb}^d = (\Vtil,\Etil)$, which one can interpret as the infinite limit of $\Tmb^d_n$ from the point of view of a leaf. This is sometimes known as the $d$-canopy tree \cite[Lemma 2.8]{dembo-montanari}, pictured in Figure \ref{fig:canopytree} below:
\begin{equation*}
\Vtil = \Nmb_0 \times \Nmb_0, \quad \Etil = \{ ((i,j),(i+1,\lfl j/(d-1) \rfl)) : i,j \in \Nmb_0 \}.
\end{equation*}
Intuitively, $\L(X^{\widetilde{\Tmb}^d,x}_{(i,0)})$ is the limiting law of a particle at distance $i$ from the nearest leaf.

\begin{proposition} \label{pr:regtree}
Suppose $x=(x_v)_{v \in \Tmb^d}$ are i.i.d.\ $\newspace$-valued random elements, and let $x^n=(x_v)_{v \in \Tmb^d_n}$ for $n \in \N$. Then
\begin{equation*}
\lim_{n\to\infty}\frac{1}{|\Tmb^d_n|}\sum_{v \in \Tmb^d_n}\delta_{X^{\Tmb^d,x}_v} = \Lmc(X^{\Tmb^d,x}_{\oSlash}),
\end{equation*}
in probability, whereas
\begin{equation*}
\lim_{n\to\infty}\frac{1}{|\Tmb^d_n|}\sum_{v \in \Tmb^d_n}\delta_{X^{\Tmb^d_n,x}_v} = \sum_{i=0}^\infty \frac{d-2}{(d-1)^{i+1}} \Lmc(X^{\widetilde{\Tmb}^d,x}_{(i,0)}). 
\end{equation*}
\end{proposition}
\begin{proof}
The first claim follows from Corollary \ref{co:transitive}, since $\Tmb^d$ is vertex transitive. To prove the second, let us assign a random root $\oSlash$ in  $\widetilde{\Tmb}^d$ by setting $\widetilde\oSlash = (i,0)$ with probability $(d-2)/(d-1)^{i+1}$, for each $i \in \N_0$. Then $\compu(\Tmb^d_n)$ converges to $(\widetilde{\Tmb}^d,\widetilde\oSlash)$ locally in probability by \cite[Lemma 2.8]{dembo-montanari}, i.e.,
\begin{equation*}
\frac{1}{|\Tmb^d_n|}\sum_{v\in \Tmb^d_n} \delta_{(\Tmb^d_n,v)} \to \Lmc((\widetilde{\Tmb}^d,\widetilde\oSlash))
\end{equation*}
in probability in $\P(\G_*)$.
Since $x$ is assumed i.i.d., Corollary \ref{co:localweak iid} implies that
\begin{equation*}
\frac{1}{|\Tmb^d_n|}\sum_{v\in \Tmb^d_n} \delta_{(\Tmb^d_n,v,x)} \to \Lmc((\widetilde{\Tmb}^d,\widetilde\oSlash,x))
\end{equation*}
in probability in $\P(\G_*[\newspace])$. This shows that $(\Tmb^d_n,x)$ converges in probability in the local weak sense to $(\widetilde\Tmb^d,x)$, and we may thus apply Theorem \ref{thm:localprob-disc} (and Remark \ref{re:local-law-root}) to deduce that
\begin{equation*}
\frac{1}{|\Tmb^d_n|}\sum_{v\in \Tmb^d_n} \delta_{(\Tmb^d_n,v,X^{\Tmb^d_n,x})} \to \Lmc((\widetilde{\Tmb}^d,\widetilde\oSlash,X^{\widetilde{\Tmb}^d,x})) = \sum_{i=0}^\infty \frac{d-2}{(d-1)^{i+1}} \Lmc(\widetilde{\Tmb}^d,(i,0),X^{\widetilde{\Tmb}^d,x})
\end{equation*}
in probability in $\P(\G_*[\newpathspace])$. Apply the continuous mapping theorem using the root map $(G,\oSlash,y) \mapsto y_{\oSlash}$ to complete the proof.
\end{proof}

\begin{figure}
\tikzstyle{every node}=[fill,scale=0.5]
\begin{center}
  \begin{tikzpicture}[scale=0.8]
  \def \n {7}
  \def \m {7}
  \def \k {3}
     labels
    \foreach \i in {0,...,\m}
      \foreach \j in {0,...,\n}{
        \node[draw, circle] at (\i,\j) {};
        }
    \foreach \i in {1,...,\m}
      \foreach \j in {0,...,\n}{
        \draw (\i - 1,\j) edge (\i , {floor(\j / (\k-1))});
      }
  \end{tikzpicture}
	\caption{Part of the $\kappa$-canopy tree $\Gtil$, for $\kappa=3$.}
	\label{fig:canopytree}
\end{center}
\end{figure}
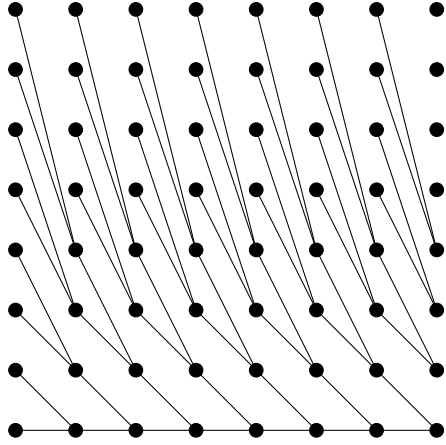

\subsection{Propagation of ergodicity} \label{se:propergo}

The convergence of empirical fields studied in the previous section is closely related to \emph{ergodicity}, discussed in this section. This section explains that the dynamics under investigation in this paper propagate the property of ergodicity as well, in the most natural context of Cayley graphs.
We first collect some basic terminology from group theory and ergodic theory, and assume throughout that all groups are countable and endowed with the discrete topology.
If $G$ is a group and $S$  a finite set of generators satisfying $S=S^{-1}$ (where $S^{-1} := \{g^{-1} : g \in S\}$), then the \emph{Cayley graph} of $(G,S)$ is the graph with vertex set $G$ and edge set $E=\{(g,gs) : s \in S, g \in G\}$.
Classical examples include $\Z^d$ and the $d$-regular tree for $d \ge 2$ even, the latter being the Cayley graph of the free group of order $d/2$.

For a Polish space $\newPolish$ and a group $G$, 
recall that $G$ acts naturally on $\newPolish^G$ via $g\newpolish := (\newpolish_{gv})_{v \in G}$ for $\newpolish=(\newpolish_v)_{v \in G} \in \newPolish^G$. We write $gA := \{g\newpolish : \newpolish \in A\}$ for any set $A \subset \newPolish^G$.
We say that $\mu \in \P(\newPolish^G)$ is \emph{$G$-invariant} if $\mu(g A)=\mu(A)$ for any Borel set $A \subset \newPolish^G$ and $g \in G$, and \emph{ergodic} if
\begin{align*}
\mu(A \, \Delta \, g A) = 0  \ \ \ \forall g \in G \quad \Longrightarrow \quad \mu(A)=0 \text{ or } \mu(A)=1.
\end{align*}
Alternatively, for a $\newPolish^G$-valued random variable $\newRV=(\newRV_g)_{g \in G}$, we say that $\newRV$ is $G$-invariant (resp.\ ergodic) if its law is.
The following proposition explains the notable and remarkably simple phenomenon of  \emph{propagation of ergodicity}:

\begin{proposition} \label{pr:ergopropagation}
Suppose $G$ is the Cayley graph of a finitely generated group. Suppose $x=(x_v)_{v\in G}$ is a $G$-invariant, ergodic $\newspace^G$-valued random element. Then $X^{G,x}$ (resp.\ $X^{G,x}[k]$, for any $k \in \N$) is a $G$-invariant, ergodic $\newpathspace^G$-valued (resp.\ $\newpathspacek^G$-valued) random element.
\end{proposition}
\begin{proof}
  The essential fact is that a \emph{factor} of an ergodic law is ergodic, and  we apply this by expressing
    the  solution of the SDE system as a factor of the initial configuration and noises.
  
   More precisely, for a measurable space $\newPolish$, a \emph{factor} of a random element $z=(z_v)_{v \in G}$ of $\newPolish^G$ is defined as any random element $\bar{z}=(\bar{z}_v)_{v \in G}$ of $\overline\newPolish^G$, for some measurable space $\overline\newPolish$, such that $\bar{z}=\varphi(z)$ a.s.\ for some product-measurable map $\varphi : \newPolish^G \to \overline\newPolish^G$ which commutes with the group action, i.e., $\varphi(gz)=g\varphi(z)$ for all $g \in G$, $z \in \newPolish^G$.
   It is a straightforward consequence of the definitions that if $z$ is ergodic then so is any factor of it.
  
To  apply this  in  our  context,  
for $v \in G$, let $y_v=W_v$ denote the Brownian motions in the diffusive case, or let $y_v=\xi_v$ denote the noise process in the discrete case.
Note that $y=(y_v)_{v \in G}$ are i.i.d.\ and thus ergodic.
Because $x$ is ergodic by assumption, and $x$ and $y$ are independent, we deduce that $(x,y)$ is ergodic as well.
In either the diffusive or discrete case, because the SDE admits a strong solution in the former case by Theorem \ref{th:uniqueness-hominfSDE}, we can express the solution process $X^{G,x}=\Phi(x,y)$ as a measurable function $\Phi$ of the initial states and noises. This map satisfies $\Phi(gx,gy)=g\Phi(x,y)$ a.s., thanks to the automorphism-invariance of the dynamics. 
\end{proof}

\begin{remark}	\label{rmk:Dereudre-Roelly}
In the diffusion setting, the phenomenon of propagation of ergodicity is shown in Theorem 2.2 of \cite{DereudreRoelly2017} in the special case when $\sigma=1$ and $G=\Zmb^d$.
The authors allow for a more general class of possibly discontinuous drifts for which there may be non-uniqueness of the infinite SDE solution, and this non-uniqueness makes the problem far more delicate. In particular, weak solutions that are not strong solutions cannot be constructed as factors of the initial conditions and Brownian motions, as in the proof of Proposition \ref{pr:ergopropagation} given above. Instead,
 they   use an approach using a Girsanov change of measure, exploiting the 
    fact that the  diffusion coefficient $\sigma$  is constant.
\end{remark}

\appendix

\section{Local convergence of marked graphs} \label{ap:localconvergence}

\subsection{Essential properties of the metric space of local convergence}
\label{subs:essential}

This appendix develops the essential properties of the space $\G_*[\Polish]$ of isomorphism classes of
rooted connected marked graphs, defined in Section \ref{se:background-localmarkedgraphs}.
Throughout this section, $(\Polish,d)$ is a fixed metric space. Some of these results (albeit with a different choice of metric that induces the same topology) can be found in \cite[Section 3.2]{Bordenave2016}.

Let $I(G,G')$ denote the set of isomorphisms between two graphs $G,G'\in \G_*$.
Recall from Section \ref{subsub-unmarked} that a sequence $\{(G_n,{\polish}^n)\}$ $\subset \G_*[\Polish]$ \emph{converges locally} to $(G,{\polish}) \in \G_*[\Polish]$ if, for every $k \in \N$ and $\epsilon > 0$, there exist $N \in \N$ such that for all $n \ge N$ there exists $\varphi \in I(B_k(G_n),B_k(G))$ with $d({\polish}^n_v,{\polish}_{\varphi(v)}) < \epsilon$ for all $v \in B_k(G_n)$, where recall
that  $B_k(G_n)$ represents the induced subgraph of $G_n$ on vertices of $G_n$ that 
are no greater than distance $k$ from the root.  
We may endow $\G_*[\Polish]$ with either of the following two metrics: 
\begin{align*}
d_*((G,{\polish}),(G',{\polish}')) &= \sum_{k=1}^\infty 2^{-k}\left(1 \wedge \inf_{\varphi \in I(B_k(G),B_k(G'))}\max_{v \in B_k(G)}d(\polish_v,\polish'_{\varphi(v)})\right), \\
d_{*,1}((G,{\polish}),(G',{\polish}')) &= \sum_{k=1}^\infty 2^{-k}\left(1 \wedge \inf_{\varphi \in I(B_k(G),B_k(G'))}\frac{1}{|B_k(G)|}\sum_{v \in B_k(G)}d(\polish_v,\polish'_{\varphi(v)})\right), 
\end{align*} 
where the infimum of the empty set is understood to be infinite. We will show in Lemma \ref{le:G*metric} that these are genuine metrics on $\G_*[\Polish]$. 
The following proposition confirms first that they are indeed compatible with the aforementioned notion of local convergence.

\begin{proposition}
Let $(G,{\polish}),(G_n,{\polish}^n) \in \G_*[\Polish]$, for $n \in \N$. The following are equivalent:
\begin{enumerate}
\item $(G_n,{\polish}^n)$ converges locally to $(G,{\polish})$.
\item $d_*((G_n,{\polish}^n),(G,{\polish})) \rightarrow 0$.
\item $d_{*,1}((G_n,{\polish}^n),(G,{\polish})) \rightarrow 0$.
\end{enumerate} 
\end{proposition}
\begin{proof}
Clearly $d_{*,1} \le d_*$, so (2) $\Rightarrow$ (3). To prove (1) $\Rightarrow$ (2), suppose $(G_n,{\polish}^n)$ converges locally to $(G,{\polish})$. Fix $\epsilon > 0$ and $k \in \N$ such that $2^{1-k} \le \epsilon$. Find $n_k$ such that for all $n \ge n_k$ there exists $\varphi_n \in I(B_k(G_n),B_k(G))$ with $d({\polish}^n_v,{\polish}_{\varphi_n(v)}) < 2^{-k}$ for all $v \in B_k(G_n)$. Note that  for $j < k$ the restriction $\varphi_n|_{B_j(G_n)}$ belongs to $I(B_j(G_n),B_j(G))$. We deduce that, for $n \ge n_k$,
\begin{align*}
d_*((G_n,{\polish}^n),(G,{x})) &< \sum_{j=1}^k2^{-j}2^{-k} + \sum_{j=k+1}^{\infty}2^{-j}\left(1 \wedge \inf_{\varphi \in I(B_j(G),B_j(G'))}\max_{v \in B_j(G)}d(\polish_v,\polish'_{\varphi(v)})\right) \\
	&\le 2^{-k} + 2^{-k} \le \epsilon.
\end{align*}

Finally, to prove (3) $\Rightarrow$ (1), fix $k \in \N$ and $\epsilon > 0$. Choose $M \in \N$ such that $2^{-M} < \epsilon / |B_k(G)|$ and $M \ge k$. Find $N$ such that $d_{*,1}((G_n,{\polish}^n),(G,{\polish})) < 2^{-2M}$ for all $n \ge N$. Then 
\[
\inf_{\varphi \in I(B_j(G),B_j(G_n))}\frac{1}{|B_j(G)|}\sum_{v \in B_j(G)}d(\polish_v,\polish^n_{\varphi(v)}) < 2^{j-2M} \le 2^{-M}, \qquad n \ge N, \ j \le M.
\]
In particular, choosing $j=k$, we may thus find for each $n$ some $\varphi_n \in I(B_k(G),B_k(G_n))$ such that
\[
\frac{1}{|B_k(G)|}\sum_{v \in B_k(G)}d(\polish_v,\polish^n_{\varphi(v)}) < 2^{-M}.
\]
Bounding the maximum by the sum,
\begin{align*}
\max_{v \in B_k(G)}d(\polish_v,\polish^n_{\varphi(v)}) < 2^{-M} |B_k(G)| \le \epsilon.
\end{align*}
In summary, we have shown that for each $k\in \N$ and $\epsilon > 0$ there exists $N \in \N$ such that for all $n \ge N$ there exists $\varphi_n \in I(B_k(G),B_k(G_n))$ such that $\max_{v \in B_k(G)}d(\polish_v,\polish^n_{\varphi(v)}) <  \epsilon$. This shows that $(G_n,{\polish}^n) \to (G,{\polish})$ locally, and the proof is complete.
\end{proof}

\begin{lemma}
  \label{le:G*metric} 
$(\G_*[\Polish],d_*)$ and $(\G_*[\Polish],d_{*,1})$ are metric spaces.
\end{lemma}
\begin{proof}
  We first check that $d_*$ is a metric. Symmetry is clear, as is the fact that $(G,{\polish}) \cong (G',{\polish}')$ implies $d_*((G,{\polish}), (G',{\polish}'))=0$. Conversely, if $d_*((G,{\polish}), (G',{\polish}'))=0$, we show that $(G,{\polish})$ and  $(G',{\polish}')$ are isomorphic as follows: Find a sequence $\varphi_k \in I(B_k(G),B_k(G'))$ such that $\polish_v = \polish'_{\varphi_k(v)}$ for all $v \in B_k(G)$. Extend each $\varphi_k$ arbitrarily to a function from $G$ to $G'$, and view each $\varphi_k$ as an element of the space $(V')^{V}$. Endowing $V'$ and $V$ with the discrete topology, we may equip $(V')^V$ with the topology of pointwise convergence. The sequence $(\varphi_n)$ is pre-compact in this topology since $\varphi_n|_{B_k(G)} \subset B_k(G')^{V}$ for each $n \ge k$, so we may find a subsequential limit point $\varphi : V \rightarrow V'$. The restriction  $\varphi|_{B_k(G)}$ belongs to $I(B_k(G),B_k(G'))$ for each $k$, and it follows that $\varphi$ must be an isomorphism from $G$ to $G'$. Moreover, we must have $\polish_v = \polish'_{\varphi(v)}$ for all $v \in B_k(G)$, for all $k$, and we conclude that $\varphi$ is an isomorphism from $(G,{\polish})$ to $(G',{\polish}')$.
  
  Next, note that $d_{*,1}((G,{\polish}), (G',{\polish}'))=0$ if and only if $d_*((G,{\polish}), (G',{\polish}'))=0$. Therefore $d_{*,1}$ is also a metric.
\end{proof}

The following lemma is taken from \cite[Lemma 3.4]{Bordenave2016}.
\begin{lemma}
If $\Polish$ is a Polish space, then so is $\G_*[\Polish]$.
\end{lemma}

\subsection{Auxiliary results}
\label{apsub:auxiliary}

With the essential properties of the metric space $(\G_*[\Polish],d_*)$ now established, we
now establish two auxiliary results.  The first addresses the question of convergence of empirical measures.

\begin{proposition} \label{pr:localconvergence-empiricalmeasure}
Suppose $(\Polish,d)$ is a complete, separable metric space.
Let $(G,{\polish})$, $(G_n,{\polish}^n) \in \G_*[\Polish]$, and assume $G$ and $G_n$ are finite graphs. Define the empirical measures
\[
\mu^{G} = \frac{1}{|G|}\sum_{v \in G}\delta_{\polish_v}, \quad\quad \mu_n = \frac{1}{|G_n|}\sum_{v \in G_n}\delta_{\polish^n_v}.
\]
If $(G_n,{\polish}^n) \rightarrow (G,{\polish})$ in $\G_*[\Polish]$, then $\mu_n \rightarrow \mu^G$ in $\P(\Polish)$.  
\end{proposition}
\begin{proof}
Fix finite graphs $G$, $G_n$ in $\G_*$. 
Consider the $1$-Wasserstein (Kantorovich) metric,
\[
W_1(m,m') = \sup\left\{ \int_\X f\,d(m-m') : f : \X \rightarrow \R, \ |f(x)-f(y)| \le d(x,y) \, \forall x,y \in \X \right\}.
\]
It is well known that convergence in this metric implies weak convergence.
For any (rooted connected) graph $G'=(V',E',\oSlash') \in \G_*$ let
$R(G') = \inf\{n \ge 0 : G' = B_n(G')\}$, and note that $R(G')$ is simply the distance from the root to the furthest vertex.
A graph $G' \in \G_*$ is finite if and only if $R(G') < \infty$. Moreover, $G'=B_{R(G')}(G')=B_r(G')$ for any $r \ge R(G')$. Because the graph is connected, it is also clear that if $B_r(G') = B_s(G')$ for some $s > r$, then there are no vertices that are at a distance greater than  $r$ from the root, and so $G'=B_r(G')$ and $R(G') \le r$.

Now, let $r=2R(G)$. 
Let $\epsilon > 0$.  The assumed convergence $(G_n, {\polish}^n) \rightarrow (G,{\polish})$ implies 
the existence of $N \in \N$ such that for all $n \ge N$
there exists $\varphi_n \in I(B_r(G),B_r(G_n))$ such that  $\max_{v \in B_r(G)}d(\polish_v,\polish^n_{\varphi_n(v)}) < \epsilon$. Now, since $G = B_r(G) = B_{R(G)}(G)$, by isomorphism we must have $B_r(G_n)=B_{R(G)}(G_n)$. From the argument of the previous paragraph we deduce that $G_n=B_r(G_n)$ and $R(G_n) = R(G)$. Thus $\varphi_n$ is an isomorphism from $G$ to $G_n$, and
\begin{equation*}
W_1(\mu_n,\mu) = \sup_f \frac{1}{|G|}\sum_{v \in G} \left(f(\polish_v) - f(\polish^n_{\varphi_n(v)})\right) \le \frac{1}{|G|}\sum_{v \in G} d(\polish_v,\polish^n_{\varphi_n(v)}) < \epsilon. \qedhere
\end{equation*}
\end{proof}

We now present the proof of Lemma \ref{le:localweakprob characterization} stated in Section \ref{se:local weak in prob},
which provides equivalent characterizations of convergence in probability in the local weak sense. 

\begin{proof}[Proof of Lemma \ref{le:localweakprob characterization}]
The proof is similar to that of the Sznitman-Tanaka theorem \cite[Proposition 2.2(i)]{sznitman1991topics}.
A simple and well known argument shows that the total variation distance between $\L((U^n_1,U^n_2) \,|\, G_n)$ and $\L((\pi_n(1),\pi_n(2)) \,|\, G_n)$ is no more than $2/|G_n|$ on the set $|G_n| \geq 2$,
where $\pi_n$ is a uniformly random permutation of the vertex set of $G_n$ (given $G_n$, and assuming without loss of generality that the vertex set of $G_n$ is $\{1,\dotsc,|G_n|\}$). 
 Since $|G_n|\to\infty$ in probability, we deduce that the total variation distance between $\L(U^n_1,U^n_2)$ and $\L(\pi_n(1),\pi_n(2))$ vanishes. 
Thus, \eqref{def:asympindep} is equivalent to
\begin{align}
\E[g_1(\comp_{\pi_n(1)}(G_n,\polish^n))g_2(\comp_{\pi_n(2)}(G_n,\polish^n))] \to \E[g_1(G,\polish)]\E[g_2(G,\polish)], \quad \forall  g_1,g_2 \in C_b(\G_*[\Polish]). \label{pf:asympindep1}
\end{align}
Let $\mu_n := \frac{1}{|G_n|}\sum_{v \in G_n} \delta_{\comp_{\pi_n(v)}(G_n,\polish^n)}$.
Since the  (conditional) joint  law $\L\big((\comp_{\pi_n(v)}(G_n,\polish^n))_{v \in G_n} | G_n \big)$ is exchangeable, for $g_1,g_2 \in C_b(\G_*[\Polish])$ we have 
\begin{align}
\E\left[ \lan \mu_n,g_1\ran \lan\mu_n, g_2\ran \right] &= \E \left[ \frac{1}{|G_n|^2} \sum_{u,v \in G_n} g_1(\comp_{\pi_n(u)}(G_n,\polish^n))g_2(\comp_{\pi_n(v)}(G_n,\polish^n)) \right]  \nonumber \\
	&= \E\Bigg[ \frac{|G_n|-1}{|G_n|} g_1(\comp_{\pi_n(1)}(G_n,\polish^n))g_2(\comp_{\pi_n(2)}(G_n,\polish^n)) \label{pf:asympindep3} \\
	&\qquad\qquad + \frac{1}{|G_n|} g_1(\comp_{\pi_n(1)}(G_n,\polish^n)) g_2(\comp_{\pi_n(1)}(G_n,\polish^n))\Bigg].  \nonumber
\end{align}

Now suppose that \eqref{def:asympindep}, or equivalently \eqref{pf:asympindep1},  holds.  Let $f \in C_b(\G_*[\Polish])$, and take $g_i(\cdot):= f(\cdot)-\E[f(G,\polish)]$ for $wi=1,2$. Then the right-hand side of \eqref{pf:asympindep3} converges to $\E[g_1(G,\polish)]\E[g_2(G,\polish)]$ $=0$, and we deduce that 
\[
\E\left[\left(\lan \mu_n,f\ran - \E[f(G,\polish)]\right)^2\right] = \E[\lan \mu_n, g_1\ran\lan\mu_n,g_2\ran] \to 0.
\]
As this holds for arbitrary $f$, we deduce that 
\begin{align}
\lim_{n\to\infty}\frac{1}{|G_n|}\sum_{v \in G_n} \delta_{\comp_v(G_n,\polish^n)} = \lim_{n\to\infty}\mu_n = \L(G,\polish), \qquad \text{in } \P(\G_*[\Polish]), \text{ in probability}. \label{pf:asympindep2}
\end{align}
Note that the first identity is just the definition of $\mu_n$, upon removing the permutation.
Thus, \eqref{pf:asympindep2} is precisely the convergence in probability in the local weak sense of $(G_n,\polish)$ to $(G,\polish)$, which completes the proof of the ``if" part of the claim. 

To prove the converse, we assume \eqref{pf:asympindep2} holds and deduce \eqref{pf:asympindep1} as follows. Note that \eqref{pf:asympindep2} implies $\E\left[ \lan \mu_n,g_1\ran \lan\mu_n, g_2\ran \right]$ converges to $\E[g_1(G,\polish)]\E[g_2(G,\polish)]$, whereas the right-hand side of \eqref{pf:asympindep3} clearly has the same $n\to\infty$ limit as the left-hand side of \eqref{pf:asympindep1} since $|G_n|\to\infty$. 
\end{proof}

\section{Proofs of local weak convergence of Gibbs measures} \label{ap:Gibbsmeasures}

The goal of this section is to prove the results in Section \ref{subsub-icconv}.
A number of prior works, such as \cite{dembo2013factor,dembo-montanari}, have studied Gibbs measures on locally converging (sparse) graph sequences and
  Lemma \ref{le:gibbsunif} on the convergence of the whole particle configuration is well known in a more general context (see  \cite{georgii-gibbs}),
  but we include it here for completeness.

Although we have focused our attention on \emph{factor models} with pairwise interactions,  the same arguments extend easily to bounded-range interactions.
Recall the definitions given in  Section \ref{subsub-icconv}, and fix the pair $(\psi, \lambda)$ as defined therein.   
Also, recall that given a Polish space $\Polish$ and a graph $G = (V,E)$,   $P \in \P(\Polish^V)$ is said to be a
   {\it Markov random field}  with respect to $G$ if 
  for every finite $A \subset V$, 
  \begin{equation}
    \label{P-MRF}
   P(y_A\,|\,\polish_{V\setminus A}) = P(y_A\,|\, \polish_{\partial A}) \quad \mbox{for } P\mbox{-a.e. }  \polish_{V\setminus A}  \in \Polish^{V\setminus A}. 
  \end{equation}  

The first step toward the proofs of Propositions \ref{pr:lw gibbs} and \ref{pr:lwprob gibbs}
is the following lemma, which is inspired by
\cite[Proposition 7.11]{georgii-gibbs}.

\begin{lemma} \label{le:gibbsunif}
  Suppose $G=(V,E) \in \mathcal{U}$. Let $P_G$ be the unique $(\psi,\lambda)$-Gibbs measure, and    
  let $A_n$ be any increasing sequence of finite sets with $\cup_nA_n = V$. Then, for any $m \in \N$
  and $f \in C_b(\Polish^{A_m})$, we have
   \begin{align*}
\lim_{n\to\infty}\sup_{\polish_{\partial A_n} \in \Polish^{\partial A_n}}\left| \int_{\Polish^{A_n}} f(\polish_{A_m})\,\gamma^G_{A_n}(d\polish_{A_n} \,|\, \polish_{\partial A_n}) - \int_{\Polish^V} f(\bar{\polish}_{A_m})\,P_G(d\bar{\polish}_V)\right| = 0.
\end{align*}
\end{lemma}

Recall that the definition of the kernel $\gamma^G_A(d\polish_A\,|\, \polish_{\partial A})$ is given pointwise in \eqref{def:gammaGAgibbs}, in terms of the continuous interaction function $\psi$. Because we work with this particular version of the conditional probability measures, it makes sense that Lemma \ref{le:gibbsunif} is stated in terms of a supremum rather than an essential supremum.

\begin{proof}[Proof of Lemma \ref{le:gibbsunif}]
We first note that $\polish_{\partial A} \mapsto \gamma^G_A( \cdot \,|\, \polish_{\partial A})$ is continuous with respect to weak convergence, for any finite graph $G$ and nonempty finite set of vertices $A$, because $\psi$ is bounded and continuous. Moreover, because $\psi$ is bounded, we have
\begin{align*}
\sup_{\polish \in \Polish^V} \frac{d\gamma^G_A(\cdot \,|\, \polish_{\partial A})}{d\lambda^A}(\polish_A) < \infty,
\end{align*}
In particular, this readily implies that
\begin{align}
\{\gamma^G_A(\cdot \,|\, \polish_{\partial A}) : \polish_{\partial A} \in \Polish^{\partial A}\} \subset \P(\Polish^A) \ \ \text{ is tight for each } G \text{ and } A. \label{eq:gibbstightness}
\end{align}

Now suppose that, in contradiction to the assertion of the lemma,  there exist an increasing sequence of finite sets $A_n$ with $\cup_nA_n=V$, $m \in \N$, $f \in C_b(\Polish^{A_m})$,  $\epsilon > 0$, and $y^n_{\partial A_n} \in \Polish^{\partial A_n}$ such that
\begin{equation}
\left| \int_{\Polish^{A_n}} f(\polish_{A_m})\,\gamma^G_{A_n}(d\polish_{A_n} \,|\, y^n_{\partial A_n}) - \int_{\Polish^V} f(\polish_{A_m})\,P_G(d\polish_V)\right| \ge \epsilon, \quad \forall n \ge m. \label{pf:gibbsunique1}
\end{equation}
Define $P^n \in \P(\Polish^V)$ by setting  
\begin{align*}
P^n(d\polish_V) = \gamma^G_{A_n}(d\polish_{A_n} \,|\, y^n_{\partial A_n}) \prod_{v \in V \setminus A_n} \lambda(d\polish_v).
\end{align*}
Then it is easy to verify  that $P^n$ is a Markov random field with  respect to $G$, in the sense that \eqref{P-MRF} holds when $P$ is replaced with $P^n$. 
  Moreover, 
  as a consequence of \eqref{eq:gibbstightness}, the sequence $(P^n)$ is tight and thus has a weak limit point, say $P \in \P(\Polish^V)$. 
By \eqref{pf:gibbsunique1}, we have 
\begin{align}
 \left| \int_{\Polish^V} f(\polish_{A_m}) \,(P-P_G)(d\polish_V) \right|  \ge \epsilon. \label{pf:gibbsunique2}
\end{align}
Now, let $(P^{n_k})$ denote a subsequence of $(P^n)$ that converges weakly to $P$. 
Also, let $\RV^k=(\RV^{k}_v)_{v \in V}$ and $\RV=(\RV_v)_{v \in V}$ be random $\Polish^V$-valued elements  with laws $P^{n_k}$ and $P$, respectively.  Consider disjoint finite sets $B,C \subset V$ and  $g \in C_b(\Polish^B)$, $h \in C_b(\Polish^C)$. Then, first using the weak convergence of $(P^{n_k})$ to $P$, 
  then the Markov random field property of  $P^n$,   the definition of $P^{n_k}$, and finally the
continuity of $\gamma^G_B$, we have 
\begin{align*}
\E[g(\RV_B)h(\RV_C)] &= \lim_{k\to\infty}\E[g(\RV^k_B)h(\RV^k_C)] \\
	&= \lim_{k\to\infty}\E[\E[g(\RV^k_B)\,|\, \RV^k_{V \setminus B}]h(\RV^k_C)] \\
	&= \lim_{k\to\infty}\E[\lan \gamma^G_B(\cdot\,|\,\RV^k_{\partial B}), \, g \ran h(\RV^k_C)] \\
	&= \E[\lan \gamma^G_B(\cdot\,|\,\RV_{\partial B}), \, g \ran h(\RV_C)]. 
\end{align*}
Because $B$ and $C$ are arbitrary finite subsets of $V$, this is enough to conclude that $P$ belongs to $\mathrm{Gibbs}(G) = \mathrm{Gibbs}(G, \psi,\lambda)$. Since $G \in {\mathcal U}$, 
this implies $P=P_G$, which contradicts \eqref{pf:gibbsunique2}.  
\end{proof}

The following Lemma includes Proposition \ref{pr:lw gibbs} as a special case (by taking $G_n^2$ to be an independent copy of $G_n^1$ and $G_n$ to be the disjoint union of
$G_n^1$ and $G_n^2$), and it will also be useful in proving Proposition \ref{pr:lwprob gibbs}:

\begin{lemma} \label{le:lwgibbs double}
For $n \in \N$, let $G_n$ be a finite (possibly disconnected) random graph, and  for $i=1,2$, let $o^i_n$ be a (random) vertex in $G_n$, and let $G^i_n$ be an induced (random) subgraph of $G_n$ rooted at $o^i_n$.
Assume $\L(G^1_n,G^2_n) \to \L(G^1,G^2)$ in $\P(\G_* \times \G_*)$ for some random elements $G^1,G^2$ of $\mathcal{U}$, and assume also that $d_{G_n}(o^1_n,o^2_n) \to \infty$ as $n\to\infty$ in probability. 
Then, for  random elements $Y^{G_n}$, $Y^{G^1}$, and $Y^{G^2}$  with laws $P_{G_n}$, $P_{G^1}$, and $P_{G^2}$, respectively,  we have
\[
\Lmc\big((G^1_n,\RV^{G_n}_{G^1_n}),(G^2_n,\RV^{G_n}_{G^2_n})\big) \to \Lmc(G^1,\RV^{G^1}) \times \Lmc(G^2,\RV^{G^2}), \quad \text{in } \G_*[\Polish] \times \G_*[\Polish].
\]
\end{lemma}
\begin{proof}
By the Skorohod representation theorem, we may assume that $G_n$, $(G_n^1,o_n^1)$, and $(G_n^2,o_n^2)$ are non-random. 
Fix $r \in \N$ and  $f_1,f_2 \in C_b(\G_*[\Polish])$ with $|f_1|,|f_2| \le 1$. Recall that $B_r(G)$ denotes the ball of radius $r$ around the root in $G$, and we similarly write $B_r(G,\polish)$ for a marked graph. It suffices to show that
\begin{align}
\lim_{n\to\infty}\E \left[f_1(B_r(G^1_n,\RV^{G_n}_{G^1_n}))f_2(B_r(G^2_n,\RV^{G_n}_{G^2_n}))\right] = \E\left[f_1(B_r(G^1,\RV^{G^1}))\right]\E\left[f_2(B_r(G^2,\RV^{G^2}))\right]. \label{pf:gibbsdouble1}
\end{align}
Let $\epsilon > 0$.  
We may define a function $\widehat{f}_i \in C_b( \Polish^{B_r(G^i)})$ by $\widehat{f}_i(\polish) := f_i(B_r(G^i,\polish))$, for $i = 1, 2$. 
By Lemma \ref{le:gibbsunif} we may find $\ell > r$ such that, for each $i=1,2$, 
\begin{align}
  \sup_{\polish \in \Polish^{\partial B_{\ell}(G^i)}}
  \left| \int_{ \Polish^{B_\ell(G^i)} }  \widehat{f}_i(\polish_{B_r(G^i)}) \, \gamma^{G^i}_{B_{\ell}(G^i)}(d\polish_{B_{\ell}(G^i)} \,|\, \polish_{\partial B_\ell(G^i)})
    - \int_{\Polish^{G^i}} \widehat{f}(\polish_{B_r(G^i)}) \,P_{G^i}(d\polish)\right| \le \epsilon. \label{pf:gibbslocal12}
\end{align}
For each $i=1,2$, since $G^i_n \to G$, we may find $N < \infty$ such that for all $n \ge N$ there exists an isomorphism $\varphi^i_n : B_{\ell+1}(G^i) \to B_{\ell+1}(G^i_n)$.
For any positive integer $m \le \ell+1$ we may also view $\varphi^i_n$ as a dual map $\Polish^{B_m(G^i_n)} \to \Polish^{B_m(G^i)}$ by setting
\begin{align*}
\varphi^i_n\polish = (\polish_{\varphi^i_n(v)})_{v \in B_m(G^i)}, \qquad \text{for } \polish=(\polish_v)_{v \in B_m(G^i_n)}.
\end{align*}
For $\bar{\polish} \in \Polish^{\partial B_\ell(G^i_n)}$, we have
\begin{align*}
\E \left[ \widehat{f}_i(\varphi^i_n \RV^{G_n}_{B_r(G^i_n)}) \, | \, \RV^{G_n}_{\partial B_{\ell}(G^i_n)} = \bar{\polish}\right] &= \int_{\Polish^{B_{\ell}(G^i_n)}} \widehat{f}_i(\varphi^i_n \polish_{B_r(G^i_n)}) \, \gamma^{G^i_n}_{B_{\ell}(G^i_n)}(d\polish_{B_{\ell}(G^i_n)} \,|\, \bar{\polish}) \\
	&= \int_{\Polish^{B_{\ell}(G^i)}} \widehat{f}_i(\polish_{B_r(G^i)}) \, \gamma^{G^i}_{B_{\ell}(G^i)}(d\polish_{B_{\ell}(G^i)} \,|\, \varphi^i_n \bar{\polish}),
\end{align*}
  and thus \eqref{pf:gibbslocal12} implies 
\begin{align}
 \sup_{\bar\polish \in \Polish^{\partial B_{\ell}(G^i_n)}}
\left| \E[ \widehat{f}_i(\varphi^i_n \RV^{G_n}_{B_r(G^i_n)}) \, | \, \RV^{G_n}_{\partial B_{\ell}(G^i_n)} = \bar\polish]
    - \E[\widehat{f}_i(\RV^{G^i}_{B_r(G^i)})]\right| \le \epsilon. \label{pf:gibbslocal22}
\end{align}
By assumption we may choose $n$ large enough so that $d_{G_n}(o^1_n,o^2_n) \geq 2\ell$.
Then, use the fact that $\RV^{G_n}$ is a Markov random field over the graph $G_n$ and the fact that $B_r(G^1_n)$ and $B_r(G_n^2)$ are disjoint, 
to get
\begin{align*}
\E &[\widehat{f}_1(\varphi^1_n \RV^{G_n}_{B_r(G^1_n)})\widehat{f}_2(\varphi^2_n \RV^{G_n}_{B_r(G^2_n)})] \\
	&= \E\Big[\E[ \widehat{f}_1(\varphi^1_n \RV^{G_n}_{B_r(G^1_n)}) \, | \, \RV^{G_n}_{\partial B_{\ell}(G^1_n)}] \, \E[ \widehat{f}_2(\varphi^2_n \RV^{G_n}_{B_r(G^2_n)}) \, | \, \RV^{G_n}_{\partial B_{\ell}(G^2_n)}] \Big].
\end{align*}
Combine this with \eqref{pf:gibbslocal22}, and recall that $|f_i| \le 1$, to obtain
\begin{align*}
\left|\E[\widehat{f}_1(\varphi^1_n \RV^{G_n}_{B_r(G^1_n)})\widehat{f}_2(\varphi^2_n \RV^{G_n}_{B_r(G^2_n)})] - \E[\widehat{f}_1(\RV^{G^1}_{B_r(G^1)})]\E[\widehat{f}_2(\RV^{G^2}_{B_r(G^2)})]\right| \le 2\epsilon,
\end{align*}
for sufficiently large $n$.
Plugging in the definitions of $\widehat{f}_i$ and $\varphi_n^i$, this becomes
\begin{align*}
\left|\E[f_1(B_r(G^1_n,\RV^{G_n}_{G^1_n}))f_2(B_r(G^2_n,\RV^{G_n}_{G^2_n}))] - \E[f_2(B_r(G^1,\RV^{G^1}))]\E[f_2(B_r(G^2,\RV^{G^2}))]\right| \le 2\epsilon,
\end{align*}
for $n$ large. Since $\epsilon$ was arbitrary, this implies \eqref{pf:gibbsdouble1}.
\end{proof}

\begin{proof}[Proof of Proposition \ref{pr:lwprob gibbs}]  
We first prove the ``in law" case.
  Note that the convergence of $G_n \to G$ (resp.\ $(G_n,Y^{G_n}) \to (G,Y^G)$) in distribution in the local weak sense is equivalent to the convergence in law of $\comp_{U^n}(G_n) \to G$ in $\G_*$ (resp.\ $\comp_{U^n}(G_n,Y^{G_n}) \to (G,Y^G)$ in $\G_*[\Polish]$), where $U^n$ is a uniform random vertex in $G_n$.
  The ``in law" case then follows immediately from Proposition \ref{pr:lw gibbs} via continuous mapping or marginalization.

  Next we prove the ``in probability" case.
By Lemma \ref{le:localweakprob characterization}, we know that 
\[
\L(\comp_{U^n_1}(G_n),\comp_{U^n_2}(G_n)) \to \L(G) \times \L(G), \quad \text{in } \P(\G_* \times \G_*),
\]
where $U^n_1,U^n_2$ are independent uniform random vertices in $G_n$. Because $G_n \to G$ in probability in the local weak sense, it is known from \cite[Corollary 2.13]{van2020randomII} that $d_{G_n}(U^n_1,U^n_2) \to \infty$. By passing to a Skorohod representation, we may assume the limits are all almost sure, and then invoke Lemma \ref{le:lwgibbs double} to deduce that 
\[
\Lmc(\comp_{U^n_1}(G_n,\RV^{G_n}),\comp_{U^n_2}(G_n,\RV^{G_n})) \to \Lmc(G,\RV^G) \times \Lmc(G,\RV^G), \quad \text{in } \P(\G_*[\Polish] \times \G_*[\Polish]).
\]
By Lemma \ref{le:localweakprob characterization}, this is equivalent to the claim.
\end{proof}

\begin{corollary} \label{co:assmpB-Gibbs}
  Suppose $G$ is a random element of $\mathcal{U}$. Suppose $\{G_n\}$ is a sequence of finite (possibly disconnected) random graphs.
  Let $H_n \subset G_n$ be random induced subgraphs, and let $A \subset \G_*$ be a Borel set with $\PP(G \in A) > 0$. Suppose $H_n$ converges in probability in the local weak sense to a random element $\widetilde{H}$ of $\G_*$ with $\L(\widetilde{H})=\L(G\,|\,G \in A)$. Then, given random elements $Y^{G_n}$ and $Y^G$ with laws $P_{G_n}$ and $P_{G}$, respectively,
  the sequence of marked random graphs $(H_n,Y^{G_n}_{H_n})$ converges in probability in the local weak sense to the marked random graph with law $\L((G,Y^G)\,|\,G \in A)$.
\end{corollary}
\begin{proof}
By Proposition \ref{pr:lwprob gibbs},  $(H_n,Y^{G_n}_{H_n})$ converges in probability in the local weak sense to $(\widetilde{H},Y^{\widetilde{H}})$.
So we must only argue that $\L(\widetilde{H},Y^{\widetilde{H}}) = \L((G,Y^G) \,|\, G \in A)$.
But this is easy to see:
recalling that $\mathcal{U}$ is the collection of graphs on which the Gibbs measure is unique,  
there exists a map $\Phi : \mathcal{U} \to \G_*[\Polish]$, which is in fact continuous by
Proposition \ref{pr:lw gibbs}, such that $\L(H, Y^H)= \L(\Phi(H))$ for each $H \in \mathcal{U}$. 
 Together with the assumption  $\L(\widetilde{H})=\L(G\,|\,G \in A)$, this implies the desired result:
\begin{align*}
\L(\widetilde{H},Y^{\widetilde{H}}) &= \L(\Phi(\widetilde{H})) = \L(\Phi(G)\,|\,G \in A) = \L((G,Y^G) \,|\, G \in A).
\end{align*}
\end{proof}

\section{Existence and uniqueness for the infinite SDE under Lipschitz assumptions}
\label{ap:uniqueness-infSDE}

\begin{proof}[Proof of Theorem \ref{th:uniqueness-hominfSDE}] 
Let $(\Omega,\F,\FF=(\F_t)_{t \ge 0},\PP)$ be a filtered probability space supporting independent $\FF$-Wiener processes $(W_v)_{v \in V}$ and initial conditions $(\xi_v)_{v \in V}$ that are $\F_0$-measurable and i.i.d.\ with law $\lambda_0$.
Let $(X_v)_{v \in V}$ and $(\widetilde{X}_v)_{v \in V}$ denote two continuous $\FF$-adapted processes, satisfying $\max_{v \in V}\E\|X_v\|^2_{*,T} < \infty$ for each $T > 0$, where we recall that $\|x\|_{*,T} = \sup_{0 \le t \le T}|x(t)|$.
Fix $T < \infty$.
Define $(Y_v)_{v \in V}$ and $(\widetilde{Y}_v)_{v \in V}$ by
\begin{align*}
dY_v(t) &= b(t,X_v,X_{N_v})dt + \sigma(t,X_v)dW_v(t), \quad Y_v(0) = \xi_v, \\
d\widetilde{Y}_v(t) &= b(t,\widetilde{X}_v,\widetilde{X}_{N_v})dt + \sigma(t,\widetilde{X}_v)dW_v(t), \quad \widetilde{X}_v(0) = \xi_v.
\end{align*}
For each $v \in V$ and $t \in [0,T]$,  we may use It\^{o}'s formula and the assumed Lipschitz condition on the drift and diffusion coefficients to get
\begin{align*}
\E \left[\|Y_v - \widetilde{Y}_v\|_{*,t}^2\right] & \le 2t
  \E \left[\int_0^t \left| b(s,X_v,X_{N_v}) - b(s,\widetilde{X}_v,\widetilde{X}_{N_v})\right|^2 ds \right]\\
	& \quad + 8\E \left[\int_0^t \left| \sigma(s,X_v) - \sigma(s,\widetilde{X}_v)\right|^2 ds \right]\\
	& \le 4tK_T^2 \int_0^t \E \left[\left(\|X_v - \widetilde{X}_v\|_{*,s}^2 + \frac{1}{|N_v|}\sum_{u \in N_v} \|X_u - \widetilde{X}_u\|_{*,s}^2 \right) ds \right]\\
  & \quad + 8\bK_T^2 \int_0^t \E \left[ \|X_v - \widetilde{X}_v\|_{*,s}^2\right] \, ds. 
  \end{align*}
Hence,
\begin{align*}
\sup_{v \in V} \E \left[\|Y_v - \widetilde{Y}_v\|_{*,T}^2 \right] & \le 8(tK_T^2+\bK_T^2) \int_0^T \sup_{v \in V} \E \left[ \|X_v - \widetilde{X}_v\|_{*,t}^2 \right] dt.
\end{align*}
Existence and uniqueness now follows from a standard Picard iteration argument that invokes  Gronwall's inequality.
\end{proof}

\section{Verification of Condition \ref{cond-graphs}}
\label{ap:condgraphs}

The goal of this section is to prove Proposition \ref{prop-graphs}, namely
verify that Condition \ref{cond-graphs} is satisfied by the
the {\Erdos } and CM graph sequences described in Theorem
  \ref{thm:CompEmpMeas}.  
This relies on a useful duality property of $\mathrm{UGW}$ trees (as defined in Example \ref{ex:unimodularGW}) which we state first:

\begin{lemma} 
  \label{lem-duality}
  Given a probability distribution $\rho$ on $\N_0$ with finite and nonzero first and second
  moments,    let $\parm_\rho$ be as defined in \eqref{def-parm},
  and let $m_\rho := \sum_{i \ge 1} i \rho_i$.  Then  let $\tree_\rho := \mathrm{UGW} (\rho)$ and define 
$s_\rho :=  s_{\tree_\rho} = \PP(|\tree_\rho| = \infty).$  
  If $\parm_\rho > 1$, then  $s_\rho  > 0$ 
  and $\Lmc( \Tmc_\rho \, | \, |\Tmc_\rho| <\infty) = \Lmc(\widetilde{\Tmc}_\rho)$,
  where $\widetilde{\Tmc}_\rho = \mathrm{UGW} (\widetilde{\rho})$, with 
  \[  \widetilde{\rho}_k  := \frac{\rho_k}{1 - s_{\rho}} \left( 1 - \frac{2 \alpha_\rho}{m_\rho} \right)^{k/2}, \quad k \in \N_0, \quad
  \tilde{\parm}_\rho :=  \frac{\sum_{k \in \N_0} k(k-1) \tilde{\rho}_k}{\sum_{k \in \N_0} k \tilde{\rho}_k}
\le 1, 
  \]
with $\alpha_\rho \in [0,m_\rho/2]$   equal to the smallest positive solution to the equation $H_\rho(x) = 0$, where 
 \[ H_\rho(x) := m_\rho  - 2 x -  \sum_{k \in \N_0}  k \rho_k \left( 1 - \frac{2 x}{m_\rho} \right)^{k/2}, \quad x \in  [0, m_\rho/2].  \]
 Furthermore, if $\rho = {\rm Poisson} (\parm)$, then $\widetilde{\rho} = {\rm Poisson} (\tilde{\parm})$, where
 $\tilde{\parm} < 1$ satisfies $\tilde{\parm} e^{-\tilde{\parm}} = \theta e^{-\theta}$. 
\end{lemma}
\begin{proof}
In the case when $\rho = \mathrm{Poisson} (\parm )$ with $\parm > 1$, and $\tree_\parm = \tree_\rho$,
    the fact that $\Lmc (\tree_\parm||\tree_\parm| < \infty) = \Lmc(\tree_{\tilde{\parm}})$, where
    $\tilde{\parm}$ solves the stated equation   is a consequence of
    the Poisson duality principle enunciated in \cite[Theorem 3.15]{van2016random}.
    For general $\rho$,  dropping the explicit dependence on $\rho$ from all quantities, we 
    start by noting that the existence of 
$\alpha$  in the statement of the lemma follows from 
    the fact that the continuously differentiable
    function $H$ satisfies  
    $H(0)=0$, $H'(0)>0$ and  $H(\frac{m}{2})=0$, where  we used $\parm > 1$ to conclude these
    properties.
    For the subsequent analysis, we also note that  the above properties also show that $H^\prime (\alpha) \le 0$.

        We now show that  $\Lmc(\widetilde{\Tmc}) = \Lmc( \Tmc \, | \, |\Tmc| <\infty)$.
        This  uses an argument  similar to that used to prove the duality principle for
        (not necessarily unimodular) branching processes
       \cite[Theorem 3.7]{van2016random},  and so we provide just a sketch of the proof.
 Let  $\RV_i, i \in \N,$ and $\tilde{\RV}_i, i \in \N,$ be iid random variables distributed according to
        $\rho$ and $\rhotil$, respectively.
        Here,  $\RV_i$ (resp.\ $\tilde{\RV}_i$) represents  the number of children of the $i$-th vertex in $\Tmc$ (resp.\
        $\widetilde{\Tmc}$) when
        the tree is  explored starting with the root vertex labeled $1$, and increasing labels
        for  vertices in each succesive generation 
        (with an arbitrary  ordering of labels for vertices within each generation).
Let $H=(\RV_1,\dotsc,\RV_{|\tree|})$ be the vector of the offspring number of each vertex in $\Tmc$ and 
 let $\Htil=(\tilde{\RV}_1,\dotsc,\tilde{\RV}_{|\Tmc|})$ be the corresponding quantity in $\widetilde{\Tmc}$.  
	Write $\beta := \sqrt{1-\frac{2\alpha}{m}}$. 
	From the definition of $\alpha$ we have $\sum_{k \in \Nmb_0} k\rho_k \beta^k = \beta^2 \sum_{k \in \Nmb_0} k\rho_k$.
	Using this and the definitions of $m$, $\rhotil$ as stated in the lemma,
         and the  size-biased versions 
        $\rhohat$ and $\hat{\rhotil}$ of $\rho$ and $\rhotil$, as defined in \eqref{def-widehatrho},
        one can see that $\hat{\rhotil}_k / \rhohat_k = \beta^{k-1}$.   
	From this it follows that for each $t \in \N$ and $\polish \in \N_0^t$, 
	\begin{align*}
      & \Pmb(H=(\polish_1,\dotsc,\polish_t)\,|\, |\Tmc|<\infty) = \frac{\Pmb(H=(\polish_1,\dotsc,\polish_t))}{\Pmb(|\Tmc|<\infty)} = \frac{1}{1-s_\rho} \rho_{\polish_1} \prod_{i=2}^t \rhohat_{\polish_i} \\
	  & \quad = \left( \rhotil_{\polish_1} \prod_{i=2}^t \hat{\rhotil}_{\polish_i} \right) \left( \beta^{-\polish_1} \prod_{i=2}^t \beta^{1-\polish_i}\right) 
	  = \rhotil_{\polish_1} \prod_{i=2}^t \hat{\rhotil}_{\polish_i} 
      = \Pmb(\Htil=(\polish_1,\dotsc,\polish_t)),
	\end{align*}
	where the third equality uses the definition of $\rhotil_{\polish_1}$, and the fourth equality uses the observation that $1+\polish_1+\dotsb+\polish_t=t$ is the total number of vertices.
    This completes the proof.
\end{proof}

\begin{proof}[Proof of Proposition \ref{prop-graphs}]
  {Recall the description of $G_n$, $n \in \Nmb$, and $G$ in Theorem \ref{thm:CompEmpMeas}. 
	Recall $s_G := \PP (|G| =  \infty)$.
	In case (ii), also recall that $\theta$ was defined in \eqref{def-parm}. 

	We first verify Condition \ref{cond-graphs}(1) for $s_G > 0$.
	Since $s_G>0$, from \cite[Theorems 2.1.2 and 2.1.3]{durrett2007random} we must have $\theta > 1$.  
	Then it is known that  there exists $\beta > 0$ such that
	with probability approaching $1$ there is exactly one connected component of $G_n$ with
	more than $\beta\log n$ vertices, and the size of this component divided by $s_{\parm}n$
	approaches $1$ in probability.   Indeed, this follows from \cite[Theorem 2.3.2]{durrett2007random}
	and Remark \ref{rem-npn} below for case (i), and  \cite[Theorem 1]{MolloyReed1998size} for case (ii).
	Therefore, in both cases, Condition \ref{cond-graphs}(1) holds.

	To verify  Condition \ref{cond-graphs}(2) when $0<s_G<1$, we first consider the \Erdos \ graph in case (i).	
	It follows from  \cite[Theorem 4.15]{van2016random} that the law of $G_n \setminus \compmax(G_n)$ is close to that of
        $\widetilde G_n$ as $n\to\infty$, where $\widetilde{G}_n$ is the
	\Erdos \ graph $\Gmc(\lfloor n s_G\rfloor,\parm'/\lfloor n s_G\rfloor)$, where $\parm' < 1$ satisfies $\parm' e^{-\parm'} = \parm e^{-\parm}$.
From Example \ref{ex:GW}, we know that $\widetilde{G}_n$ converges in probability in the local weak sense to $\widetilde{\tree}$, the Galton-Watson tree with Poisson($\theta'$) offspring distribution. 
	Now, it follows from the Poisson duality principle enunciated in \cite[Theorem 3.15]{van2016random}  
	that $\Lmc(\widetilde{\Tmc}) = \Lmc( \Tmc \, | \, |\Tmc| <\infty)$, where $\Tmc$ is the Galton-Watson tree with Poisson($\theta$) offspring distribution. 
	This establishes Condition \ref{cond-graphs}(2) for case (i).

	Now, consider the CM model of  case (ii), let $\rhotil$ and $\tilde{\parm} \le 1$ be  as defined in Lemma \ref{lem-duality},
	and for each $n \in \N$,  let $m_n$, $\alpha_n \in [0,m_n/2]$ be defined
	as in Lemma \ref{lem-duality}, but with $\rho$ replaced with $d_{\cdot}(n)/n$, the empirical degree distribution of 
	$G_n$, and let $\tilde{d}_k(n)=\rho_k \left( 1 -  2\alpha_n/m_n \right)^{k/2} n$.  
	Our assumptions on the graphic sequence $d(n)$, most notably that the second moment of the degree sequence converges to the finite second moment of $\rho$, ensure that the degree sequence is \emph{well-behaved} in the sense of \cite{MolloyReed1998size}.
	It follows from the duality principle of \cite[Theorem 2]{MolloyReed1998size} that the law of
	$G_n \setminus \compmax(G_n)$ is close to that of $\widetilde G_n$ as $n\to\infty$, where
	$\widetilde{G}_n \sim \CM((1- s_G)n,\tilde{d}(n))$. 
From Example \ref{ex:unimodularGW}, we know that $\widetilde{G}_n$ converges in probability in the local weak sense to
	$\widetilde{\tree} \sim$ UGW($\rhotil$). Moreover, $\tilde{\parm} \le 1$ implies  
	$\PP ( |\widetilde{\Tmc}| < \infty) = 1$. 
	The verification of Condition \ref{cond-graphs}(2) for case (ii) is completed by
	observing that 
	from Lemma \ref{lem-duality},   $\Lmc(\widetilde{\Tmc}) = \Lmc( \Tmc \, | \, |\Tmc| <\infty)$, where $\Tmc \sim$ UGW($\rho$).}
\end{proof}

\begin{remark} \label{rem-MolloyReed}
In the proof of Proposition \ref{prop-graphs}, it is clear that the technical assumptions on the graphic sequence $d(n)$ in case (ii) are essentially only needed to apply the results of Molloy-Reed \cite{MolloyReed1998size}. Proposition \ref{prop-graphs} and Theorem \ref{thm:CompEmpMeas} remain valid as long as $d(n)$ satisfies the somewhat more general (but longer to state) assumptions of \cite[Theorems 1 and 2]{MolloyReed1998size}.
The other technical assumption $\rho_2<1$ in case (ii) is only used to conclude from $s_G>0$ and \cite[Theorems 2.1.2 and 2.1.3]{durrett2007random} that $\theta>1$. In general, the sequence of the configuration model (and the empirical measure on it) when $\rho_2=1$ could behave quite differently, even if the limit random graph $G$ is just an infinite $2$-regular tree.
This is because, as illustrated in \cite[Page 130 and Exercises 4.3--4.5]{van2020randomII}, the size of the maximal component could be either $O(n)$ or $o(n)$, depending the way the graphic sequence converges to $\rho$.
\end{remark}

\begin{remark}  \label{rem-npn}
  Strictly speaking, in the above proof of Proposition \ref{prop-graphs}, the results cited from
  \cite[Theorem 2.3.2]{durrett2007random} and \cite[Theorem 4.15]{van2016random} are established only for $p_n = \parm/n$, and not under the general condition $np_n \rightarrow \parm$.  
	However for \cite[Theorem 2.3.2]{durrett2007random}, the latter case can be deduced from the former by using stochastic ordering.
	Specifically, with $L_1(p_n)$ denoting the size of the largest component in ${\mathcal G}(n,p_n)$, a simple coupling argument can be used to show that for any $\varepsilon > 0$, for all sufficiently large $n$, $L_1((\parm-\varepsilon)/n) \leq L_1(p_n)   \leq L_1((\parm +\varepsilon)/n)$.
On the other hand, for \cite[Theorem 4.15]{van2016random}, the proof therein works in fact under the general condition $np_n \rightarrow \parm$. 
\end{remark}

\textbf{Acknowledgments}.
We thank an anonymous referee for suggesting the more direct and general argument presented  in the proof of Proposition \ref{pr:ergopropagation}. 
In the case of an \emph{amenable} group $G$, the mean ergodic theorem states that the ergodicity of $x$ is equivalent to convergence of the empirical fields $\frac{1}{|A_n|}\sum_{g \in A_n} \delta_{gx} \to \L(x)$ for any F\o lner sequence $(A_n)$
(see \cite[Chapter 8]{einsiedler2013ergodic} for definitions).  
In a previous version of the paper, we used this along with Proposition \ref{pr:empfield} to prove Proposition \ref{pr:ergopropagation} in the amenable case.

\bibliographystyle{amsplain}


\end{document}